\documentclass[a4paper]{article}
\usepackage[T1]{fontenc}
\usepackage{lmodern}
\usepackage{amsmath,amssymb}
\usepackage{amsthm}
\usepackage{mathrsfs}
\usepackage[all]{xy}
\usepackage{rotate}
\usepackage{enumitem}
\usepackage{mathtools}
\usepackage{graphics}
\usepackage{tikz}
\usepackage{dynkin-diagrams}

\usepackage[setpagesize=false]{hyperref}

\usepackage{here}

\theoremstyle{plain}
\newtheorem{theorem}{Theorem}[section]
\newtheorem*{theorem*}{Theorem}
\newtheorem{proposition}[theorem]{Proposition}
\newtheorem{lemma}[theorem]{Lemma}
\newtheorem{corollary}[theorem]{Corollary}

\newtheorem{fact}[theorem]{Fact}

\theoremstyle{definition}
\newtheorem{definition}[theorem]{Definition}
\newtheorem{hypothesis}[theorem]{Hypothesis}
\newtheorem{example}[theorem]{Example}

\newtheorem{remark}[theorem]{Remark}

\newtheorem{algorithm}[theorem]{Algorithm}

\numberwithin{equation}{subsection}

\newenvironment{enumparen}{\begin{enumerate}[label=\upshape(\arabic*),ref=(\arabic*)]}{\end{enumerate}}

\newcommand{\define}[1]{\textit{#1}}

\newcommand{\spn}[2]{\mathrm{span}_{#1}\left\{#2\right\}}
\newcommand{\set}[1]{\left\{#1 \right\}}

\newcommand{\Ker}{\mathrm{Ker}}

\newcommand{\Hom}{\mathrm{Hom}}
\DeclareMathOperator{\End}{End}
\DeclareMathOperator{\Aut}{Aut}
\DeclareMathOperator{\id}{id}

\newcommand{\tr}{\mathrm{tr}}

\newcommand{\BasicRing}[1]{\mathbb{#1}}
\newcommand{\ZZ}{\BasicRing{Z}}

\newcommand{\RR}{\BasicRing{R}}
\newcommand{\CC}{\BasicRing{C}}
\newcommand{\HH}{\BasicRing{H}}
\newcommand{\KK}{\BasicRing{K}}

\newcommand{\Ann}{\mathrm{Ann}}

\DeclareMathOperator{\gr}{gr}

\DeclareMathOperator{\rank}{rank}
\newcommand{\Ad}{\mathrm{Ad}}
\newcommand{\ad}{\mathrm{ad}}
\newcommand{\classicalG}[1]{\mathrm{#1}}
\newcommand{\classicalg}[1]{\mathfrak{#1}}

\newcommand{\lieC}[1]{\mathfrak{#1}_{\CC}}
\newcommand{\lie}[1]{\mathfrak{#1}}
\newcommand{\lieCent}{\mathfrak{c}}
\newcommand{\lieNorm}{\mathfrak{n}}

\newcommand{\univ}[1]{\mathcal{U}(\lieC{#1})}
\newcommand{\univfilt}[2]{\mathcal{U}_{#1}(\lieC{#2})}
\newcommand{\univcent}[1]{\mathcal{Z}(\lieC{#1})}

\newcommand{\AV}{\mathcal{AV}}
\newcommand{\WF}{\mathrm{WF}}

\newcommand{\orbit}[1]{\mathcal{#1}}

\newcommand{\regular}{\mathcal{O}}
\newcommand{\rring}[1]{\regular(#1)}

\let\Im\relax

\DeclareMathOperator{\Im}{\mathrm{Im}}

\newcommand{\lieSL}{\classicalg{sl}}

\newcommand{\lieGL}{\classicalg{gl}}
\newcommand{\LieGL}{\classicalG{GL}}
\newcommand{\lieSU}{\classicalg{su}}

\newcommand{\lieU}{\classicalg{u}}
\newcommand{\LieU}{\classicalG{U}}
\newcommand{\lieSp}{\classicalg{sp}}

\newcommand{\lieSO}{\classicalg{so}}
\newcommand{\LieSO}{\classicalG{SO}}

\newcommand{\lieG}{\classicalg{g}}
\newcommand{\lieE}{\classicalg{e}}
\newcommand{\lieF}{\classicalg{f}}

\newcommand{\Hilbert}[1]{\mathcal{#1}}

\newcommand{\assertion}{\mathrm{P}}
\newcommand{\proj}{\mathrm{pr}}
\newcommand{\Nilpotent}{\mathcal{N}}
\newcommand{\jprod}[1]{\{#1\}}
\newcommand{\KS}{\mathrm{KS}}

\newcommand{\nc}{\mathrm{n}}

\begin{document}

\title{Criteria for discrete decomposability and admissibility in the branching problem}

\author{Masatoshi Kitagawa}

\date{}

\maketitle

\begin{abstract}
	In this paper, we study when the restriction of an irreducible representation of a reductive Lie group $G$ to a reductive subgroup $H$ is discretely decomposable.
	T.\ Kobayashi gave a necessary condition for discrete decomposability in terms of associated varieties.
	In a previous paper, we proved that the condition is also sufficient if $(G, H)$ is a symmetric pair.
	The purpose of this paper is to prove the sufficiency for any reductive subgroup $H$ under some hypothesis on $G$ (e.g., $G$ is linear).
	As an application, we show that discrete decomposability is equivalent to $\mathfrak{h}$-admissibility when the compact factor of $\mathfrak{h}$ is maximal in some sense.
\end{abstract}

\section{Introduction}

When an irreducible module of a reductive Lie algebra is restricted to a reductive subalgebra, it is no longer irreducible in general, and its structure can be quite complicated.
For example, even a finitely generated submodule of the restriction need not have finite length in general.
T.\ Kobayashi \cite{Ko94, Ko98_discrete_decomposable_2, Ko98_discrete_decomposable_3} introduced the notion of discrete decomposability as a condition under which branching laws become more tractable.
He also proposed several conjectures in \cite{Ko00_discretely_decomposable,Ko11,Ko15_vogan}.
Some of these conjectures have been resolved in \cite{ZhLi10}, \cite{Ko19_admissible} and \cite{Ki24}.
The main purpose of this paper is to solve one of the remaining conjectures.

Although restrictions of modules are rarely discretely decomposable, discretely decomposable restrictions often arise for important representations such as the Weil representation, minimal representations, and highest weight modules.
In recent work, the discrete branching laws for $A_{\lie{q}}(\lambda)$ upon restriction to symmetric subgroups have been completely determined \cite{Os25}.
For further examples, we refer the reader to the references therein.
We shall not discuss explicit branching laws for individual representations any further in this paper.

This paper continues the work in \cite{Ki24} by giving a necessary and sufficient condition for discrete decomposability in terms of associated varieties and wave front sets.
As a consequence, we will prove Kobayashi's conjecture \cite[Conjecture B]{Ko00_discretely_decomposable} on discrete decomposability for any linear reductive Lie group.
For non-linear groups, see Hypothesis \ref{hyp:SchmidVilonen} and remarks following it.
In \cite{Ki24}, we proved the conjecture for symmetric pairs $(G,H)$.
Y.\ Oshima also proved the conjecture in the same setting in an unpublished note.

We also discuss the relationship between discrete decomposability and admissibility in the branching problem of reductive Lie groups.
We will show that discrete decomposability is equivalent to admissibility modulo compact factors in some sense.

We shall recall a necessary condition and a sufficient condition for discrete decomposability, and state our main result.
Let $G$ be a connected reductive Lie group with a Cartan involution $\theta$ and $H$ a $\theta$-stable connected reductive subgroup of $G$.
Set $K\coloneq G^\theta$ and $K_H\coloneq H\cap K$.
Let $\lie{g}$ and $\lie{h}$ be the Lie algebras of $G$ and $H$, respectively.

Fix a $G$-invariant non-degenerate symmetric bilinear form $(\cdot, \cdot)$ on $\lie{g}$ such that $-(\cdot, \theta(\cdot))$ is an inner product on $\lie{g}$.
We identify $\lie{g}$ with its dual space $\lie{g}^*$ via the bilinear form $(\cdot, \cdot)$.

Let $V$ be an irreducible $(\lie{g}, K)$-module.
We say that $V|_{\lie{h}, K_H}$ is discretely decomposable if there exists a $(\lie{h}, K_H)$-module filtration $V = \bigcup_{i=0}^\infty V_i$ such that each $V_i$ has finite length.
See \cite[Definition 1.1]{Ko98_discrete_decomposable_3}.
We remark that, for unitarizable $V$, the restriction $V|_{\lie{h}, K_H}$ is discretely decomposable if and only if it is completely reducible.

To describe the necessary and sufficient condition for discrete decomposability, we need the notions of associated varieties and wave front sets.
Let $\AV(V) \subset \lieC{g}$ be the associated variety of the $(\lie{g}, K)$-module $V$.
We denote by $\proj_{\lieC{h}}$ the orthogonal projection from $\lieC{g}$ to $\lieC{h}$.
For a finite-dimensional real representation $V$ of a Lie group $L$, we set
\begin{align*}
	\Nilpotent(V, L) \coloneq \set{v \in V : 0 \in \overline{L\cdot v}}.
\end{align*}
Then it is well known that $\AV(V)$ is a closed $K_{\CC}$-invariant cone in $\lieC{k}^\perp$, where $K_\CC$ is the complexification of $K$.
We use analogous notation $K_{H, \CC}$ for the complexification of $K_H$.
Kobayashi's necessary condition \cite[Corollary 3.4]{Ko98_discrete_decomposable_3} for discrete decomposability is as follows.

\begin{fact} \label{fact:IntroKobayashiNecessaryCondition}
	Let $V$ be an irreducible $(\lie{g}, K)$-module.
	If $V|_{\lie{h}, K_H}$ is discretely decomposable, then one has $\proj_{\lieC{h}}(\AV(V)) \subset \Nilpotent((\lie{k}_{H, \CC})^\perp, K_{H, \CC})$.
\end{fact}

The wave front set $\WF(V)$ of $V$ is a closed $G$-stable cone in $\lie{g}$, which is an analytic counterpart of the associated variety $\AV(V)$.
By Schmid--Vilonen \cite{ScVi00}, the two invariants contain exactly the same information, at least when $G$ is linear.
Let $\KS$ be the Kostant--Sekiguchi correspondence of nilpotent orbits (Fact \ref{fact:KostantSekiguchi}), that is,
\begin{align*}
	\KS \colon \Nilpotent(\lie{g}, G)/G \xrightarrow{\simeq} \Nilpotent(\lieC{k}^\perp, K_\CC)/K_\CC.
\end{align*}
We extend $\KS$ to a map from the set of all $G$-stable subsets of $\Nilpotent(\lie{g}, G)$ to the set of all $K_\CC$-stable subsets of $\Nilpotent(\lieC{k}^\perp, K_\CC)$.

\begin{hypothesis} \label{hyp:SchmidVilonen}
	For any irreducible $(\lie{g}, K)$-module $V$, one has $\KS(\WF(V)) = \AV(V)$.
\end{hypothesis}

By the Schmid--Vilonen theorem \cite{ScVi00} (Fact \ref{fact:SchmidVilonen}), Hypothesis \ref{hyp:SchmidVilonen} holds for any linear reductive Lie group $G$.
We expect Hypothesis \ref{hyp:SchmidVilonen} to hold for a broader class of reductive Lie groups, not necessarily linear.
However, we state it as a hypothesis since we have been unable to find a reference for such a generalization.
Throughout this paper, every connected reductive Lie group under consideration is assumed to satisfy Hypothesis \ref{hyp:SchmidVilonen}.

In light of the Schmid--Vilonen theorem, the condition in Fact \ref{fact:IntroKobayashiNecessaryCondition} should be equivalent to some condition on wave front sets.
We have obtained the following result in \cite[Theorem 17]{Ki24}.

\begin{fact} \label{fact:IntroSufficientCondition}
	Let $V$ be an irreducible $(\lie{g}, K)$-module.
	If $\WF(V)\cap \Nilpotent(\lie{h}^\perp, H) = \set{0}$, then $V|_{\lie{h}, K_H}$ is discretely decomposable.
\end{fact}

We remark that the definition of $\Nilpotent(\lie{h}^\perp, H)$ in this paper is slightly different from that in \cite{Ki24}.
We bridge the gap in Proposition \ref{prop:NullCone}.
Note that, to show Fact \ref{fact:IntroSufficientCondition}, we do not need Hypothesis \ref{hyp:SchmidVilonen}.

Notably, in Facts \ref{fact:IntroKobayashiNecessaryCondition} and \ref{fact:IntroSufficientCondition}, the necessary condition is easier to prove in terms of $\AV(V)$, whereas the sufficient condition is easier to prove in terms of $\WF(V)$.
Consequently, it remains only to establish the required implication between the condition involving $\AV(V)$ and that involving $\WF(V)$.
This reduces the problem to a purely Lie-algebraic one and allows us to avoid difficult representation-theoretic arguments.
We will prove the following theorem, which is the main result of this paper.

\begin{theorem} \label{thm:Introduction}
	Let $V$ be an irreducible $(\lie{g}, K)$-module.
	Then the following are equivalent.
	\begin{enumparen}
		\item $\WF(V)\cap \Nilpotent(\lie{h}^\perp, H) = \set{0}$.
		\item $V|_{\lie{h}, K_H}$ is discretely decomposable.
		\item $\proj_{\lieC{h}}(\AV(V)) \subset \Nilpotent((\lie{k}_{H, \CC})^\perp, K_{H, \CC})$.
	\end{enumparen}
\end{theorem}

We have seen in Facts \ref{fact:IntroKobayashiNecessaryCondition} and \ref{fact:IntroSufficientCondition} that the implications (1) $\Rightarrow$ (2) $\Rightarrow$ (3) have been proved.
The remaining part of the proof of Theorem \ref{thm:Introduction} follows from the implication
\begin{align}
	\overline{\orbit{O}} \cap \Nilpotent(\lie{h}^\perp, H) \neq \set{0} \Rightarrow \proj_{\lieC{h}}(\KS(\overline{\orbit{O}})) \not\subset \Nilpotent((\lie{k}_{H, \CC})^\perp, K_{H, \CC}) \label{eqn:IntroTargetImplication}
\end{align}
for a nilpotent orbit $\orbit{O}$ in $\lie{g}$.
We have already shown \eqref{eqn:IntroTargetImplication} in \cite[Theorem 21]{Ki24} under certain assumptions; for example, this holds when $(\lie{g}, \lie{h})$ is a symmetric pair.

It is difficult to describe how nilpotent orbits, $\KS$ and $\lie{h}$ interact in general.
In this paper, we reduce the problem to concrete pairs $(\lie{g},\lie{h})$ by applying the following reductions:
\begin{enumerate}[label=(R\arabic*)]
	\item remove compact and abelian factors of $\lie{g}$;
	\item remove compact factors of $\lie{h}$;
	\item for a nonzero $\lie{a}_H$-weight vector $X\in \overline{\orbit{O}} \cap \Nilpotent(\lie{h}^\perp, H)$ with nonzero weight, let $\lie{g}'$ be the smallest Lie subalgebra containing $X$, $\theta(X)$, and $\lie{a}_H$ and satisfying $\lie{g}'=(\lie{g}'\cap\lie{h})\oplus(\lie{g}'\cap \lie{h}^\perp)$, and replace $(\lie{g}, \lie{h}, \orbit{O})$ with $(\lie{g}', \lie{g}'\cap\lie{h}, \Ad(G')X)$.
\end{enumerate}
Here $\lie{a}_H$ is a maximal abelian subspace of $\lie{h}^{-\theta}$.
The conclusion is simple: iterating the above reductions, we can reduce the problem to the case where $\lie{g} = \lieSL(2,\RR)$ (Theorem \ref{thm:MainTargetMinimal}).

By reduction (R3), we may assume that $\lie{h}$ has real rank one.
Moreover, except for the cases in which \eqref{eqn:IntroTargetImplication} has already been established, it suffices to consider the cases where $\lie{g}$ admits a $5$-grading or a $3$-grading whose characteristic element lies in $\lie{h}$.

The hardest part of the proof is to find a subalgebra $\lie{g}'$ as in the third reduction.
For $\lie{h} = \lieSO(1, k)$ ($k \geq 2$), we use the structure theory of Jordan triple systems to find the subalgebra $\lie{g}'$ (Theorems \ref{thm:NonMinimalTypeASimple} and \ref{thm:NonMinimalTypeASemisimple}).
For $\lie{h} = \lieSp(1, k)$ ($k \geq 2$) and $\lie{h} = \lieF_{4(-20)}$, the conclusion follows from general considerations, since the adjoint action of $\lie{h}$ on $\lie{g}$ is subject to strong restrictions.
For $\lie{h} = \lieSU(1, k)$ ($k \geq 2$), we classify all possible $\lie{g}$ and check the assumptions of some general reduction propositions.
The classification for exceptional $\lie{g}$ is carried out using a straightforward algorithm implemented on a computer \cite{Ki26-code}.

As an application of Theorem \ref{thm:Introduction}, we discuss the relationship between discrete decomposability and admissibility using only wave front sets.
We say that an irreducible $(\lie{g}, K)$-module $V$ is $\lie{h}$-admissible if $V|_{\lie{h}, K_H}$ is discretely decomposable and the multiplicity of each irreducible $(\lie{h}, K_H)$-module in $V|_{\lie{h}, K_H}$ is finite.
We refer the reader to \cite{Os25} for details of the multiplicities in $V|_{\lie{h}, K_H}$.
See also \cite[Definition 16 and Proposition 16]{Ki24} and Definition \ref{def:Admissible}.
The following criterion for $\lie{h}$-admissibility was proved in \cite{Ko98_discrete_decomposable_2, Ko19_admissible} for compact $H$, and in \cite{Ki24} for general $H$.

\begin{fact} \label{fact:IntroAdmissibleCriterion}
	Let $V$ be an irreducible $(\lie{g}, K)$-module.
	Then $V$ is $\lie{h}$-admissible if and only if $\WF(V)\cap \lie{h}^\perp = \set{0}$.
\end{fact}

Note that the original form of Kobayashi's criterion \cite{Ko98_discrete_decomposable_2} is stated in terms of asymptotic $K$-support.
See \cite[Corollary 9]{Ki24} for the relation between asymptotic $K$-support and wave front sets.
We remark that Fact \ref{fact:IntroSufficientCondition} is regarded as an analogue of Fact \ref{fact:IntroAdmissibleCriterion} for discrete decomposability.

Let $\lie{h}_{\nc}$ be the ideal of $\lie{h}$ generated by $\lie{h}^{-\theta}$, i.e., the non-compact part of $\lie{h}$.
Set
\begin{align*}
	\lie{h}' \coloneq \lieNorm_{\lie{g}}(\lie{h}_{\nc}), \quad \lie{k}_{H'} \coloneq \lie{h}' \cap \lie{k}, \quad \lie{h}'' \coloneq \lie{h} + \lie{k}_{H'}.
\end{align*}
Let $K_{H'}$ be the analytic subgroup of $K$ with the Lie algebra $\lie{k}_{H'}$.
Then the Lie algebra $\lie{h}''$ is obtained from $\lie{h}$ by enlarging its compact factor as much as possible.
We will prove the following result in Theorem \ref{thm:DiscretelyDecomposableAndAdmissibility}.

\begin{theorem} \label{thm:IntroDiscretelyDecomposableAndAdmissibility}
	Let $V$ be an irreducible $(\lie{g}, K)$-module.
	Then the following are equivalent.
	\begin{enumparen}
		\item $V|_{\lie{h}, K_H}$ is discretely decomposable.
		\item $V$ is $\lie{h}''$-admissible.
		\item $V$ is $K_{H'}$-admissible.
		\item $V$ is $\lie{h}'$-admissible.
	\end{enumparen}
\end{theorem}

By Theorem \ref{thm:IntroDiscretelyDecomposableAndAdmissibility}, whether a discretely decomposable restriction $V|_{\lie{h}, K_H}$ is $\lie{h}$-admissible depends only on the size of the compact factor, and $V$ is always $\lie{h}$-admissible when the compact factor is maximal.

This paper is organized as follows.
In Section \ref{section:Criteria}, we review $(\lie{g}, K)$-modules and discrete decomposability, as well as associated varieties and wave front sets.
In Section \ref{section:Reduction}, we give a procedure to reduce the proof of Theorem \ref{thm:Introduction} to special pairs $(\lie{g}, \lie{h})$.
The main theorem is proved in Sections \ref{section:TypeA1} and \ref{section:TypeBC1}.
We deal with the case where $\lie{h} \simeq \lieSO(1, k)$ in Section \ref{section:TypeA1}, and the cases $\lie{h}\simeq\lieSU(1,k)$, $\lie{h}\simeq\lieSp(1,k)$ and $\lie{h}\simeq\lieF_{4(-20)}$ in Section \ref{section:TypeBC1}.
The proof of Theorem \ref{thm:IntroDiscretelyDecomposableAndAdmissibility} is given in Section \ref{section:Admissibility}.
In that section, we also give an alternative short proof of Fact \ref{fact:IntroSufficientCondition}.

\subsection*{Notation and conventions}

In this paper, any Lie algebra is finite dimensional.
We denote real Lie groups and their Lie algebras by Roman letters and the corresponding German letters, respectively; we use the same convention for complex Lie groups and their Lie algebras, with the subscript $(\cdot)_\CC$.
Similarly, we denote the complexification of a real Lie algebra by the same German letter as the real form, with the subscript $(\cdot)_\CC$.
For example, the Lie algebras of real Lie groups $G$, $K$ and $H$ are denoted by $\lie{g}$, $\lie{k}$ and $\lie{h}$, respectively, and their complexifications by $\lie{g}_\CC$, $\lie{k}_\CC$ and $\lie{h}_\CC$, respectively.

For a Lie group, we denote its identity component by the same symbol with subscript $o$.
For example, the identity component of a Lie group $G$ is denoted by $G_o$.

Let $\lie{g}$ be a Lie algebra.
For a subset $S\subset \lie{g}$, we denote by $\lieNorm_{\lie{g}}(S)$ the normalizer of $S$ in $\lie{g}$, by $\lieCent_{\lie{g}}(S)$ the centralizer of $S$ in $\lie{g}$, and by $\lieCent(\lie{g})$ the center of $\lie{g}$.
For a $\lie{g}$-module $V$, $h \in \lie{g}$ and $\lambda \in \CC$, we denote by $V_\lambda$ the eigenspace of $h$ in $V$ with the eigenvalue $\lambda$.
In this paper, $h$ is used to denote a characteristic element of a grading of $\lie{g}$.

For a $\lie{t}$-module $V$ of a commutative Lie algebra $\lie{t}$, we denote by $\Delta(V, \lie{t})$ the set of all non-zero weights in $V$, and by $V_\lambda$ the weight space of $V$ corresponding to the weight $\lambda \in \lie{t}^*$.
For a $G$-set $X$ of a group $G$, we write $X^G$ for the set of all $G$-invariant elements in $X$.
For automorphisms $\theta$ and $\sigma$ of a Lie algebra $\lie{g}$, we denote by $\lie{g}^\theta$ the set of all elements in $\lie{g}$ fixed by $\theta$ and set $\lie{g}^{\theta, \sigma} \coloneq \lie{g}^\theta \cap \lie{g}^\sigma$.

\subsection*{Acknowledgement}

This work was supported by the Japan Society for the Promotion of Science, Grant-in-Aid for Transformative Research Areas (A) (22A201, 22H05107).

\section{Criteria for discrete decomposability} \label{section:Criteria}

In this section, we recall the notions of $(\lie{g}, K)$-modules and discrete decomposability of $(\lie{g}, K)$-modules.
We also state the main theorem of this paper, which asserts that the necessary condition given by T.\ Kobayashi for discrete decomposability is sufficient without any assumptions on the subgroup $H$ of $G$.
To do this, we review known results on the null cone, as well as known results on the associated varieties and wave front sets of $(\lie{g}, K)$-modules.

\subsection{\texorpdfstring{$(\lie{g}, K)$}{(g,K)}-modules}

We shall recall the definition of $(\lie{g}, K)$-modules and the notion of discrete decomposability.
We refer the reader to \cite{Wa88_real_reductive_I} for the theory of $(\lie{g}, K)$-modules and to \cite{Ko98_discrete_decomposable_3} for the discrete decomposability of $(\lie{g}, K)$-modules.

Let $G$ be a connected real reductive Lie group with a Cartan involution $\theta$.
Set $K\coloneq G^\theta$, which is the fixed-point subgroup of $\theta$.
Then $K$ is connected.

\begin{definition}
	A $\lie{g}$-module $V$ (over $\CC$) is called a $(\lie{g}, K)$-module if $V$ is locally $\lie{k}$-finite and any finite-dimensional $\lie{k}$-submodule of $V$ lifts to a continuous representation of $K$.
\end{definition}

Since $K$ is connected, the $K$-action on the $\lie{g}$-module $V$ is unique.
Since any finite-dimensional (continuous) representation of $K$ is completely reducible, any $(\lie{g}, K)$-module is completely reducible as a representation of $K$.

\begin{definition}[T.\ Kobayashi {\cite[Definition 1.1]{Ko98_discrete_decomposable_3}}] \label{def:DiscretelyDecomposable}
	Let $V$ be a $(\lie{g}, K)$-module.
	We say that $V$ is discretely decomposable if there exists an increasing filtration $\set{V_n}_{n=0}^\infty$ of $(\lie{g}, K)$-submodules of $V$ such that $\bigcup_{n=0}^\infty V_n = V$ and each $V_n$ is of finite length.
\end{definition}

If $V$ admits an invariant inner product and has at most countable dimension, then $V$ is discretely decomposable if and only if $V$ is completely reducible.
The condition `at most countable dimension' is inessential, and is automatically satisfied in the setting of the branching problem.

\begin{definition} \label{def:Admissible}
	Let $V$ be a $(\lie{g}, K)$-module.
	We say that $V$ is $K$-admissible if $\Hom_{K}(F, V|_{K})$ is finite dimensional for any finite-dimensional irreducible representation $F$ of $K$.
	We say that $V$ is $\lie{g}$-admissible if $V$ is discretely decomposable and has a filtration $V = \bigcup_{n=0}^\infty V_n$ such that each $V_n$ has finite length and $\lim_{n\to \infty} m(V_n, V') < \infty$ for any irreducible $(\lie{g}, K)$-module $V'$.
	Here $m(V_n, V')$ is the multiplicity of $V'$ in $V_n$.
\end{definition}

\begin{remark}
	In the definition of $\lie{g}$-admissibility, the finiteness of the limit $\lim_{n\to \infty} m(V_n, V')$ does not depend on the choice of the filtration.
	See \cite[Definition 3.2]{Os25}.
\end{remark}

In the context of the branching problem, $\lie{g}$-admissibility is equivalent to $K$-admissibility.
For unitarizable modules, this was proved in \cite{ZhLi10}, and for general modules, it was proved in \cite[Theorem 10]{Ki24}.
Thus, although we do not need the notion of $\lie{g}$-admissibility in this paper, we define it in order to compare conditions on wave front sets with the corresponding conditions on modules.
See Section \ref{section:Admissibility} for this comparison.

\subsection{Null cone, associated variety and wave front set}

In this subsection, we recall the notions of the null cone, associated variety and wave front set.
The Kostant--Sekiguchi correspondence and the Schmid--Vilonen theorem play a crucial role in linking conditions on wave front sets with those on associated varieties.
Retain the notation from the previous subsection.

\begin{definition}
	Let $V$ be a finite-dimensional real representation of a Lie group $G$.
	Define
	\begin{align*}
		\Nilpotent(V, G) \coloneq \set{v \in V: 0 \in \overline{Gv}}.
	\end{align*}
	We call $\Nilpotent(V, G)$ the \define{null cone} of the representation $V$.
\end{definition}

In \cite{Ki24}, we adopted the following definition of the null cone:
\begin{align*}
	\Nilpotent'(V, G) = \set{v \in V: f(v) = f(0) \ (\forall f \in \rring{V}^G)},
\end{align*}
where $\rring{V}^G$ is the ring of $G$-invariant polynomial functions on $V$.
Although the two definitions are equivalent if $G$ is algebraic and $V$ is a rational representation, they are not equivalent in general.
We bridge the gap between the two definitions by the following proposition.

\begin{proposition} \label{prop:NullCone}
	Let $(\pi, V)$ be a finite-dimensional real representation of a connected reductive Lie group $G$ and $W$ a subrepresentation of $V$.
	Assume that $V$ has an inner product such that $\pi(\theta(g)) = (\pi(g)^*)^{-1}$ for any $g \in G$, where $\theta$ is a Cartan involution of $G$.
	Let $O$ be a $G$-stable closed cone in $V$.
	Then $O \cap \Nilpotent'(W, G) = \set{0}$ if and only if $O \cap \Nilpotent(W, G) = \set{0}$.
\end{proposition}

\begin{proof}
	Since $\Nilpotent(W, G) \subset \Nilpotent'(W, G)$, the `only if' part is clear.
	We shall show the `if' part.
	Assume that $O \cap \Nilpotent'(W, G) \neq \set{0}$.

	Let $\lie{g'}$ be the ideal generated by $(\lie{g})^{-\theta}$, and $G'$ the analytic subgroup of $G$ with the Lie algebra $\lie{g}'$.
	Then $G'$ is closed and normal in $G$ and $G/G'$ is compact.
	Let $\lie{a}$ be a maximal abelian subspace of $\lie{g}^{-\theta}$.
	By assumption, $\pi(X)$ is self-adjoint for any $X \in \lie{a}$.

	Assume that any element of $O \cap \Nilpotent'(W, G)$ is $\lie{a}$-invariant.
	Let $0\neq X \in O \cap \Nilpotent'(W, G)$.
	Since $O \cap \Nilpotent'(W, G)$ is $G$-stable, $X$ is $(\Ad(g)\lie{a})$-invariant for any $g \in G$.
	This implies that $X$ is $G'$-invariant.
	Hence there exists a non-zero $(G/G')$-invariant polynomial $f$ on $W^{G'}$ such that $f(X) \neq f(0)$.
	Since the $G$-action on the polynomial ring on $W$ is completely reducible by assumption, $f$ can be extended to a $G$-invariant polynomial $\tilde{f}$ on $W$.
	Then $\tilde{f}(X) \neq \tilde{f}(0)$, which contradicts $X \in \Nilpotent'(W, G)$.

	Therefore, there exists $X \in O \cap \Nilpotent'(W, G)$ such that $\pi(\lie{a})X \neq 0$.
	In particular, we have $\lie{a} \neq 0$.
	Let $X = X_0 + X_1 + \cdots + X_n$ be the decomposition of $X$ into $\lie{a}$-weight vectors, where $X_i$ has weight $\lambda_i$ for $i = 0, \ldots, n$.
	Since $X$ is not $\lie{a}$-invariant, the convex hull of $\{\lambda_0, \ldots, \lambda_n\}$ contains a vertex $\lambda_j$ with $\lambda_j \neq 0$.
	Hence there exists an element $h \in \lie{a}$ such that $\lambda_j(h) \neq 0$ and $\lambda_j(h) > \lambda_i(h)$ for all $i \neq j$.
	Then we have
	\begin{align*}
		O \cap W \ni \lim_{t \to \infty} e^{-t\lambda_j(h)} \pi(\exp(th))X = X_j \neq 0.
	\end{align*}
	Since $X_j$ is an $\lie{a}$-weight vector with non-zero weight, we have $X_j \in \Nilpotent(W, G)$, and hence $O\cap \Nilpotent(W, G) \neq \set{0}$.
\end{proof}

The following proposition is proved in the same manner as Proposition \ref{prop:NullCone}. In Section \ref{section:Admissibility}, it will be used to relate $\lie{g}$-admissibility to discrete decomposability.

\begin{proposition} \label{prop:IntersectionW}
	Retain the notation from Proposition \ref{prop:NullCone}.
	Let $G'$ be the subgroup of $G$ defined in the proof of Proposition \ref{prop:NullCone}.
	Assume that $O\cap W^{G'} = \set{0}$.
	Then $O \cap W = \set{0}$ if and only if $O \cap \Nilpotent(W, G) = \set{0}$.
\end{proposition}

\begin{proof}
	If $O \cap W \neq \set{0}$, then there exists an $\lie{a}$-weight vector in $O\cap W$ with non-zero weight.
	The proof is similar to the proof of Proposition \ref{prop:NullCone}.
	We omit the details.
\end{proof}

We shall recall the definition of the associated variety of a $(\lie{g}, K)$-module.
We refer the reader to \cite{Vo89} for the theory of associated varieties of $(\lie{g}, K)$-modules.
Let $K_\CC$ be the complexification of $K$.

\begin{definition}
	Let $V$ be a finitely generated $\lie{g}$-module.
	Take a finite-dimensional generating subspace $V_0$ and define a filtration $V = \bigcup_{i=0}^\infty V_i$ by $V_i \coloneq \univfilt{i}{g} V_0$.
	The support of the $S(\lieC{g})$-module $\gr(V)$ is called the \define{associated variety} of $V$ and denoted by $\AV(V)$.
\end{definition}

More precisely, $\AV(V)$ is defined as
\begin{align*}
	\AV(V) = \set{X \in \lieC{g}^*: f(X) = 0 \ (\forall f \in \Ann_{S(\lieC{g})}(\gr(V)))}.
\end{align*}
The associated variety $\AV(V)$ does not depend on the choice of the generating subspace $V_0$.
See, e.g., \cite{Vo89}.

If necessary, replacing $V_0$ by $\univ{k}V_0$, we may assume that each $V_i$ is $K$-stable.
This implies that $\AV(V)$ is $K_\CC$-stable.
If $V$ is irreducible, then $V$ is finitely generated and $\univcent{g}$ acts on $V$ locally finitely.
We have the following fact.

\begin{fact}[See {\cite[Corollary 5.13]{Vo89}}] \label{fact:AssociatedVariety}
	Let $V$ be an irreducible $(\lie{g}, K)$-module.
	Then $\AV(V)$ is a $K_\CC$-stable closed conical subvariety in the nilpotent cone $\Nilpotent(\lie{k}_\CC^\perp, K_\CC)$.
\end{fact}

\begin{remark} \label{remark:AssociatedVariety}
	It is well known that
	\begin{align*}
		\Nilpotent(\lie{k}_\CC^\perp, K_\CC) = \lie{k}_{\CC}^\perp \cap \Nilpotent(\lieC{g}^*, G_\CC).
	\end{align*}
\end{remark}

We shall recall the definition of the wave front set of a $(\lie{g}, K)$-module.
We refer the reader to \cite[Chapter VIII]{Ho03} for wave front sets of distributions, and to \cite{Ho81} and \cite{HaHeOl16} for wave front sets of representations.
Note that the definition in \cite{HaHeOl16} differs from that in \cite{Ho81}.
For unitary representations, the two definitions are equivalent by \cite[Proposition 2.4]{HaHeOl16}.
Since the proof does not rely on unitarity, the equivalence also holds for Hilbert representations.

Note that any irreducible $(\lie{g}, K)$-module admits a Hilbert globalization, that is, an irreducible admissible Hilbert representation of $G$ whose underlying $(\lie{g}, K)$-module is isomorphic to the given $(\lie{g}, K)$-module.
This is a consequence of the Harish-Chandra subquotient theorem \cite[Theorem 3.5.6]{Wa88_real_reductive_I}.

\begin{definition}
	Let $V$ be an irreducible $(\lie{g}, K)$-module.
	Take a Hilbert globalization $\Hilbert{H}$ of $V$.
	We set
	\begin{align*}
		\WF(V) = \left(T^*_e(G) \cap \overline{\bigcup_{v, w \in \Hilbert{H}} \WF(\Phi_{v,w})}\right) \cup \set{0},
	\end{align*}
	where $\WF(\Phi_{v,w})$ is the wave front set of the matrix coefficient $\Phi_{v,w}(g) = \langle \pi(g)v, w\rangle$ and $T^*_e(G)$ is the cotangent space of $G$ at the identity.
	We call $\WF(V)$ the \define{wave front set} of $V$.
\end{definition}

The wave front set $\WF(V)$ does not depend on the choice of a Hilbert globalization $\Hilbert{H}$.
In fact, the wave front set depends only on the smooth Fr\'echet globalization $V^\infty$ of moderate growth \cite[Proposition 3.2]{HaHeOl16}.
Although \cite[Proposition 3.2]{HaHeOl16} is stated for unitary representations, one may prove the result for Hilbert representations with a slight modification of \cite[Lemma 3.3]{HaHeOl16}.

By the definition of wave front sets of distributions, $\WF(V)$ is a closed conical subset of $\lie{g}^*$.
Since any matrix coefficient $\Phi_{v,w}$ satisfies certain differential equations determined by the infinitesimal character of $V$, the wave front set $\WF(V)$ is contained in $\Nilpotent(\lie{g}^*, G)$.

The wave front set of a $(\lie{g}, K)$-module can be regarded as an analytic counterpart of the associated variety of the $(\lie{g}, K)$-module.
This relation is made precise by the Kostant--Sekiguchi correspondence and the Schmid--Vilonen theorem.
Fix a $G$-invariant non-degenerate symmetric bilinear form $(\cdot, \cdot)$ on $\lie{g}$ such that $-(\cdot, \theta(\cdot))$ is positive definite and symmetric.
We identify $\lie{g}$ with $\lie{g}^*$ via the bilinear form $(\cdot, \cdot)$.
Then $\AV(V)$ and $\WF(V)$ are regarded as subsets of $\lieC{g}$.

Let $\orbit{O}$ be a nilpotent orbit in $\lie{g}$.
Then there exists an element $X \in \orbit{O}$ such that $\set{H, X, Y}$ forms an $\lieSL_2$-triple, where $H\coloneq [X, -\theta(X)]$ and $Y\coloneq -\theta(X)$.
We have another $\lieSL_2$-triple:
\begin{align*}
	H' = -\sqrt{-1}(X-Y), \quad X'=\frac{1}{2}(X+Y-\sqrt{-1}H), \quad Y'=\frac{1}{2}(X+Y+\sqrt{-1}H).
\end{align*}
We set
\begin{align*}
	\KS(\orbit{O}) \coloneq K_\CC \cdot X'.
\end{align*}
Note that $\KS(\orbit{O})$ does not depend on the choice of the $\lieSL_2$-triple.
For a $G$-stable subset $S$ of $\Nilpotent(\lie{g}^*, G)$, we set
\begin{align*}
	\KS(S) \coloneq \bigcup_{\orbit{O} \subset S} \KS(\orbit{O}),
\end{align*}
where the union is taken over all nilpotent $G$-orbits $\orbit{O}$ contained in $S$.
See \cite[Theorem 9.5.1]{CoMc93} for the Kostant--Sekiguchi correspondence.

\begin{fact}[Kostant--Sekiguchi correspondence] \label{fact:KostantSekiguchi}
	$\KS$ gives a bijection between the set of nilpotent $G$-orbits in $\lie{g}$ and the set of nilpotent $K_\CC$-orbits in $\lie{k}_\CC^\perp$.
\end{fact}

More precisely, $\KS$ preserves the closure relation of nilpotent orbits.
In particular, a $G$-stable subset $S$ of $\Nilpotent(\lie{g}^*, G)$ is closed if and only if $\KS(S)$ is closed.
See \cite{BaSe98}.

\begin{fact}[Schmid--Vilonen {\cite{ScVi00}}] \label{fact:SchmidVilonen}
	Assume that $G$ is linear.
	Let $V$ be an irreducible $(\lie{g}, K)$-module.
	Then one has $\AV(V) = \KS(\WF(V))$.
\end{fact}

We remark that the Kostant--Sekiguchi correspondence is compatible with embeddings of subalgebras.
Let $\lie{h}$ be a $\theta$-stable subalgebra of $\lie{g}$ and $H$ the analytic subgroup of $G$ with the Lie algebra $\lie{h}$.

\begin{proposition} \label{prop:KostantSekiguchiCompatible}
	Let $\orbit{O}$ be a nilpotent $G$-orbit in $\lie{g}$.
	Then $\orbit{O}\cap \lie{h}$ is contained in the nilpotent cone $\Nilpotent(\lie{h}, H)$, and we have
	\begin{align*}
		\KS(\orbit{O}\cap \lie{h}) = \KS(\orbit{O}) \cap \lieC{h}.
	\end{align*}
\end{proposition}

\begin{proof}
	The first assertion is a consequence of the Jordan decomposition.
	Let $\orbit{O}'\subset \orbit{O} \cap \lie{h}$ be an $H$-orbit.
	Let $K_{H,\CC}$ be the analytic subgroup of $K_\CC$ with the Lie algebra $\lie{k}_{H,\CC} = \lieC{h} \cap \lieC{k}$.
	Take an $\lieSL_2$-triple $\set{h, x, y}$ such that $x \in \orbit{O}'$ and $y = -\theta(x)$.
	Then we have
	\begin{align*}
		\KS(\orbit{O}') = \Ad(K_{H,\CC})\frac{1}{2}(x + y - \sqrt{-1}h) \subset \KS(\orbit{O}) \cap \lieC{h}.
	\end{align*}
	This shows that $\KS(\orbit{O}\cap \lie{h}) \subset \KS(\orbit{O}) \cap \lieC{h}$.
	The converse follows from a similar argument replacing $\KS$ by $\KS^{-1}$.
\end{proof}

\subsection{Criterion for discrete decomposability}

We shall state the criterion for discrete decomposability of $(\lie{g}, K)$-modules and the main theorem of this paper.
Let $G$ be a connected real reductive Lie group with a Cartan involution $\theta$ and $K = G^\theta$.
Let $H$ be a $\theta$-stable closed connected reductive subgroup of $G$.
Set $K_H = K \cap H$.

Fix a $G$-invariant non-degenerate symmetric bilinear form $(\cdot, \cdot)$ on $\lie{g}$ such that $-(\cdot, \theta(\cdot))$ is positive definite and symmetric.
For a subspace $W$ of $\lieC{g}$ such that $(\cdot, \cdot)$ is non-degenerate on $W$, we denote by $\proj_{W}\colon \lieC{g} \to W$ the orthogonal projection to $W$ with respect to the bilinear form $(\cdot, \cdot)$.

In view of the Schmid--Vilonen theorem (Fact \ref{fact:SchmidVilonen} and see also Hypothesis \ref{hyp:SchmidVilonen}), $\AV(V)$ and $\WF(V)$ carry the same information.
Thus, by considering the analytic counterpart of Kobayashi's necessary condition for discrete decomposability formulated in terms of $\AV(V)$, we obtain a criterion using $\WF(V)$.
In the following theorem, the implication (2) $\Rightarrow$ (3) is proved in \cite[Corollary 3.4]{Ko98_discrete_decomposable_3} and the implication (1) $\Rightarrow$ (2) is proved in \cite[Theorem 17]{Ki24}.

\begin{fact} \label{fact:Implication}
	Let $V$ be an irreducible $(\lie{g}, K)$-module.
	Then the implication (1) $\Rightarrow$ (2) $\Rightarrow$ (3) holds.
	\begin{enumparen}
		\item $\WF(V)\cap \Nilpotent(\lie{h}^\perp, H) = \set{0}$.
		\item $V|_{\lie{h}, K_H}$ is discretely decomposable.
		\item $\proj_{\lieC{h}}(\AV(V)) \subset \Nilpotent((\lie{k}_{H, \CC})^\perp, K_{H, \CC})$.
	\end{enumparen}
\end{fact}

An interesting point is that the direction which is easy to prove depends on the formulation: in terms of $\AV(V)$, it is the necessary condition, while in terms of $\WF(V)$, it is the sufficient condition.
Therefore, once we prove the purely Lie-algebraic implication (3) $\Rightarrow$ (1), the equivalence of all the conditions in Fact \ref{fact:Implication} follows.

Under some assumptions, we have proved the implication (3) $\Rightarrow$ (1) in \cite[Theorem 21]{Ki24}.
For example, if $H$ is a symmetric subgroup, then the conditions in Fact \ref{fact:Implication} are equivalent.
Our main purpose in this paper is to show that the conditions in Fact \ref{fact:Implication} are equivalent without any assumptions on the subgroup $H$ of $G$.

\begin{theorem} \label{thm:MainTheorem}
	The conditions in Fact \ref{fact:Implication} are equivalent.
\end{theorem}

By the definition of the Kostant--Sekiguchi correspondence, Theorem \ref{thm:MainTheorem} follows from the following theorem, which is stronger than the implication (3) $\Rightarrow$ (1).

\begin{theorem} \label{thm:MainTarget}
	Let $\orbit{O}$ be a nilpotent $G$-orbit in $\lie{g}$ such that $\overline{\orbit{O}}\cap \Nilpotent(\lie{h}^\perp, H) \neq \set{0}$.
	Then $\proj_{\lieC{h}}(\KS(\overline{\orbit{O}}))$ is not contained in $\Nilpotent((\lie{k}_{H, \CC})^\perp, K_{H, \CC})$.
\end{theorem}

\begin{remark}
	We will prove Theorem \ref{thm:MainTarget} without Hypothesis \ref{hyp:SchmidVilonen}.
	The hypothesis is needed to show Theorem \ref{thm:MainTheorem} from Theorem \ref{thm:MainTarget}.
\end{remark}

We show the theorem by reducing it to the case where $H$ has real rank one and the embedding $\lie{h}\subset \lie{g}$ satisfies some strong conditions (see Definitions \ref{def:typeA} and \ref{def:typeBC}).
After that, we classify the possible cases and reduce them to the minimal case $(\lie{g}, \lie{h}) = (\lieSL(2,\RR), \RR)$ by a case-by-case analysis.

\section{Reduction} \label{section:Reduction}

In this section, we give a procedure to reduce the proof of Theorem \ref{thm:MainTarget} to a classification problem.

\subsection{Setting and conventions}

From this section onward, we use the following convention.
For a real reductive Lie algebra $\lie{g}$ with a Cartan involution $\theta$, we fix a connected reductive Lie group $G$ with the Lie algebra $\lie{g}$ such that $G$ has a Cartan involution $\theta$ whose differential is the Cartan involution $\theta$ on $\lie{g}$.
Fix also a $G$-invariant non-degenerate symmetric bilinear form $(\cdot, \cdot)$ on $\lie{g}$ such that $-(\cdot, \theta(\cdot))$ is positive definite and symmetric.
Extend $(\cdot, \cdot)$ to a bilinear form on $\lieC{g}$.
Set $K \coloneq G^\theta$.
For a $\theta$-stable subgroup $H$ of $G$, we set $K_H \coloneq H \cap K$.

Throughout this section, for a Lie subalgebra of $\lie{g}$, we denote by the corresponding capital Roman letter the analytic subgroup of $G$ with that Lie algebra; for example, $H$ denotes the analytic subgroup of $G$ with Lie algebra $\lie{h}$.
In several propositions, we use reductions involving analytic subgroups.
We neither assume nor verify that these analytic subgroups are closed.
The arguments are valid without the closedness assumption.

We say that a $\theta$-stable subalgebra $\lie{h}$ of $\lie{g}$ is compact if $\lie{h} \subset \lie{k}$.
Note that $H$ may not be compact even if $\lie{h}$ is compact because $H$ may not be closed in $G$.
Clearly, $H$ is relatively compact in $K$.

\subsection{Minimal tuples}

In this subsection, we prepare several notions and lemmas to reduce the proof of Theorem \ref{thm:MainTarget} to small cases.
Let $\lie{g}$ be a real reductive Lie algebra with Cartan involution $\theta$, $\lie{h}$ a $\theta$-stable reductive subalgebra of $\lie{g}$ and $\orbit{O}$ a nilpotent orbit in $\lie{g}$.

\begin{definition}
	We say that the triple $(\lie{g}, \lie{h}, \orbit{O})$ is NDD (non-discretely decomposable) if $\overline{\orbit{O}} \cap \Nilpotent(\lie{h}^\perp, H) \neq \set{0}$.
\end{definition}

\begin{definition} \label{def:MainAssertion}
	We denote by $\assertion(\lie{g}, \lie{h}, \orbit{O})$ the assertion that $\proj_{\lieC{h}}(\KS(\overline{\orbit{O}})) \not \subset \Nilpotent((\lie{k}_{H, \CC})^\perp, K_{H, \CC})$.
\end{definition}

\begin{remark}
	The two definitions are independent of the choice of the bilinear form $(\cdot, \cdot)$ as follows.
	In fact, the following diagram commutes:
	\begin{align*}
		\xymatrix{
			\lieC{g} \ar[r] \ar[d]^{\proj_{\lieC{h}}} & \lieC{g}^* \ar[d]^{\mathrm{res}_{\lieC{h}}} & \lieC{g} \ar[l] \ar[d]^{\proj'_{\lieC{h}}} \\
			\lieC{h} \ar[r] & \lieC{h}^* & \lieC{h} \ar[l],
		}
	\end{align*}
	where $\proj'_{\lieC{h}}$ is the orthogonal projection with respect to another $G$-invariant non-degenerate symmetric bilinear form, the horizontal maps are linear isomorphisms induced by the bilinear forms and $\mathrm{res}_{\lieC{h}}$ is the restriction map to $\lieC{h}$.
	Clearly, all the maps in the diagram are $\lieC{h}$-equivariant.
	Moreover, since $\orbit{O}$ is closed under the $\RR_{>0}$-action on each simple ideal of $\lie{g}$, the bijection $\Nilpotent(\lie{g}, G)/G \simeq \Nilpotent(\lie{g}^*, G)/G$ is independent of the choice of the bilinear form.
\end{remark}

It is obvious that Theorem \ref{thm:MainTarget} follows if the assertion $\assertion(\lie{g}, \lie{h}, \orbit{O})$ is proved for any NDD triple $(\lie{g}, \lie{h}, \orbit{O})$.
Hereafter, assume that $(\lie{g}, \lie{h}, \orbit{O})$ is NDD.
Let $\lie{a}_H$ be a maximal abelian subspace of $\lie{h}^{-\theta}$.

\begin{lemma}\label{lem:ExistenceWeightVector}
	Let $\lie{a}$ be a maximal abelian subspace of $\lieNorm_{\lie{g}}(\lie{h})^{-\theta}$ containing $\lie{a}_H$.
	There exists an $\lie{a}$-weight vector $X \in \overline{\orbit{O}} \cap \Nilpotent(\lie{h}^\perp, H)$ with non-zero weight on $\lie{a}_H$.
\end{lemma}

\begin{proof}
	Let $\lie{h}_{\mathrm{nc}}$ be the subalgebra of $\lie{h}$ generated by $\lie{h}^{-\theta}$.
	Assume that any element in $\overline{\orbit{O}} \cap \Nilpotent(\lie{h}^\perp, H)$ is $\lie{a}_H$-invariant.
	Then any element in $\overline{\orbit{O}} \cap \Nilpotent(\lie{h}^\perp, H)$ is $\lie{h}_{\mathrm{nc}}$-invariant since $\overline{\orbit{O}}\cap \Nilpotent(\lie{h}^\perp, H)$ is $K_H$-stable.
	For any $X \in \overline{\orbit{O}} \cap \Nilpotent(\lie{h}^\perp, H)$, we have
	\begin{align*}
		0 \in \overline{\Ad(H)X} \subset \set{Z \in \lie{g} : (Z, \theta(Z)) = (X, \theta(X))}
	\end{align*}
	and hence $X = 0$.
	This contradicts the assumption that $(\lie{g}, \lie{h}, \orbit{O})$ is NDD.

	Thus, there exists $X \in \overline{\orbit{O}} \cap \Nilpotent(\lie{h}^\perp, H)$ that is not $\lie{a}_H$-invariant.
	Write $X = X_1 + X_2 + \cdots + X_k$ according to the $\lie{a}_H$-weight space decomposition, where $X_i$ has weight $\lambda_i$.
	The convex hull of $\set{\lambda_1, \lambda_2, \ldots, \lambda_k}$ is not $\set{0}$.
	Hence there exist $\lambda_i$ and $H \in \lie{a}_H$ such that $\lambda_i \neq 0$ and $\lambda_i(H) > \lambda_j(H)$ for all $j \neq i$.
	Since $\overline{\orbit{O}}$ is a $G$-stable closed cone, we have
	\begin{align*}
		X_i = \lim_{t \to \infty} e^{-t\lambda_i(H)}\Ad(e^{tH})X \in \overline{\orbit{O}} \cap \Nilpotent(\lie{h}^\perp, H).
	\end{align*}
	$X_i$ is the desired $\lie{a}_H$-weight vector with non-zero weight $\lambda_i$.

	Since $\lie{a}$ normalizes $\lie{h}^\perp$, we can apply the same argument to the $\lie{a}_H$-weight vector $X_i$, with $\lie{a}$ in place of $\lie{a}_H$.
	Then we obtain an $\lie{a}$-weight vector in $\overline{\orbit{O}} \cap \Nilpotent(\lie{h}^\perp, H)$ with non-zero weight on $\lie{a}_H$.
	This proves the lemma.
\end{proof}

Before discussing reductions for NDD tuples, we prepare a lemma which allows us to remove the unnecessary factors of $\lie{g}$ and $\lie{h}$.

\begin{lemma} \label{lem:ReductionIdeal}
	Let $\lie{g'}$ be a $\theta$-stable ideal of $\lie{g}$ such that $\orbit{O}\subset \lie{g'}$.
	Let $\lie{h'}$ be the orthogonal projection of $\lie{h}$ to $\lie{g'}$.
	Then $(\lie{g'}, \lie{h'}, \orbit{O})$ is NDD.
	Moreover, $\assertion(\lie{g'}, \lie{h'}, \orbit{O})$ holds if and only if $\assertion(\lie{g}, \lie{h}, \orbit{O})$ holds.
\end{lemma}

\begin{remark}
	The assumption on $\lie{g'}$ is satisfied if $(\lie{g'})^\perp$ is a direct sum of abelian ideals and compact ideals.
\end{remark}

\begin{proof}
	Clearly, we have $\lie{h}^\perp \cap \lie{g'} = (\lie{h'})^\perp \cap \lie{g'}$
	and the $H'$-orbits in $\lie{g'}$ are the same as the $H$-orbits in $\lie{g'}$.
	In particular, $(\lie{g'}, \lie{h'}, \orbit{O})$ is NDD.

	Let $Z \in \lieC{g}'$.
	It is sufficient to show that $\proj_{\lieC{h}'}(Z)$ is nilpotent if and only if $\proj_{\lieC{h}}(Z)$ is nilpotent (see Remark \ref{remark:AssociatedVariety}).
	If $Z \in \lieC{g}'\cap (\lieC{h}')^\perp$, then we have $\proj_{\lieC{h}'}(Z) = \proj_{\lieC{h}}(Z) = 0$.
	Hence we may assume that $Z \in \lieC{h}'$.
	Note that $\proj_{\lieC{h}}|_{\lieC{h}'}$ is injective and $\lieC{h}$-equivariant.
	Hence $\overline{\Ad(H'_\CC)Z} \ni 0$ if and only if $\overline{\Ad(H_\CC)\proj_{\lieC{h}}(Z)} \ni 0$.
	This implies that $\proj_{\lieC{h}}(Z)$ is nilpotent if and only if $\proj_{\lieC{h}'}(Z) = Z$ is nilpotent.
\end{proof}

\begin{lemma} \label{lem:ReductionCompactFactor}
	Let $\lie{k'}$ be a subalgebra of $\lieCent_{\lie{k}}(\lie{h}) \cap \lie{h}^\perp$.
	Then $(\lie{g}, \lie{h}, \orbit{O})$ is NDD if and only if $(\lie{g}, \lie{h} \oplus \lie{k'}, \orbit{O})$ is NDD.
	Moreover, $\assertion(\lie{g}, \lie{h}, \orbit{O})$ holds if and only if $\assertion(\lie{g}, \lie{h} \oplus \lie{k'}, \orbit{O})$ holds.
\end{lemma}

\begin{proof}
	Let $p \colon \lie{g} \rightarrow \lie{k'}$ be the orthogonal projection.
	Since $\lie{k'}$ is contained in $\lie{g}^H$, we have $p(\Nilpotent(\lie{h}^\perp, H)) \subset \Nilpotent(\lie{k'}, H) = \set{0}$, and hence
	\begin{align*}
		\overline{\orbit{O}} \cap \Nilpotent(\lie{h}^\perp, H) = \overline{\orbit{O}} \cap \Nilpotent((\lie{h} \oplus \lie{k'})^\perp, H).
	\end{align*}
	Since $K' \subset K$, the analytic subgroup $K' \subset G$ preserves the inner product $-(\cdot, \theta(\cdot))$ on $\lie{g}$.
	This implies that 
	\begin{align*}
		\overline{\orbit{O}} \cap \Nilpotent(\lie{h}^\perp, H) = \overline{\orbit{O}} \cap \Nilpotent((\lie{h} \oplus \lie{k'})^\perp, H) = \overline{\orbit{O}} \cap \Nilpotent((\lie{h} \oplus \lie{k'})^\perp, H\times K').
	\end{align*}
	Hence $(\lie{g}, \lie{h}, \orbit{O})$ is NDD if and only if $(\lie{g}, \lie{h} \oplus \lie{k'}, \orbit{O})$ is NDD.

	Recall that $\KS(\overline{\orbit{O}})$ is contained in $(\lieC{g})^{-\theta}$.
	See Fact \ref{fact:KostantSekiguchi}.
	Hence, for fixed $\lie{g}$ and $\orbit{O}$, the assertion $\assertion(\lie{g}, \lie{h}, \orbit{O})$ depends on $\lie{h}$ only through $(\lieC{h})^{-\theta}$.
	This shows that $\assertion(\lie{g}, \lie{h}, \orbit{O})$ holds if and only if $\assertion(\lie{g}, \lie{h} \oplus \lie{k'}, \orbit{O})$ holds.
\end{proof}

\begin{lemma} \label{lem:ReductionNilpotent}
	Let $X$ be an $\lie{a}_H$-weight vector in $\overline{\orbit{O}} \cap \Nilpotent(\lie{h}^\perp, H)$ with non-zero weight.
	Define $\lie{g'}$ to be the smallest subalgebra of $\lie{g}$ that contains $\lie{a}_H$, $X$ and $\theta(X)$, and satisfies $\lie{g'} = (\lie{g'}\cap \lie{h}) \oplus (\lie{g'}\cap \lie{h}^\perp)$.
	Set $\lie{h'} \coloneq \lie{g'} \cap \lie{h}$ and $\orbit{O}' \coloneq \Ad(G')X$.
	Then $(\lie{g'}, \lie{h'}, \orbit{O}')$ is NDD.
	Moreover, if $\assertion(\lie{g}', \lie{h}', \orbit{O}')$ holds, then so does $\assertion(\lie{g}, \lie{h}, \orbit{O})$.
\end{lemma}

\begin{proof}
	By construction, the weight vector $X$ belongs to $\orbit{O}'\cap \Nilpotent(\lie{g'}\cap \lie{h}^\perp, H')$.
	Hence $(\lie{g'}, \lie{h'}, \orbit{O}')$ is NDD.
	
	Since $\lie{g'} = (\lie{g'}\cap \lie{h}) \oplus (\lie{g'}\cap \lie{h}^\perp)$, we have $\proj_{\lieC{h}}|_{\lieC{g}'} = \proj_{\lieC{h'}}|_{\lieC{g}'}$. Recall that $\KS(\overline{\orbit{O}'}) \subset \KS(\overline{\orbit{O}})$ by Proposition \ref{prop:KostantSekiguchiCompatible}.
	The second assertion follows from this.
\end{proof}

In Lemmas \ref{lem:ReductionIdeal}, \ref{lem:ReductionCompactFactor} and \ref{lem:ReductionNilpotent}, we have given reductions for NDD tuples.
We shall state them explicitly.

\begin{definition}\label{def:Reduction}
	We define three reductions for NDD tuples $(\lie{g}, \lie{h}, \orbit{O})$ as follows.
	\begin{enumerate}[label=(\alph*)]
		\item Let $\lie{g}'$ be the smallest ideal of $\lie{g}$ containing $\orbit{O}$.
		Define $\lie{h}'$ as in Lemma \ref{lem:ReductionIdeal}.
		Replace $(\lie{g}, \lie{h}, \orbit{O})$ by $(\lie{g'}, \lie{h'}, \orbit{O})$.
		\item Let $\lie{i}$ be the sum of all compact ideals of $\lie{h}$ and $\lie{h'}$ be the orthogonal complement of $\lie{i}$ in $\lie{h}$.
		Replace $(\lie{g}, \lie{h}, \orbit{O})$ by $(\lie{g}, \lie{h'}, \orbit{O})$.
		\item Choose a $G$-invariant non-degenerate symmetric bilinear form $(\cdot, \cdot)$ on $\lie{g}$ such that $-(\cdot, \theta(\cdot))$ is positive definite and symmetric.
		Let $X$ be an $\lie{a}_H$-weight vector in $\overline{\orbit{O}} \cap \Nilpotent(\lie{h}^\perp, H)$ with non-zero weight.
		Define $\lie{g'}$ and $\lie{h'}$ as in Lemma \ref{lem:ReductionNilpotent}.
		Replace $(\lie{g}, \lie{h}, \orbit{O})$ by $(\lie{g'}, \lie{h'}, \Ad(G')X)$.
	\end{enumerate}
\end{definition}

\begin{remark} \label{remark:minimal}
	In Definition \ref{def:Reduction} (c), the choice of the $G$-invariant bilinear form $(\cdot, \cdot)$ on $\lie{g}$ is crucial.
	Let $\lie{h} \coloneq \lieSL(2,\RR) \hookrightarrow \lieSL(2,\RR)\oplus \lieSL(2,\RR) \eqcolon \lie{g}$ be the diagonal embedding.
	Consider the Killing form on $\lie{g}$.
	Then $\lie{h}^\perp$ is the anti-diagonal embedding of $\lieSL(2,\RR)$, and hence $\lie{h}^\perp$ contains a non-zero nilpotent element $X$ such that $\set{[X, -\theta(X)], X, -\theta(X)}$ forms an $\lieSL_2$-triple.
	In other words, $\lie{g}'$ is a proper subalgebra of $\lie{g}$ for the choice of $X$.
	
	On the other hand, consider the bilinear form on $\lie{g}$ defined by the direct sum of the Killing form on the first factor and some scalar multiple of the Killing form on the second factor.
	Then $\lie{h}^\perp$ may not contain any non-zero nilpotent element $X$ such that $\set{[X, -\theta(X)], X, -\theta(X)}$ forms an $\lieSL_2$-triple.
	Hence $\lie{g}' = \lie{g}$ for any choice of $X$.
\end{remark}

\begin{definition} \label{def:minimal}
	If $(\lie{g}, \lie{h}, \orbit{O})$ is unchanged under every choice of the reductions (a), (b) and (c) in Definition \ref{def:Reduction}, we say that the tuple $(\lie{g}, \lie{h}, \orbit{O})$ is minimal.
\end{definition}

For any NDD tuple $(\lie{g}, \lie{h}, \orbit{O})$, after finitely many reductions (a), (b) and (c) in Definition \ref{def:Reduction}, one obtains a minimal tuple.
Note that there are only finitely many nilpotent orbits in $\lie{g}$.

The minimality condition is so strong that essentially only one case can occur.
By Lemmas \ref{lem:ReductionIdeal}, \ref{lem:ReductionCompactFactor} and \ref{lem:ReductionNilpotent}, our goal is reduced to proving the following theorem.

\begin{theorem} \label{thm:MainTargetMinimal}
	If an NDD tuple $(\lie{g}, \lie{h}, \orbit{O})$ is minimal, then one has $(\lie{g}, \lie{h}) \simeq (\lieSL(2,\RR), \RR)$, where $\RR$ means the subalgebra of diagonal matrices in $\lieSL(2,\RR)$.
\end{theorem}

\begin{remark}
	If $(\lie{g}, \lie{h}) \simeq (\lieSL(2,\RR), \RR)$, the assertion $\assertion(\lie{g}, \lie{h}, \orbit{O})$ is easy to verify.
	It is also easy to see that $(\lie{g}, \lie{h}, \orbit{O})$ is minimal.
\end{remark}

\begin{remark}
	In the proof of \cite[Theorem 21]{Ki24}, we used the existence of an $\lieSL_2$-triple $\set{[X, -\theta(X)], X, -\theta(X)}$ such that $X \in \Nilpotent(\lie{h}^\perp, H) \cap \overline{\orbit{O}}$.
	Theorem \ref{thm:MainTargetMinimal} does not assert that, for any NDD tuple $(\lie{g}, \lie{h}, \orbit{O})$, there exists such an $\lieSL_2$-triple.
	In fact, this is not true when $(\lie{g}, \lie{h}) = (\lieSO(4, 3), \lieSU(2, 1))$ and $(\lie{g}, \lie{h}) = (\lieG_{2(2)}, \lieSU(2, 1))$.
\end{remark}

If $\lie{h}$ is abelian, Theorem \ref{thm:MainTargetMinimal} is easy to prove.
We will see this in Theorem \ref{thm:MinimalAbelian}.
The remaining proof of Theorem \ref{thm:MainTargetMinimal} will be given in Sections \ref{section:TypeBC1} and \ref{section:TypeA1}.
In the remaining part of this section, we give several necessary conditions for minimal tuples.

To show the non-minimality of an NDD tuple $(\lie{g}, \lie{h}, \orbit{O})$, it is enough to find a suitable subalgebra $\lie{g}'$ of $\lie{g}$ instead of finding a suitable nilpotent element $X$ in $\lie{h}^\perp \cap \overline{\orbit{O}}$.

\begin{proposition} \label{prop:NonMinimalSubalgebra}
	Assume that there exists a $\theta$-stable reductive subalgebra $\lie{g}'\subsetneq \lie{g}$ satisfying the following conditions.
	\begin{enumparen}
		\item $\lie{g}'$ is $\lie{a}_H$-stable.
		\item $\lie{g}' = (\lie{g}' \cap \lie{h}) \oplus (\lie{g}' \cap \lie{h}^\perp)$.
		\item There exists an $\lie{a}_H$-weight vector $X$ in $\lie{g}' \cap \lie{h}^\perp \cap \overline{\orbit{O}}$ with non-zero weight.
	\end{enumparen}
	Then $(\lie{g}, \lie{h}, \orbit{O})$ is not minimal.
\end{proposition}

\begin{remark}
	An $\lie{a}_H$-weight vector in $\lie{h}^\perp$ with non-zero weight automatically belongs to $\Nilpotent(\lie{h}^\perp, H)$.
\end{remark}

\begin{proof}
	If $\lie{g}$ has non-trivial center, then $(\lie{g}, \lie{h}, \orbit{O})$ is not minimal by definition.
	Hence we may assume that $\lie{g}$ is semisimple.

	Since $\lie{g}'$ is $\lie{a}_H$-stable, $\lie{g}' + \lie{a}_H$ is a subalgebra of $\lie{g}$.
	Let $\lie{g}''$ be the smallest subalgebra of $\lie{g}$ that contains $\lie{a}_H$, $X$ and $\theta(X)$, and satisfies $\lie{g}'' = (\lie{g}''\cap \lie{h}) \oplus (\lie{g}''\cap \lie{h}^\perp)$.
	By assumption, we have $\lie{g}'' \subset \lie{g}' + \lie{a}_H$.
	This implies that $[\lie{g}'', \lie{g}''] \subset \lie{g}' \subsetneq \lie{g}$.
	Since $\lie{g}$ is semisimple, $\lie{g}''$ is a proper subalgebra of $\lie{g}$.
	This means that reduction (c) in Definition \ref{def:Reduction} replaces $\lie{g}$ by the proper subalgebra $\lie{g}''$.
	This shows that $(\lie{g}, \lie{h}, \orbit{O})$ is not minimal.
\end{proof}

\begin{corollary} \label{cor:MinimalGeneratedBy}
	Let $X$ be an $\lie{a}_H$-weight vector in $\overline{\orbit{O}} \cap \Nilpotent(\lie{h}^\perp, H)$ with non-zero weight $\lambda$.
	Assume that $(\lie{g}, \lie{h}, \orbit{O})$ is minimal.
	Then each $\lie{a}_H$-weight in $\lie{g}$ is of the form $n\lambda$ ($n \in \ZZ$).
\end{corollary}

\begin{proof}
	Set $\lie{g'} \coloneq \bigoplus_{n \in \ZZ} \lie{g}_{n\lambda}$.
	Then we have $X, \theta(X) \in \lie{g'}$ and $\lie{a}_H \subset \lie{g'}$.
	Since $\lie{g'}$ is $\theta$-stable and satisfies the assumptions in Proposition \ref{prop:NonMinimalSubalgebra}, we have $\lie{g'} =\lie{g}$ by minimality.
	This shows the assertion.
\end{proof}

The following proposition follows from the definition of minimal tuples and Corollary \ref{cor:MinimalGeneratedBy}.

\begin{proposition} \label{prop:BasicMinimal}
	Assume that $(\lie{g}, \lie{h}, \orbit{O})$ is minimal.
	Then the following conditions hold.
	\begin{enumparen}
		\item $\lie{g}$ is semisimple without any compact factor.
		\item $\lie{h}$ has no compact factor.
		\item $\lie{h}$ has real rank one.
		\item $\lie{h}$ is simple or $1$-dimensional non-compact abelian.
		\item $\orbit{O} = \Ad(G)X$ for any non-zero $X \in \overline{\orbit{O}} \cap \Nilpotent(\lie{h}^\perp, H)$.
		\item The projection of any non-zero element in $\overline{\orbit{O}} \cap \Nilpotent(\lie{h}^\perp, H)$ to each simple factor of $\lie{g}$ is non-zero.
	\end{enumparen}
\end{proposition}

\begin{proposition} \label{prop:BasicMinimalH}
	Assume that $(\lie{g}, \lie{h}, \orbit{O})$ is minimal.
	Then the projection of $\lie{h}$ to each simple factor of $\lie{g}$ is non-zero.
\end{proposition}

\begin{proof}
	Let $\lie{i}$ be a simple ideal of $\lie{g}$ and $p\colon \lie{g} \rightarrow \lie{i}$ be the projection.
	Take a non-zero $X \in \overline{\orbit{O}} \cap \Nilpotent(\lie{h}^\perp, H)$.
	By Proposition \ref{prop:BasicMinimal} (6), we have
	\begin{align*}
		0 \neq p(X) \in p(\Nilpotent(\lie{h}^\perp, H)) \subset \Nilpotent(\lie{i}, H).
	\end{align*}
	This implies that $p(\lie{h})$ is non-zero.
\end{proof}

\subsection{Necessary conditions for minimality}

In \cite[Theorem 21]{Ki24}, we have given sufficient conditions for the implication (3) $\Rightarrow$ (1) in Fact \ref{fact:Implication}.
In this subsection, we rewrite the sufficient conditions in the context of NDD tuples.
Let $(\lie{g}, \lie{h}, \orbit{O})$ be an NDD tuple.
Assume that $\lie{g} \not \simeq \lieSL(2,\RR)$.

\begin{lemma} \label{lem:NonMinimalX}
	Assume that there exists an $\lie{a}_H$-weight vector $X \in \lie{h}^\perp \cap \overline{\orbit{O}}$ with non-zero weight such that $\set{[X, -\theta(X)], X, -\theta(X)}$ spans a subalgebra of $\lie{g}$ isomorphic to $\lieSL(2,\RR)$.
	Then $(\lie{g}, \lie{h}, \orbit{O})$ is not minimal.
\end{lemma}

\begin{proof}
	To show the assertion, we may assume that $\lie{g}$ is semisimple.
	Let $\lie{g}'$ be the subalgebra generated by $X, -\theta(X)$ and $\lie{a}_H$.
	By assumption, $[\lie{g}', \lie{g}']$ is isomorphic to $\lieSL(2,\RR)$ and hence $\lie{g'}$ is a proper subalgebra of $\lie{g}$.
	It is easy to see that $\lie{g}'$ satisfies the assumptions in Proposition \ref{prop:NonMinimalSubalgebra}.
	Hence $(\lie{g}, \lie{h}, \orbit{O})$ is not minimal.
\end{proof}

\begin{lemma} \label{lem:SufficientSymmetric}
	Let $\lie{k}'$ be a subalgebra of $\lieCent_{\lie{k}}(\lie{h}) \cap \lie{h}^\perp$.
	If $(\lie{g}, \lie{h} \oplus \lie{k}')$ is a symmetric pair, then $(\lie{g}, \lie{h}, \orbit{O})$ is not minimal.
\end{lemma}

\begin{proof}
	Take an involution $\sigma$ of $\lie{g}$ commuting with $\theta$ such that $\lie{g}^\sigma = \lie{h} \oplus \lie{k}'$.
	If necessary, replacing the form $(\cdot, \cdot)$ on $\lie{g}$ by $(\cdot, \cdot) + (\sigma(\cdot), \sigma(\cdot))$, we may assume that $(\lie{h} \oplus \lie{k}')^\perp = \lie{g}^{-\sigma}$.

	By Lemma \ref{lem:ExistenceWeightVector}, there exists an $\lie{a}_H$-weight vector $X \in \overline{\orbit{O}} \cap \Nilpotent(\lie{h}^\perp, H)$ with non-zero weight.
	Note that $X \in (\lie{h}\oplus \lie{k}')^\perp$ as in the proof of Lemma \ref{lem:ReductionCompactFactor}.
	Since $(\lie{g}, \lie{h} \oplus \lie{k}')$ is a symmetric pair, we have $[X, -\theta(X)] \in (\lie{h} \oplus \lie{k}') \cap \lie{g}^{-\theta} \subset \lie{h}$.
	Moreover, since $X$ is a weight vector, $[X, -\theta(X)]$ commutes with $\lie{a}_H$.
	This implies $[X, -\theta(X)] \in \lie{a}_H$.
	Hence $\set{[X, -\theta(X)], X, -\theta(X)}$ spans a subalgebra of $\lie{g}$ isomorphic to $\lieSL(2,\RR)$.
	By Lemma \ref{lem:NonMinimalX}, the assertion follows.
\end{proof}

\begin{lemma} \label{lem:SufficientMaxAbelian}
	Let $\lie{a}$ be a maximal abelian subspace of $\lie{g}^{-\theta}$ containing $\lie{a}_H$.
	Assume that there exists an $\lie{a}$-weight vector $X$ in $\overline{\orbit{O}} \cap \Nilpotent(\lie{h}^\perp, H)$ with non-zero weight on $\lie{a}_H$.
	Then $(\lie{g}, \lie{h}, \orbit{O})$ is not minimal.
\end{lemma}

\begin{proof}
	By a standard argument, $\set{[X, -\theta(X)], X, -\theta(X)}$ spans a subalgebra of $\lie{g}$ isomorphic to $\lieSL(2,\RR)$.
	Lemma \ref{lem:NonMinimalX} shows that $(\lie{g}, \lie{h}, \orbit{O})$ is not minimal.
\end{proof}

\begin{corollary} \label{cor:SufficientSameRank}
	If $\lieNorm_{\lie{g}}(\lie{h})$ has the same real rank as $\lie{g}$, then $(\lie{g}, \lie{h}, \orbit{O})$ is not minimal.
\end{corollary}

\begin{proof}
	Take a maximal abelian subspace $\lie{a}$ of $\lieNorm_{\lie{g}}(\lie{h})^{-\theta}$ containing $\lie{a}_H$.
	Then $\lie{a}$ is a maximal abelian subspace of $\lie{g}^{-\theta}$ by assumption.
	By Lemma \ref{lem:ExistenceWeightVector}, there exists an $\lie{a}$-weight vector $X \in \overline{\orbit{O}} \cap \Nilpotent(\lie{h}^\perp, H)$ with non-zero weight on $\lie{a}_H$.
	Hence the assertion follows from Lemma \ref{lem:SufficientMaxAbelian}.
\end{proof}

\begin{corollary} \label{cor:SufficientNonRoot}
	Let $X$ be an $\lie{a}_H$-weight vector in $\overline{\orbit{O}} \cap \Nilpotent(\lie{h}^\perp, H)$ with non-zero weight $\lambda$.
	Assume that the weight space $\lie{g}_\lambda$ is contained in $\lie{h}^\perp$.
	Then $(\lie{g}, \lie{h}, \orbit{O})$ is not minimal.
\end{corollary}

\begin{proof}
	Take a maximal abelian subspace $\lie{a}$ of $\lie{g}^{-\theta}$ containing $\lie{a}_H$.
	By assumption, the weight space $\lie{g}_\lambda$ is contained in $\lie{h}^\perp$ and $\lie{a}$-stable.
	Hence, by the same argument as in the proof of Lemma \ref{lem:ExistenceWeightVector}, there exists an $\lie{a}$-weight vector $X' \in \lie{g}_\lambda \cap \overline{\orbit{O}} \cap \Nilpotent(\lie{h}^\perp, H)$.
	The assertion follows from Lemma \ref{lem:SufficientMaxAbelian}.
\end{proof}

As a consequence of Corollary \ref{cor:SufficientSameRank}, we prove Theorem \ref{thm:MainTargetMinimal} when $\lie{h}$ is abelian.

\begin{theorem} \label{thm:MinimalAbelian}
	Let $(\lie{g}, \lie{h}, \orbit{O})$ be a minimal NDD tuple.
	Assume that $\lie{h}$ is abelian.
	Then $(\lie{g}, \lie{h}) \simeq (\lieSL(2,\RR), \RR)$, where $\RR$ means the subalgebra of diagonal matrices in $\lieSL(2,\RR)$.
\end{theorem}

\begin{proof}
	By Proposition \ref{prop:BasicMinimal} (4), $\lie{h}$ is non-compact.
	Hence $\lieNorm_{\lie{g}}(\lie{h})$ contains a maximal abelian subspace of $\lie{g}^{-\theta}$.
	By Corollary \ref{cor:SufficientSameRank}, we have $\lie{g} \simeq \lieSL(2,\RR)$.
	Therefore, we obtain $(\lie{g}, \lie{h}) \simeq (\lieSL(2,\RR), \RR)$.
\end{proof}

\subsection{Reduction to rank one subalgebras} \label{subsect:ReductionRankOne}

In this subsection, we shall see that the minimality of an NDD tuple $(\lie{g}, \lie{h}, \orbit{O})$ strongly restricts the structure of $\lie{g}$ and $\lie{h}$.

\begin{lemma}
	Let $(\lie{g}, \lie{h}, \orbit{O})$ be a minimal NDD tuple.
	Then $\lieNorm_{\lie{g}}(\lie{h})$ has real rank one.
	In particular, $\lieCent_{\lie{g}}(\lie{h})$ is compact if $\lie{h}$ is simple.
\end{lemma}

\begin{proof}
	Assume that $\lieNorm_{\lie{g}}(\lie{h})$ has real rank greater than one.
	Take a maximal abelian subspace $\lie{a}$ of $\lieNorm_{\lie{g}}(\lie{h})^{-\theta}$ containing $\lie{a}_H$.
	By Lemma \ref{lem:ExistenceWeightVector}, there exists an $\lie{a}$-weight vector $X \in \overline{\orbit{O}} \cap \Nilpotent(\lie{h}^\perp, H)$ with weight $\lambda$ that is non-zero on $\lie{a}_H$.
	Let $\lie{g'}$ be the smallest subalgebra of $\lie{g}$ that contains $\lie{a}$, $X$ and $\theta(X)$, and satisfies $\lie{g'} = (\lie{g'}\cap \lie{h}) \oplus (\lie{g'}\cap \lie{h}^\perp)$.
	Since $\dim(\lie{a}) > 1$, we have $0 \neq \Ker(\lambda) \subset \lieCent(\lie{g'})$.
	Since $\lie{g}$ has no center, $\lie{g'}$ is a proper subalgebra of $\lie{g}$.
	By Proposition \ref{prop:NonMinimalSubalgebra}, $(\lie{g}, \lie{h}, \orbit{O})$ is not minimal.
	This is a contradiction, and we have shown the assertion.
\end{proof}

In the following, we consider a grading of $\lie{g}$ associated with an element $h \in \lie{a}_H$.
For $i \in \RR$ and $\ad(h)$-stable subspace $W \subset \lieC{g}$, we denote by $W_i$ the eigenspace of $\ad(h)$ with eigenvalue $i$.
Then we have a grading $\lie{g} = \bigoplus_{i \in \RR} \lie{g}_i$.

\begin{definition} \label{def:typeA}
	Let $\lie{g}$ be a semisimple Lie algebra and $\lie{h}$ a $\theta$-stable subalgebra.
	We say that $(\lie{g}, \lie{h})$ is of type $A_1$ if the pair satisfies the following conditions:
	\begin{enumparen}
		\item $\lie{h}$ is simple and has restricted root system of type $A_1$,
		\item $\lieCent_{\lie{g}}(\lie{h})$ is compact,
		\item There exists a non-zero element $h \in \lie{h}^{-\theta}$ such that $\lie{g} = \lie{g}_{-1} \oplus \lie{g}_0 \oplus \lie{g}_1$, and 
		\item The projection of $\lie{g}_1$ to each simple factor of $\lie{g}$ is non-zero.
	\end{enumparen}
\end{definition}

\begin{remark}
	If $(\lie{g}, \lie{h})$ is of type $A_1$, then the projection of $\lie{h}$ to each simple factor of $\lie{g}$ is non-zero.
	This follows from the fact that, by condition (4), the projection of $h$ in condition (3) to each simple factor of $\lie{g}$ is non-zero.
\end{remark}

\begin{definition} \label{def:typeBC}
	Let $\lie{g}$ be a semisimple Lie algebra and $\lie{h}$ a $\theta$-stable subalgebra.
	We say that $(\lie{g}, \lie{h})$ is of type $BC'_1$ if the pair satisfies the following conditions.
	\begin{enumparen}
		\item $\lie{g}$ is simple.
		\item $\lie{h}$ is simple and has restricted root system of type $BC_1$.
		\item $\lieCent_{\lie{g}}(\lie{h})$ is compact.
		\item There exists a non-zero element $h \in \lie{h}^{-\theta}$ such that $\lie{g} = \lie{g}_{-2} \oplus \lie{g}_{-1} \oplus \lie{g}_0 \oplus \lie{g}_1 \oplus \lie{g}_2$.
	\end{enumparen}
	We say that $(\lie{g}, \lie{h})$ is of type $BC_1$ if it is of type $BC'_1$ and satisfies the following additional condition:
	\begin{enumparen} \setcounter{enumi}{4}
		\item $\lie{g}_2 = \lie{h}_2$.
	\end{enumparen}
\end{definition}

We shall show that, if $(\lie{g}, \lie{h}, \orbit{O})$ is a minimal NDD tuple with non-abelian $\lie{h}$, then $(\lie{g}, \lie{h})$ is of type $A_1$ or $BC'_1$.
We divide the proof into several steps.

Let $(\lie{g}, \lie{h}, \orbit{O})$ be a minimal NDD tuple with non-abelian $\lie{h}$.
By Proposition \ref{prop:BasicMinimal}, $\lie{h}$ is simple and has real rank one.

\begin{lemma} \label{lem:ReductionLemmaStep1}
	There exists an element $h \in \lie{h}^{-\theta}$ satisfying the following conditions.
	\begin{enumparen}
		\item $\lie{g} = \bigoplus_{i \in \ZZ} \lie{g}_i$.
		\item $\lie{h}_1 \neq 0$.
		\item For a positive integer $c$, one has $\lie{g}_c\cap \overline{\orbit{O}} \cap \lie{h}^\perp \neq \set{0}$ if and only if $c = 1$.
	\end{enumparen}
\end{lemma}

\begin{proof}
	Take a non-zero element $h \in \lie{a}_H$ such that $\lie{h} = \bigoplus_{-2\leq i \leq 2} \lie{h}_i$ and $\lie{h}_1 \neq 0$.
	By the representation theory of $\lieSL(2,\RR)$, every eigenvalue of $\ad_{\lie{g}}(h)$ is in $\frac{1}{2}\ZZ$,
	and hence we have $\lie{g} = \bigoplus_{i \in \ZZ} \lie{g}_{i/2}$.
	By Lemma \ref{lem:ExistenceWeightVector}, there exists a vector $X \in \overline{\orbit{O}} \cap \lie{h}^\perp \cap \lie{g}_c$ for some positive $c \in \frac{1}{2}\ZZ$.
	We take $X$ so that $c$ is maximal.

	If $\lie{h}_c = 0$, then $(\lie{g}, \lie{h}, \orbit{O})$ is not minimal by Corollary \ref{cor:SufficientNonRoot}.
	Hence we may assume that $\lie{h}_c \neq 0$, and then $c$ is either $1$ or $2$.
	By Corollary \ref{cor:MinimalGeneratedBy}, we have $\lie{g} = \bigoplus_{i \in \ZZ} \lie{g}_{ci}$.
	If $c = 2$, we have $\lie{g}_1 = 0$.
	This contradicts $\lie{h}_1 \neq 0$, and hence $c = 1$.
	Therefore we have shown the assertion.
\end{proof}

Fix an element $h \in \lie{h}^{-\theta}$ satisfying the conditions in Lemma \ref{lem:ReductionLemmaStep1}.

\begin{lemma} \label{lem:ReductionLemmaStep2}
	For any $d \geq 2$, if $\lie{h}_d = 0$, then one has $\lie{g}_d = 0$.
\end{lemma}

\begin{proof}
	Take a non-zero element $X$ in $\lie{g}_1\cap \overline{\orbit{O}} \cap \lie{h}^\perp$ by Lemma \ref{lem:ReductionLemmaStep1} (3).
	Let $d$ be the maximal positive integer such that $\lie{g}_d \neq 0$ and $p\colon \lie{g}\rightarrow \lie{g}_d$ be the projection with respect to the direct sum decomposition $\lie{g} = \bigoplus_{i \in \ZZ} \lie{g}_i$.
	Since $X$ does not belong to any proper ideal of $\lie{g}$ by Proposition \ref{prop:BasicMinimal} (6), $\Ad(G)X$ spans $\lie{g}$.
	Hence there exists $g \in G$ such that $p(\Ad(g)X) \neq 0$.
	Then we have
	\begin{align*}
		\overline{\orbit{O}} \ni \lim_{t \to \infty} e^{-dt} \Ad(\exp(th))\Ad(g)X = p(\Ad(g)X) \in \lie{g}_d.
	\end{align*}
	This shows that $\lie{g}_d \cap \overline{\orbit{O}} \neq \set{0}$.
	
	If $\lie{h}_d = 0$, then we have $\lie{g}_d \subset \lie{h}^\perp$, which contradicts the minimality of $(\lie{g}, \lie{h}, \orbit{O})$ by Corollary \ref{cor:SufficientNonRoot}.
	Since $\lie{h} = \lie{h}_{-2} \oplus \lie{h}_{-1} \oplus \lie{h}_0 \oplus \lie{h}_1 \oplus \lie{h}_2$, the assertion follows.
\end{proof}

\begin{lemma} \label{lem:ReductionLemmaStep3}
	Suppose that $\lie{g}_2 \neq 0$.
	Let $X$ be a non-zero element in $\overline{\orbit{O}} \cap \lie{g}_1$.
	Then one has
	\begin{align*}
		\lie{g}_2 = [X, \lie{g}_1] \oplus \lie{g}^{\theta(X)}_2 =  \lie{h}_2 \oplus \lie{g}^{\theta(X)}_2.
	\end{align*}
\end{lemma}

\begin{proof}
	Note that $[X, \lie{g}]^\perp = \lie{g}^X$.
	This implies that the orthogonal complement of $[X, \lie{g}_1]$ in $\lie{g}_2$ with respect to the inner product $-(\cdot, \theta(\cdot))$ is $\lie{g}_2^{\theta(X)}$, and we have $[X, \lie{g}_1] \oplus \lie{g}_2^{\theta(X)} = \lie{g}_2$.
	Hence it suffices to show that $\dim([X, \lie{g}_1]) \leq \dim(\lie{h}_2)$ and $\lie{h}_2 \cap \lie{g}_2^{\theta(X)} = 0$.

	For any $Y \in \lie{g}_1$, we have
	\begin{align*}
		\overline{\orbit{O}} &\ni \lim_{t \to \infty} e^{-2t} \Ad(\exp(th))\Ad(\exp(Y))X \\
		&= \lim_{t \to \infty} e^{-2t} \Ad(\exp(th))(X + [Y, X]) = [Y, X].
	\end{align*}
	This implies that $[X, \lie{g}_1] \subset \overline{\orbit{O}}$.
	If $[X, \lie{g}_1] \cap \lie{h}^\perp \neq 0$, then we have $\overline{\orbit{O}}\cap \lie{g}_2 \cap \lie{h}^\perp \neq \set{0}$, which contradicts Lemma \ref{lem:ReductionLemmaStep1} (3).
	Hence we have $[X, \lie{g}_1] \cap \lie{h}^\perp = 0$ and $\dim([X, \lie{g}_1]) \leq \dim(\lie{h}_2)$.

	Since $\lie{h}$ is a simple Lie algebra of real rank one, for any non-zero $x \in \lie{h}_2$, the subalgebra generated by $x$ and $\theta(x)$ is isomorphic to $\lieSL(2,\RR)$.
	This implies that $[\theta(X), x] \neq 0$ and hence $\lie{h}_2 \cap \lie{g}_2^{\theta(X)} = 0$.
	We have shown the lemma.
\end{proof}

\begin{corollary} \label{cor:ReductionLemmaStep3}
	One has $\overline{\orbit{O}} \cap \lie{g}_1 \cap \lie{h}^\perp\subset \lie{g}^{\lie{h}_1}$.
\end{corollary}

\begin{proof}
	If $\lie{g}_2 = 0$, then the assertion is trivial.
	Assume that $\lie{g}_2 \neq 0$.
	In the proof of Lemma \ref{lem:ReductionLemmaStep3}, we have shown that $[X, \lie{g}_1] \cap \lie{h}^\perp = 0$ for any $X \in \overline{\orbit{O}} \cap \lie{g}_1$.
	This implies that $[X, \lie{h}_1] = 0$ for any $X \in \overline{\orbit{O}} \cap \lie{g}_1 \cap \lie{h}^\perp$.
\end{proof}

\begin{lemma} \label{lem:ReductionLemmaStep4}
	If $\lie{g}_2 \neq 0$, then $\lie{g}$ is simple.
\end{lemma}

\begin{proof}
	Assume that $\lie{g}$ is not simple and $\lie{g} = \lie{g}' \oplus \lie{g}''$ is a direct sum of two proper non-zero ideals.
	Note that $\overline{\orbit{O}}$ is written as $\overline{\orbit{O}} = \overline{\orbit{O}'} + \overline{\orbit{O}''}$ for some nilpotent orbits $\orbit{O}'$ and $\orbit{O}''$ in $\lie{g}'$ and $\lie{g}''$, respectively.
	Take a non-zero element $X \in \overline{\orbit{O}} \cap \lie{g}_1 \cap \lie{h}^\perp$ and write $X = X' + X''$ with $X' \in \lie{g}'$ and $X'' \in \lie{g}''$.
	Then we have $X', X'' \in \overline{\orbit{O}}$.

	By Proposition \ref{prop:BasicMinimal} (6), $X'$ and $X''$ are non-zero.
	By Lemma \ref{lem:ReductionLemmaStep3}, we have
	\begin{align*}
		\dim([X, \lie{g}_1]) = \dim([X', \lie{g}'_1]) = \dim([X'', \lie{g}''_1]) = \dim(\lie{h}_2).
	\end{align*}
	On the other hand, we have $[X, \lie{g}_1] = [X', \lie{g}'_1] + [X'', \lie{g}''_1]$.
	This contradicts the above equality, and we have shown the lemma.
\end{proof}

Summarizing Propositions \ref{prop:BasicMinimal} and \ref{prop:BasicMinimalH}, and Lemmas \ref{lem:ReductionLemmaStep1}, \ref{lem:ReductionLemmaStep2}, \ref{lem:ReductionLemmaStep3} and \ref{lem:ReductionLemmaStep4}, we have the following theorem.

\begin{theorem} \label{thm:ReductionLemma}
	If $(\lie{g}, \lie{h}, \orbit{O})$ is a minimal NDD tuple with non-abelian $\lie{h}$, then $(\lie{g}, \lie{h})$ is of type $A_1$ or $BC'_1$.
\end{theorem}

\subsection{Reduction by a \texorpdfstring{$\lie{q}$}{q}-orthogonal system} \label{subsect:ReductionStronglyOrthogonal}

The proof of Theorem \ref{thm:MainTargetMinimal} relies on classification results and certain explicit computations.
In this subsection, we prove a proposition that helps reduce the amount of computation.

Let $(\lie{g}, \lie{h}, \orbit{O})$ be an NDD tuple of type $A_1$ or $BC'_1$.
Fix a non-zero element $h \in \lie{h}^{-\theta}$ as in Definitions \ref{def:typeA} and \ref{def:typeBC}.
Set $\lie{a}_H \coloneq \RR h$.
Then we have a grading $\lie{g} = \bigoplus_{i=-2}^2 \lie{g}_i$ associated with $h$.
Note that $\lie{g}_2$ is zero in the case of type $A_1$.

Let $\sigma$ be the involution on $\lie{g}$ such that $\lie{g}^\sigma = \lie{g}_{-2} \oplus \lie{g}_0 \oplus \lie{g}_2$.
Fix a maximal abelian subspace $\lie{a}$ of $\lie{g}^{-\theta}$ containing $\lie{a}_H$.
Then $\lie{a} \subset \lie{g}^{-\theta}_0$ and $\lie{g}_1$ is $\lie{a}$-stable.
Take roots $\set{\lambda_1, \ldots, \lambda_r}$ in $\Delta(\lie{g}_1, \lie{a})$
and root vectors $x_i \in \lie{g}_{\lambda_i}$ such that
\begin{align}
	[x_i, x_j] = 0, \quad [x_i, \theta(x_j)] = 0 \quad (\forall i \neq j), \label{eqn:QOrthogonal}
\end{align}
and $r$ is as large as possible.
Note that $\set{\lambda_1, \ldots, \lambda_r}$ may not be a set of strongly orthogonal roots in general.
The set $\set{x_1, \ldots, x_r}$ is called a $\lie{q}$-orthogonal system in \cite{Ma79_orbits}.

\begin{fact}[T.\ Matsuki {\cite[Theorem 2]{Ma79_orbits}}] \label{fact:Matsuki}
	$\bigoplus_{i=1}^r \RR (x_i - \theta(x_i))$ is a maximal abelian subspace in $\lie{g}^{-\sigma, -\theta} = (\lie{g}_{-1}\oplus \lie{g}_1) \cap \lie{g}^{-\theta}$.
\end{fact}

We shall construct a suitable candidate for the subalgebra $\lie{g}'$ of $\lie{g}$ in Proposition \ref{prop:NonMinimalSubalgebra} using the $\lie{q}$-orthogonal system $\set{x_1, \ldots, x_r}$.
Set $\lie{c} \coloneq \bigoplus_{i=1}^r \RR x_i$.
Let $\lie{g'}$ be the subalgebra of $\lie{g}$ generated by $\lie{c}$, $\theta(\lie{c})$ and $\lie{a}$.

\begin{lemma} \label{lem:ReductionStronglyOrthogonalBasic}
	The subalgebra $\lie{g'}$ satisfies the following conditions.
	\begin{enumparen}
		\item $\lie{g'}$ is $\theta$-stable.
		\item $\lie{g'} = \lie{a} \oplus \lie{c} \oplus \theta(\lie{c})$.
		\item $\lie{g'}$ is isomorphic to a direct sum of some copies of $\lieSL(2,\RR)$ and an abelian ideal.
	\end{enumparen}
\end{lemma}

\begin{proof}
	The assertions are clear from the definition of $\lie{g}'$ and the relations \eqref{eqn:QOrthogonal}.
\end{proof}

\begin{lemma} \label{lem:StronglyOrthogonalHCompatible}
	Assume that $\lie{c} = (\lie{c} \cap \lie{h}) \oplus (\lie{c} \cap \lie{h}^\perp)$.
	Then one has $\lie{g'} = (\lie{g'} \cap \lie{h}) \oplus (\lie{g'} \cap \lie{h}^\perp)$.
\end{lemma}

\begin{proof}
	By Lemma \ref{lem:ReductionStronglyOrthogonalBasic} (2), we have
	\begin{align*}
		\lie{g'} &= \lie{a} \oplus \lie{c} \oplus \theta(\lie{c}), \\
		\lie{g'} \cap \lie{h} &\supset \lie{a}_H \oplus (\lie{c} \cap \lie{h}) \oplus (\theta(\lie{c}) \cap \lie{h}), \\
		\lie{g'} \cap \lie{h}^\perp &\supset (\lie{a}_H^\perp \cap \lie{a}) \oplus(\lie{c} \cap \lie{h}^\perp) \oplus (\theta(\lie{c}) \cap \lie{h}^\perp).
	\end{align*}
	From this and the assumption $\lie{c} = (\lie{c} \cap \lie{h}) \oplus (\lie{c} \cap \lie{h}^\perp)$, we obtain $\lie{g'} = (\lie{g'} \cap \lie{h}) \oplus (\lie{g'} \cap \lie{h}^\perp)$.
\end{proof}

Hereafter, assume that $\lie{g}$ has a simple ideal not isomorphic to $\lieSL(2,\RR)$.
This assumption is satisfied if $\lie{h}$ is not isomorphic to $\lieSL(2,\RR)$ since $\lie{h}$ is not contained in any proper ideal of $\lie{g}$.
See Definition \ref{def:typeA} (4).

\begin{lemma} \label{lem:ReductionRealStronglyOrthogonal}
	Assume that $\lie{c} \cap \lie{h}^\perp \cap \overline{\orbit{O}} \neq \set{0}$ and $\lie{c} = (\lie{c} \cap \lie{h}) \oplus (\lie{c} \cap \lie{h}^\perp)$.
	Then $(\lie{g}, \lie{h}, \orbit{O})$ is not minimal.
\end{lemma}

\begin{proof}
	By the assumption that $\lie{g}$ has a simple factor not isomorphic to $\lieSL(2,\RR)$, we have $\lie{g} \neq \lie{g'}$.
	By Lemma \ref{lem:StronglyOrthogonalHCompatible}, we have $\lie{g'} = (\lie{g'} \cap \lie{h}) \oplus (\lie{g'} \cap \lie{h}^\perp)$.
	By definition, any non-zero vector in $\lie{c} \subset \lie{g}_1$ is an $\lie{a}_H$-weight vector.
	By Proposition \ref{prop:NonMinimalSubalgebra}, $(\lie{g}, \lie{h}, \orbit{O})$ is not minimal.
\end{proof}

A direct verification of the conditions of Lemma \ref{lem:ReductionRealStronglyOrthogonal} still requires explicit computations.
The following proposition has a narrower range of applicability, but it allows us to check the conditions of Lemma \ref{lem:ReductionRealStronglyOrthogonal} using only routine computations.

\begin{proposition} \label{prop:ReductionRealStronglyOrthogonal}
	Assume the following conditions.
	\begin{enumparen}
		\item There exist two $\theta$-stable ideals $\lie{i}$ and $\lie{j}$ of $\lie{g}^{\theta\sigma}$ such that
		\begin{align*}
			&\lie{g}^{\theta\sigma} = \lie{i} \oplus \lie{j}, \\
			&\lie{g}^{-\sigma, -\theta} \cap \lie{h} = \lie{i}^{-\theta}, \\
			&\lie{g}^{-\sigma, -\theta} \cap \lie{h}^\perp = \lie{j}^{-\theta}.
		\end{align*}
		\item The restriction to $\lie{g}_0\cap \lie{k}$ of the orthogonal projection from $\lie{g}$ to $\lie{j}\cap \lie{k}$ is surjective.
	\end{enumparen}
	Then $(\lie{g}, \lie{h}, \Ad(G)X)$ is not minimal for any $0\neq X \in \lie{g}_1 \cap \lie{h}^\perp$.
\end{proposition}

\begin{proof}
	We shall verify the two assumptions of Lemma \ref{lem:ReductionRealStronglyOrthogonal}.
	Define a map $\varphi \colon \lie{g}_1 \rightarrow \lie{g}^{-\sigma, -\theta} = (\lie{g}_{-1} \oplus \lie{g}_{1}) \cap \lie{g}^{-\theta}$ by $\varphi(X) = X - \theta(X)$.
	Then $\varphi$ is bijective and $(\lie{g}_0\cap \lie{k})$-equivariant, and we have
	\begin{align*}
		\varphi(\lie{g}_1 \cap \lie{h}) &= \lie{g}^{-\sigma, -\theta} \cap \lie{h} = \lie{i}^{-\theta}, \\
		\varphi(\lie{g}_1 \cap \lie{h}^\perp) &= \lie{g}^{-\sigma, -\theta} \cap \lie{h}^\perp = \lie{j}^{-\theta}.
	\end{align*}
	Since $\lie{g}^{\theta\sigma} = \lie{i} \oplus \lie{j}$, any maximal abelian subspace in $\lie{g}^{-\theta, -\sigma}$ is of the form $\lie{c}' \oplus \lie{c}''$ with $\lie{c}' \subset \lie{i}$ and $\lie{c}'' \subset \lie{j}$.
	Since $\varphi(\lie{c})$ is a maximal abelian subspace in $\lie{g}^{-\sigma, -\theta}$ by Fact \ref{fact:Matsuki}, we obtain 
	\begin{align*}
		\lie{c} = (\lie{c} \cap \lie{h}) \oplus (\lie{c} \cap \lie{h}^\perp).
	\end{align*}
	This is the second assumption in Lemma \ref{lem:ReductionRealStronglyOrthogonal}.

	Let $0\neq X \in \lie{g}_1 \cap \lie{h}^\perp$.
	Then we have $\varphi(X) \in \lie{j}^{-\theta}$.
	Let $K_0$ be the analytic subgroup of $G$ with the Lie algebra $\lie{g}_0\cap \lie{k}$.
	Since $\varphi(\lie{c}) \cap \lie{h}^\perp$ is a maximal abelian subspace in $\lie{j}^{-\theta}$, assumption (2) implies that there exists $k \in K_0$ such that $\Ad(k)\varphi(X) \in \varphi(\lie{c}) \cap \lie{h}^\perp$.
	Hence we obtain
	\begin{align*}
		0\neq \Ad(k)X \in \lie{c} \cap \lie{h}^\perp \cap \overline{\Ad(G)X}.
	\end{align*}
	This shows the first assumption in Lemma \ref{lem:ReductionRealStronglyOrthogonal}, that is, $\lie{c} \cap \lie{h}^\perp \cap \overline{\Ad(G)X} \neq \set{0}$.
	We have shown the proposition.
\end{proof}

\section{\texorpdfstring{$\lieSO(1, k)$}{so(1,k)}} \label{section:TypeA1}

In this section, we deal with the $A_1$ case defined in Definition \ref{def:typeA}.
We use the theory of Jordan algebras and $3$-graded Lie algebras to study the structure of $\lie{g}_1$.

\subsection{Jordan algebras and graded Lie algebras} \label{subsect:JordanAlgebra}

We shall study the case of type $A_1$ defined in Definition \ref{def:typeA}.
To do this, we first recall basic properties of Jordan algebras and $3$-graded Lie algebras.
We refer the reader to \cite{FaKo94} and \cite{HiKoMo14} for details of the relation between Jordan algebras and $3$-graded Lie algebras.
See also the references therein.

Let $\lie{g} = \lie{g}_1 \oplus \lie{g}_0 \oplus \lie{g}_{-1}$ be a graded simple Lie algebra with Cartan involution $\theta$ such that $\theta(\lie{g}_i) = \lie{g}_{-i}$, and $h \in \lie{g}_0$ the characteristic element of the grading.
We equip $\lie{g}_1$ with a Jordan triple product by $\jprod{x, y, z} \coloneq \frac{1}{2}[[x, -\theta(y)], z]$.
The product satisfies the identities:
\begin{align*}
	\jprod{x, y, z} &= \jprod{z, y, x}, \\
	\jprod{x, y, \jprod{u, v, w}} &= \jprod{\jprod{x, y, u}, v, w} - \jprod{u, \jprod{y, x, v}, w} + \jprod{u, v, \jprod{x, y, w}}
\end{align*}
for any $x, y, z, u, v, w \in \lie{g}_1$.

We say that $x \in \lie{g}_1$ is a \define{tripotent} if $\jprod{x, x, x} = x$.
In this paper, by a tripotent we always mean a nonzero tripotent.
The following proposition is clear from the definition of the Jordan triple product.

\begin{proposition} \label{prop:TripotentAndSL2Triple}
	Let $x \in \lie{g}_1$.
	Then $\set{[x, -\theta(x)], x, -\theta(x)}$ forms an $\lieSL_2$-triple if and only if $x$ is a tripotent.
\end{proposition}

Assume that there exists an element $u \in \lie{g}_1$ such that $[u, -\theta(u)] = 2h$.
Such an element is called a \define{unitary} element.
Then we have $\jprod{u, u, x} = x$ for any $x \in \lie{g}_1$ and, in particular, $u$ is a tripotent.
This implies that $\lie{g}_1$ has a Jordan algebra structure with the product $x \circ y = \jprod{x, u, y}$ and unit $u$, that is, the product $\circ$ satisfies the identities:
\begin{align*}
	x\circ y = y\circ x, \quad x\circ (x^2 \circ y) = x^2 \circ (x\circ y)
\end{align*}
for any $x, y \in \lie{g}_1$, where $x^2 = x\circ x$.
Note that the Jordan algebra structure depends on the choice of $u$.
When we need to specify the unit, we denote by $(\lie{g}_1, u)$ the Jordan algebra $\lie{g}_1$ with the unit $u$.
For a while, we fix the Jordan algebra structure on $\lie{g}_1$.

We equip $\lie{g}_1$ with a trace and a symmetric bilinear form by
\begin{align*}
	\tr(x) \coloneq -(x, \theta(u)), \tau(x, y) \coloneq \tr(x \circ y).
\end{align*}
Note that $\tau(\cdot, \cdot)$ and $\tr$ are scalar multiples of the standard bilinear form and trace, respectively, on the Jordan algebra $\lie{g}_1$.
The scalar multiples are not important in this paper, so we do not specify them.
Define $\vartheta(x) \coloneq \jprod{u, x, u}$ for $x \in \lie{g}_1$.
Then $\vartheta$ is an involution of the Jordan algebra $\lie{g}_1$ such that
\begin{align*}
	\tau(x, \vartheta(y)) = \tau(\vartheta(x), y) = -(x, \theta(y)) \quad (\forall x, y \in \lie{g}_1).
\end{align*}
Hence $\tau(\cdot, \vartheta(\cdot))$ is an inner product on $\lie{g}_1$.

\begin{proposition} \label{prop:CartanInvOrthogonal}
	One has $\lie{g}_1^{\vartheta} \perp \lie{g}_1^{-\vartheta}$ with respect to the inner product $\tau(\cdot, \vartheta(\cdot)) = -(\cdot, \theta(\cdot))$.
\end{proposition}

We call $\vartheta$ the \define{Cartan involution} of the Jordan algebra $\lie{g}_1$.
The triple product $\jprod{\cdot, \cdot, \cdot}$ is reconstructed from the Jordan product $\circ$ and the Cartan involution $\vartheta$ by
\begin{align}
	\jprod{x, y, z} = (x \circ \vartheta(y)) \circ z + (z \circ \vartheta(y)) \circ x - (x \circ z) \circ \vartheta(y). \label{eqn:TripleProductFromJordanProduct}
\end{align}

To study orbit decompositions in $\lie{g}_1$, we use Jordan frames.
An element $x \in \lie{g}_1$ is called an \define{idempotent} if $x \circ x = x$ and $x \neq 0$.
For two idempotents $x, y \in \lie{g}_1$, we say that $x$ and $y$ are \define{orthogonal} if $x \circ y = 0$.
For an idempotent $x$, if there exist no orthogonal idempotents $x_1, x_2$ with $x = x_1 + x_2$, then $x$ is called \define{primitive}.
A family of mutually orthogonal primitive idempotents $\set{u_1, \ldots, u_r}$ in $\lie{g}_1^{\vartheta}$ is called a \define{Jordan frame} if $u = \sum_{i=1}^r u_i$ and $r$ is as large as possible.
We rephrase the notions in terms of the Lie algebra $\lie{g}$.

\begin{proposition} \label{prop:IdempotentAndOrthogonal}
	Let $x, y \in \lie{g}_1^\vartheta$ be idempotents.
	\begin{enumparen}
		\item $\set{[x, -\theta(x)], x, -\theta(x)}$ forms an $\lieSL_2$-triple.
		\item $x$ and $y$ are orthogonal if and only if $[x, -\theta(y)] = [y, -\theta(x)] = 0$.
	\end{enumparen}
\end{proposition}

\begin{proof}
	(1) This follows from Proposition \ref{prop:TripotentAndSL2Triple} and \eqref{eqn:TripleProductFromJordanProduct}.

	(2) Assume that $x$ and $y$ are orthogonal.
	For $a \in \lie{g}_1$, let $L(a)$ denote the linear map $a\circ (\cdot)$.
	Then we have the Jordan identity $[L(a^2), L(a)] = 0$.
	For any $a, b\in \RR$, expanding the left-hand side of the identity $[L((ax+by)^2), L(ax + by)] = 0$,
	we have
	\begin{align*}
		a^2b[L(x), L(y)] + ab^2[L(y), L(x)] = 0.
	\end{align*}
	This implies that $[L(x), L(y)] = 0$.
	By \eqref{eqn:TripleProductFromJordanProduct}, we have
	\begin{align*}
		\jprod{x, y, z} = (z \circ y) \circ x - (x \circ z) \circ y = L(x)L(y)z - L(y)L(x)z = 0
	\end{align*}
	for any $z \in \lie{g}_1$.
	This shows that $[x, -\theta(y)] =  [y, -\theta(x)] = 0$.

	To show the converse, assume that $[x, -\theta(y)] = [y, -\theta(x)] = 0$.
	By \eqref{eqn:TripleProductFromJordanProduct}, we have
	\begin{align*}
		x\circ y = (x \circ y) \circ u = \frac{1}{2}\left(\jprod{x, y, u} + \jprod{y, x, u}\right) = 0.
	\end{align*}
	We have shown the assertion.
\end{proof}

The following fact is a consequence of the Peirce decomposition \cite[Theorem IV.2.1]{FaKo94} of the Jordan algebra $\lie{g}_1$ with respect to a Jordan frame.
Let $\sigma$ be the involution of $\lie{g}$ such that $\lie{g}^{\sigma} = \lie{g}_0$.
Let $K_0$ be the analytic subgroup of $G$ with the Lie algebra $\lie{g}^\theta_0 = \lie{k}^\sigma$.

\begin{fact} \label{fact:JordanFrame}
	Let $\set{u_1, \ldots, u_r}$ be a Jordan frame in $\lie{g}_1^{\vartheta}$.
	Set $h_i \coloneq [u_i, -\theta(u_i)]$ for $i = 1, \ldots, r$.
	\begin{enumparen}
		\item $[u_i, -\theta(u_j)] = 0$ for any $i \neq j$.
		\item $[h_i, h_j] = 0$ for any $i, j$.
		\item $[h_i, u_j] = 2\delta_{ij} u_j$ for any $i, j$.
		\item $\sum_{i=1}^r h_i = 2h$.
		\item $\spn{\RR}{u_1 - \theta(u_1), \ldots, u_{r} - \theta(u_{r})}$ is a maximal abelian subspace in $\lie{g}^{-\theta, -\sigma} = (\lie{g}_1 + \lie{g}_{-1})\cap \lie{g}^{-\theta}$.
		\item There exists a maximal abelian subspace $\lie{a}$ of $\lie{g}_0^{-\theta}$ such that $u_i$ are root vectors and the roots are strongly orthogonal.
	\end{enumparen}
\end{fact}

\begin{proposition} \label{prop:K0Orbit}
	Let $\set{u_1, \ldots, u_r}$ be a Jordan frame in $\lie{g}_1^{\vartheta}$, and set $\lie{c}\coloneq \spn{\RR}{u_1, \ldots, u_r}$.
	One has $\Ad(K_0)\lie{c} = \lie{g}_1$.
\end{proposition}

\begin{proof}
	Define $\varphi\colon \lie{g}_1 \rightarrow \lie{g}^{-\theta, -\sigma}$ by $\varphi(x) = x -\theta(x)$.
	Then $\varphi$ is a $K_0$-equivariant linear isomorphism.
	By Fact \ref{fact:JordanFrame} (5), $\varphi(\lie{c})$ is a maximal abelian subspace in $\lie{g}^{-\theta, -\sigma}$.
	This implies that $\Ad(K_0)\varphi(\lie{c}) = \lie{g}^{-\theta, -\sigma} = \varphi(\lie{g}_1)$.
	Since $\varphi$ is bijective, we obtain $\Ad(K_0)\lie{c} = \lie{g}_1$.
\end{proof}

\begin{proposition} \label{prop:TripotentAndJordanFrame}
	Let $\set{u_1, \ldots, u_r}$ be a Jordan frame in $\lie{g}_1^{\vartheta}$.
	Let $v \in \lie{g}_1$ be a tripotent.
	Then there exist an element $k \in K_0$ and $\varepsilon_1, \ldots, \varepsilon_r \in \{-1, 0, 1\}$ such that $\Ad(k)v = \sum_{i=1}^r \varepsilon_i u_i$.
	Moreover, if $[v, -\theta(v)] = 2h$, then $\varepsilon_i \in \{-1, 1\}$ for any $i$.
\end{proposition}

\begin{proof}
	By Proposition \ref{prop:K0Orbit}, there exist $k \in K_0$ and $\varepsilon_1, \ldots, \varepsilon_r \in \RR$ such that
	\begin{align*}
		\Ad(k)v = \sum_{i=1}^r \varepsilon_i u_i.
	\end{align*}
	Since $\jprod{v, v, v} = v$, we have $\sum_{i=1}^r \varepsilon_i^3 u_i = \sum_{i=1}^r \varepsilon_i u_i$.
	This implies that $\varepsilon_i \in \{-1, 0, 1\}$ for any $i$.
	The last assertion follows from a similar calculation.
\end{proof}

Let $G_0$ be the analytic subgroup of $G$ with the Lie algebra $\lie{g}_0$.
We shall review the orbit decomposition of $\lie{g}_1$ under the action of $G_0$.
We refer the reader to \cite{Ka98} for details.

Assume that $\lie{g}$ is simple.
By Table \ref{table:JordanAlgebra} (see \cite[Table II]{Ka98} and \cite[Table 4]{HiKoMo14}), there exists $u \in \lie{g}_1$ such that $\lie{g}_1^\vartheta$ is a simple Jordan algebra.
Note that there may be several possible choices of $u$ for the Jordan triple system $\RR^{p,q}$ up to the $K_0$-action, i.e., $\RR^{p,q} \simeq \RR^{q,p}$ as Jordan triple systems.
In the table, we fix the unit vector $e_1$ as $u$.
We treat the case of $\lie{g} \simeq \lieSO(p, q)$ in Subsection \ref{subsect:JordanAlgebraSO(p,q)}.

\begin{table}[htbp]
	\centering
	\small
	\renewcommand{\arraystretch}{1.15}
	\begin{tabular}{c|c|c}
		$\lie{g}$ & $\Delta(\lie{g}^{\theta\sigma}, \lie{a}')$ & $\lie{g}_1$ \\
		\hline
		$\lieSL(2, \RR)$ & $A_0$ & $\RR$ \\
		$\lieSp(r, \RR)$ ($r \geq 3$) & $A_{r-1}$ & $\operatorname{Sym}(r, \RR)$ \\
		$\lieSU(r, r)$ ($r \geq 3$) & $A_{r-1}$ & $\operatorname{Herm}(r, \CC)$ \\
		$\lieSO^*(4r)$ ($r \geq 3$) & $A_{r-1}$ & $\operatorname{Herm}(r, \HH)$ \\
		$\lieSO(2, q+1)$ ($q \geq 2$) & $A_1$ & $\RR^{1,q}$ \\
		$\lieE_{7(-25)}$ & $A_2$ & $\operatorname{Herm}(3, \mathbb{O})$ \\ \hline
		$\lieSL(2r, \RR)$ ($r \geq 3$) & $D_r$ & $M(r, \RR)$ \\
		$\lieSO(2r, 2r)$ ($r \geq 3$) & $D_r$ & $\operatorname{Skew}(2r, \RR)$ \\
		$\lieSO(p+1, q+1)$ ($p\geq 2,q \geq 2$) & $D_2$ & $\RR^{p,q}$ \\
		$\lieE_{7(7)}$ & $D_3$ & $\operatorname{Herm}(3, \mathbb{O}_s)$ \\ \hline
		$\lieSp(r, \CC)$ ($r \geq 3$) & $C_r$ & $\operatorname{Sym}(r, \CC)$ \\
		$\lieSL(2r, \CC)$ ($r \geq 3$) & $C_r$ & $M(r, \CC)$ \\
		$\lieSO(4r, \CC)$ ($r \geq 3$) & $C_r$ & $\operatorname{Skew}(2r, \CC)$ \\
		$\lieSO(n+2, \CC)$ ($n \geq 3$) & $C_2$ & $\CC^n$ \\
		$\lieE_{7(\CC)}$ & $C_3$ & $\operatorname{Herm}(3, \mathbb{O})_{\CC}$ \\ \hline
		$\lieSp(r, r)$ ($r \geq 2$) & $C_r$ & $\operatorname{Sym}(2r, \CC) \cap M(r, \HH)$ \\
		$\lieSL(2r, \HH)$ ($r \geq 2$) & $C_r$ & $M(r, \HH)$ \\
		$\lieSO(p+1, 1)$ ($p \geq 2$) & $C_1$ & $\RR^{p,0}$
	\end{tabular}
	\caption{Classification of $3$-graded simple Lie algebras with unitary element $u$ and corresponding Jordan algebras}
	\label{table:JordanAlgebra}
\end{table}

Fix a Jordan frame $\set{u_1, \ldots, u_r}$ in $\lie{g}_1^\vartheta$.
Set
\begin{align}
	o_{p, q} &\coloneq \sum_{i=1}^p u_i - \sum_{j=p+1}^{p+q} u_j \quad (p, q \geq 0, p+q \leq r), \nonumber \\
	\lie{c} &\coloneq \spn{\RR}{u_1, \ldots, u_r}, \nonumber \\
	\lie{a'} &\coloneq \spn{\RR}{u_1 - \theta(u_1), \ldots, u_{r} - \theta(u_{r})}. \label{eqn:DefineOCA}
\end{align}
By Fact \ref{fact:JordanFrame} (5), $\lie{a'}$ is a maximal abelian subspace in $\lie{g}^{-\theta, -\sigma}$.

\begin{fact}[{\cite[Table II]{Ka98}}] \label{fact:RootSystemJordan}
	The root system $\Delta(\lie{g}^{\theta\sigma}, \lie{a'})$ is either of type $A_{r-1}$, $C_r$ or $D_r$.
	If the root system is of type $A_{r-1}$, then $\lie{g}^{\theta\sigma}$ has one-dimensional center $\RR(u - \theta(u))$.
\end{fact}

\begin{remark}
	When $\lie{g} \simeq \lieSL(2,\RR)$, the root system has no roots.
	We regard the root system as of type $A_0$ in this case.
	When $\lie{g} \simeq \lieSO(1, k)$ with $k \geq 3$, the root system is of type $A_1$.
	However, we regard the root system as of type $C_1$ for the results below.
\end{remark}

\begin{proposition} \label{prop:EuclideanJordanRoot}
	$\vartheta = \id$ if and only if $\Delta(\lie{g}^{\theta\sigma}, \lie{a'})$ is of type $A_{r-1}$.
\end{proposition}

\begin{remark}
	If the equivalent conditions in Proposition \ref{prop:EuclideanJordanRoot} hold, then $\lie{g}_1$ is called \define{Euclidean}.
\end{remark}

\begin{proof}
	Assume that $\Delta(\lie{g}^{\theta\sigma}, \lie{a'})$ is of type $A_{r-1}$.
	Then $u$ is a $K_0$-fixed vector and hence $\vartheta$ is $K_0$-equivariant.
	Let $x \in \lie{g}_1$.
	We shall show that $\vartheta(x) = \jprod{u, x, u} = x$.
	By Proposition \ref{prop:K0Orbit}, we may assume that $x \in \lie{c}$.
	Since $\jprod{u, u_i, u} = u_i$ for any $i$, we obtain $\jprod{u, x, u} = x$.

	Assume that $\vartheta = \id$.
	Let $x \in \lie{g}_1$.
	Then we have
	\begin{align*}
		\ad(u - \theta(u))^2 (x - \theta(x)) &= [u - \theta(u), [u,-\theta(x)] + [-\theta(u), x]] \\
		&= 2(-\vartheta(x) + x - \theta(x) + \theta(\vartheta(x))) = 0.
	\end{align*}
	This implies that $\ad(u - \theta(u))^2$ acts on $\lie{g}^{-\theta, -\sigma}$ trivially.
	Since $u - \theta(u)$ is a semisimple element in $\lie{g}^{-\theta, -\sigma}$, $u-\theta(u)$ belongs to the center of $\lie{g}^{\theta\sigma}$.
	By Fact \ref{fact:RootSystemJordan}, $\Delta(\lie{g}^{\theta\sigma}, \lie{a'})$ is of type $A_{r-1}$.
\end{proof}

\begin{fact}[{\cite{Ka98} (see also \cite[Part II, Theorem II.2.5]{FaKa00_jordan})}] \label{fact:JordanOrbitDecomposition}
	One has
	\begin{align*}
		\lie{g}_1 &= \coprod_{p, q \geq 0, p+q \leq r} \Ad(G_0)o_{p, q} && (\Delta(\lie{g}^{\theta\sigma}, \lie{a'}): A_{r-1}), \\
		\lie{g}_1 &= \coprod_{p=0}^r \Ad(G_0)o_{p, 0} && (\Delta(\lie{g}^{\theta\sigma}, \lie{a'}): C_r), \\
		\lie{g}_1 &= \coprod_{p=0}^r \Ad(G_0)o_{p, 0} \sqcup \Ad(G_0)o_{r-1, 1} && (\Delta(\lie{g}^{\theta\sigma}, \lie{a'}): D_r).
	\end{align*}
\end{fact}

Since the linear map $\lie{g}_1 \ni X\rightarrow X - \theta(X) \in \lie{g}^{-\theta, -\sigma}$ is $K_0$-equivariant, the Weyl group $W'$ of the restricted root system $\Delta(\lie{g}^{\theta\sigma}, \lie{a'})$ acts on $\lie{c}$.
The representation is isomorphic to the natural representation and the isomorphism is given by
\begin{align*}
	\lie{c} \ni \sum_{i=1}^r a_i u_i \mapsto (a_1, a_2, \ldots, a_r) \in \RR^r.
\end{align*}
For example, if $\Delta(\lie{g}^{\theta\sigma}, \lie{a'})$ is of type $A_{r-1}$, then $W'$ acts on the set $\set{u_1, \ldots, u_r}$ by permutations.
See the discussion following \cite[Lemma 3.2]{Ka98}, and also the discussion preceding \cite[Part II, Lemma II.2.2]{FaKa00_jordan}.

\begin{fact} \label{fact:WeylGroupJordan}
	The image of $W'$ in $\LieGL(\lie{c})$ coincides with the Weyl group of the root system:
	\begin{align*}
		&\set{\pm (u_i - u_j) : 1\leq i < j \leq r} && (\Delta(\lie{g}^{\theta\sigma}, \lie{a'}): A_{r-1}), \\
		&\set{\pm (u_i \pm u_j) : 1\leq i < j \leq r} \sqcup \set{\pm 2u_i: 1\leq i \leq r} && (\Delta(\lie{g}^{\theta\sigma}, \lie{a'}): C_r), \\
		&\set{\pm (u_i \pm u_j) : 1\leq i < j \leq r} && (\Delta(\lie{g}^{\theta\sigma}, \lie{a'}): D_{r}).
	\end{align*}
\end{fact}

We prove several lemmas that will be used in a later section to establish the non-minimality of NDD tuples.

\begin{lemma} \label{lem:RankOneTripotent}
	Let $X \in \Ad(G_0)o_{1, 0} \cup \Ad(G_0)o_{0, 1}$.
	Then there exists a constant $c > 0$ such that $cX$ is a tripotent, and hence $\set{[cX, -\theta(cX)], cX, -\theta(cX)}$ forms an $\lieSL_2$-triple.
\end{lemma}

\begin{proof}
	By Proposition \ref{prop:K0Orbit} and Fact \ref{fact:JordanOrbitDecomposition}, there exists $k\in K_0$ such that $\Ad(k)X \in \RR o_{1, 0} \cup \RR o_{0, 1}$.
	Since the property of being a tripotent is preserved under the action of $K_0$, the assertion follows.
\end{proof}

\begin{lemma} \label{lem:JordanAnyInnerProduct}
	Let $u' = o_{p,q}$ with $p+q = r$, and $X$ be a non-zero element in $\lie{g}_1$.
	If $\Delta(\lie{g}^{\theta\sigma}, \lie{a'})$ is of type $A_{r-1}$, assume that $p, q\neq 0$ in addition.
	Then, for each non-zero $a \in \RR$, there exists an element $x \in \overline{\Ad(G_0)X} \cap \lie{c}$ such that $(x, -\theta(u')) = a$.
\end{lemma}

\begin{proof}
	Assume that $\Delta(\lie{g}^{\theta\sigma}, \lie{a'})$ is not of type $A_{r-1}$.
	By Fact \ref{fact:JordanOrbitDecomposition}, $\RR o_{1, 0}$ is contained in $\overline{\Ad(G_0)X} \cap \lie{c}$.
	Hence some element in $\RR o_{1, 0}$ satisfies the required condition on $x$.

	Assume that $\Delta(\lie{g}^{\theta\sigma}, \lie{a'})$ is of type $A_{r-1}$.
	By Facts \ref{fact:JordanOrbitDecomposition} and \ref{fact:WeylGroupJordan}, $\varepsilon u_1$ and $\varepsilon u_r$ belong to $\overline{\Ad(G_0)X} \cap \lie{c}$ for some $\varepsilon \in \set{-1, 1}$.
	By assumption, we have $(u_1, -\theta(u')) > 0$ and $(u_r, -\theta(u')) < 0$.
	Hence some element in $\RR_{>0} \varepsilon u_1 \cup \RR_{>0} \varepsilon u_r$ satisfies the required condition on $x$.
\end{proof}

The following proposition is not necessary in this section.
It is used in the proof of Lemma \ref{lem:SUTopGrading}.

\begin{proposition} \label{prop:CodimensionOneJordan}
	Assume that $\dim(\lie{g}_1) \geq 2$.
	Let $W$ be a subspace of $\lie{g}_1$ such that $\Ad(G_0)W \cap \bigcup_{p+q=r} \Ad(G_0)o_{p,q} = \emptyset$.
	Then one has $\dim(W) < \dim(\lie{g}_1) - 1$.
\end{proposition}

\begin{proof}
	Assume that $\dim(W) = \dim(\lie{g}_1) - 1$.
	Set $O\coloneq \bigcup_{p+q=r} \Ad(G_0)o_{p,q}$.
	Then $O$ is a dense subset of $\lie{g}_1$ by Fact \ref{fact:JordanOrbitDecomposition}.

	Assume that $W$ is $\lie{g}_0$-stable.
	Then $W$ is an ideal of $\lie{g}_1$, i.e., $\jprod{\lie{g}_1, \lie{g}_1, W} \subset W$.
	This implies that the subalgebra generated by $W$ and $\theta(W)$ is a proper non-zero ideal of $\lie{g}$, which contradicts the simplicity of $\lie{g}$.
	Hence $W$ is not $\lie{g}_0$-stable.
	
	Since $W$ is not $\lie{g}_0$-stable, there exists $X \in W$ such that $[\lie{g}_0, X] \not\subset W$.
	Then we have $W + [\lie{g}_0, X] = \lie{g}_1$.
	Hence $\Ad(G_0)W$ contains an open neighborhood of $X$ in $\lie{g}_1$.
	Since $O$ is dense in $\lie{g}_1$, we have $\Ad(G_0)W \cap O \neq \emptyset$, which contradicts the assumption.
\end{proof}

\subsection{\texorpdfstring{$\lie{g} = \lieSO(p+1, q+1)$}{g=so(p+1,q+1)}} \label{subsect:JordanAlgebraSO(p,q)}

We shall consider the case of $\lie{g} = \lieSO(p+1, q+1)$ ($p, q \geq 0$, $p+q \geq 1$).
Retain the notation $\lie{g} = \lie{g}_{-1} \oplus \lie{g}_0 \oplus \lie{g}_1$, $\theta$, $h$ and $K_0$ from the previous subsection.

We have $\lie{g}_0 \simeq \lieSO(p, q) \oplus \RR h$ and $\lie{g}_1 \simeq \RR^{p, q}$ as a $\lie{g}_0$-module.
Here $\RR^{p, q}$ is the vector space $\RR^{p+q}$ equipped with the standard symmetric bilinear form $B$ of signature $(p, q)$.
Let $\langle \cdot, \cdot \rangle$ denote the standard inner product on $\RR^{p+q}$ and set
\begin{align*}
	I_{p,q} \coloneq \begin{pmatrix}
		I_p & 0 \\
		0 & -I_q
	\end{pmatrix}.
\end{align*}
Then we have $B(x,y) = \langle x, I_{p,q} y \rangle$.

The Jordan triple product is given by
\begin{align*}
	\jprod{x, y, z} &= \langle x, y \rangle z + \langle z, y \rangle x - B(x, z)I_{p,q}y \quad (x, y, z \in \RR^{p,q}).
\end{align*}
Let $\set{e_1, \ldots, e_{p+q}}$ be the standard basis of $\RR^{p,q}$ and set $u = e_1$.

\begin{proposition} \label{prop:TripotentSO(p,q)}
	Let $0\neq x \in \RR^{p,q}$.
	Then $x$ is a tripotent if and only if it satisfies exactly one of the following conditions:
	\begin{enumparen}
		\item $B(x, x) = 0$ and $\langle x, x \rangle = 1/2$,
		\item $\langle x, x \rangle = 1$ and $I_{p,q}x = x$,
		\item $\langle x, x \rangle = 1$ and $I_{p,q}x = -x$.
	\end{enumparen}
	Moreover, $x$ is unitary if and only if either (2) or (3) holds.
\end{proposition}

\begin{proof}
	The `if' part is verified by direct calculations.
	Assume that $x$ is a tripotent.
	Then we have $x = \jprod{x, x, x} = 2\langle x, x \rangle x - B(x, x)I_{p,q}x$.
	Hence $\set{x, I_{p,q}x}$ is linearly dependent or $B(x, x)=0$.
	The assertion follows from this.
\end{proof}

By Proposition \ref{prop:TripotentSO(p,q)}, any unitary tripotent is conjugate to one of $e_1$, $-e_1$, $e_{p+1}$ and $-e_{p+1}$ under the $K_0$-action.
Let $u \in \set{\pm e_1, \pm e_{p+1}}$, and set $\varepsilon = B(u, u)$.
The Cartan involution $\vartheta$ is given by
\begin{align*}
	\vartheta(x) = 2\langle x, u \rangle u - \varepsilon I_{p,q}x \quad (x \in \lie{g}_1).
\end{align*}
Hence, if $p, q\geq 1$, the set $\set{\frac{\varepsilon_1}{2}(e_1 + e_{p+1}), \frac{\varepsilon_2}{2}(e_1 - e_{p+1})}$ is a Jordan frame for some $\varepsilon_1, \varepsilon_2 \in \set{-1, 1}$.
If $p = 0$ or $q = 0$, then $\set{u}$ is a Jordan frame.

\begin{lemma} \label{lem:InvolutionSO(1,k)}
	One has
	\begin{align*}
		(\RR^{p, q})^{\vartheta} = \begin{cases}
			\spn{\RR}{e_1, e_{p+1}, \ldots, e_{p+q}} & (u=\pm e_1), \\
			\spn{\RR}{e_1, \ldots, e_p, e_{p+1}} & (u=\pm e_{p+1}),
		\end{cases} \\
		(\RR^{p,q})^{-\vartheta} = \begin{cases}
			\spn{\RR}{e_2, \ldots, e_p} & (u=\pm e_1), \\
			\spn{\RR}{e_{p+2}, \ldots, e_{p+q}} & (u=\pm e_{p+1}).
		\end{cases}
	\end{align*}
\end{lemma}

\subsection{Embedding of \texorpdfstring{$\lie{h} = \lieSO(1, k)$}{h=so(1,k)}}

In this subsection, we complete the proof of Theorem \ref{thm:MainTargetMinimal} in the case of type $A_1$ pairs.
Let $(\lie{g}, \lie{h})$ be a pair of type $A_1$ defined in Definition \ref{def:typeA}.
Then the embedding $\lie{h}_1 \rightarrow \lie{g}_1$ is a homomorphism of Jordan triple systems.
We remark that, in contrast to the case of type $BC_1$, $\lie{g}$ need not be simple.

Fix unitary tripotents $u' \in \lie{h}_1$ and $u \in \lie{g}_1$.
By Definition \ref{def:typeA}, $u'$ is unitary in $\lie{g}_1$.
Take a Jordan frame $\set{u_1, \ldots, u_r}$ in $(\lie{g}_1, u)^\vartheta$ and set $\lie{c}\coloneq \spn{\RR}{u_1, \ldots, u_r}$.
By Proposition \ref{prop:TripotentAndJordanFrame}, replacing $\lie{h}$ with a conjugate by an element of $K_0$, we may assume that $u' \in \lie{c}$.
Let $\vartheta'$ be the Cartan involution of $(\lie{g}_1, u')$ defined by $\vartheta'(x) = \jprod{u', x, u'}$.
Then the Jordan algebra $(\lie{h}_1, u')$ is a $\vartheta'$-stable subalgebra of $(\lie{g}_1, u')$.

Throughout this subsection, orthogonal complements are taken with respect to the form $-(\cdot,\theta(\cdot))$.

\begin{lemma} \label{lem:HperpInC}
	One has $\lie{c}\cap \lie{h} = \RR u'$ and $\lie{c} \cap \lie{h}^\perp = \lie{c} \cap (\RR u')^\perp$.
	In particular, one has $\lie{c} = (\lie{c}\cap \lie{h}) \oplus (\lie{c} \cap \lie{h}^\perp)$.
\end{lemma}

\begin{proof}
	By Lemma \ref{lem:InvolutionSO(1,k)}, we have $\lie{h}_1^{\vartheta'} = \RR u'$.
	By Proposition \ref{prop:TripotentAndJordanFrame}, there exist $\varepsilon_1, \ldots, \varepsilon_r \in \set{\pm 1}$ such that $u' = \sum_{i=1}^r \varepsilon_i u_i$.
	By Fact \ref{fact:JordanFrame}, we have $\lie{c} \subset \lie{g}_1^{\vartheta'}$ .
	Hence we obtain $\lie{c}\cap \lie{h} = \RR u'$.

	Since $\lie{g}_1^{\vartheta'}$ is orthogonal to $\lie{g}_1^{-\vartheta'}$, we have
	\begin{align*}
		\lie{c} \cap (\RR u')^\perp = \lie{c} \cap (\RR u')^\perp \cap (\lie{h}_1^{-\vartheta'})^\perp = \lie{c} \cap \lie{h}^\perp.
	\end{align*}
	We have shown the assertion.
\end{proof}

\begin{lemma} \label{lem:NonTrivialInvolution}
	If $k \geq 3$, one has $\vartheta' \neq \id_{\lie{g}_1}$.
\end{lemma}

\begin{proof}
	By Lemma \ref{lem:InvolutionSO(1,k)}, since $k \geq 3$, we have $\vartheta' \neq \id_{\lie{g}_1}$.
\end{proof}

First, we consider the case of simple $\lie{g}$.
Suppose that $\lie{g}_1^\vartheta$ is a simple Jordan algebra.
Define $K_0$, $G_0$, $o_{p,q}$ and $\lie{a'}$ as in Subsection \ref{subsect:JordanAlgebra}.
Then we have $u' = o_{a, r - a}$ ($0\leq a \leq r$).

\begin{lemma} \label{lem:NontrivialTripotent}
	If $\Delta(\lie{g}^{\theta\sigma}, \lie{a'})$ is of type $A_{r-1}$ and $k \geq 3$, then one has $a \neq 0, r$.
\end{lemma}

\begin{proof}
	Assume that $\Delta(\lie{g}^{\theta\sigma}, \lie{a'})$ is of type $A_{r-1}$.
	By Proposition \ref{prop:EuclideanJordanRoot}, we have $\vartheta = \id_{\lie{g}_1} \neq \vartheta'$.
	This implies that $u \neq \pm u'$ and hence $a \neq 0, r$.
\end{proof}

Let $(\lie{g}, \lie{h}, \orbit{O})$ be an NDD tuple.
By Lemma \ref{lem:ExistenceWeightVector}, there exists a non-zero element $X \in \lie{g}_1 \cap \lie{h}^\perp \cap \overline{\orbit{O}}$.

\begin{lemma} \label{lem:NonMinimalRank3}
	If $r \geq 3$, then $(\lie{g}, \lie{h}, \orbit{O})$ is not minimal.
\end{lemma}

\begin{proof}
	If $X \in \Ad(G_0)o_{1,0}\cup \Ad(G_0)o_{0,1}$, then $\set{[X, -\theta(X)], X, -\theta(X)}$ spans a subalgebra isomorphic to $\lieSL(2,\RR)$ by Lemma \ref{lem:RankOneTripotent}.
	Hence $(\lie{g}, \lie{h}, \orbit{O})$ is not minimal by Lemma \ref{lem:NonMinimalX}.
	
	Assume that $X \in \Ad(G_0)o_{p, q}$ with $p + q \geq 2$.
	We shall show $\lie{c}\cap \lie{h}^\perp \cap \overline{\orbit{O}} \neq \set{0}$, which is one of the assumptions in Lemma \ref{lem:ReductionRealStronglyOrthogonal}.
	The other assumption in Lemma \ref{lem:ReductionRealStronglyOrthogonal} is verified in Lemma \ref{lem:HperpInC}.
	Note that $\lie{g}$ is not isomorphic to $\lieSL(2,\RR)$.

	Assume that $\Delta(\lie{g}^{\theta\sigma}, \lie{a'})$ is of type $A_{r-1}$ and $u' = \pm u$.
	By Fact \ref{fact:RootSystemJordan}, $K_0$ fixes $u'$.
	Take $k \in K_0$ such that $\Ad(k)X \in \lie{c}$ by Proposition \ref{prop:K0Orbit}.
	Then we have $(\Ad(k)X, -\theta(u')) = (X, -\theta(u')) = 0$.
	By Lemma \ref{lem:HperpInC}, we obtain $\Ad(k)X \in \lie{c} \cap \lie{h}^\perp$, and hence $\lie{c}\cap \lie{h}^\perp \cap \overline{\orbit{O}} \neq \set{0}$.

	Assume that $\Delta(\lie{g}^{\theta\sigma}, \lie{a'})$ is not of type $A_{r-1}$ or $u' \neq \pm u$.
	By Fact \ref{fact:WeylGroupJordan}, if necessary, replacing $u'$ by $\varepsilon wu'$ for some $w \in W'$, $\varepsilon \in \set{-1, 1}$, we may assume that $u'=o_{a,r-a}$ ($a\geq r-a > 0$), and then $a \geq 2$.
	Since $X \in \Ad(G_0)o_{p, q}$ with $p + q \geq 2$, $\overline{\orbit{O}}$ contains one of the following tripotents:
	\begin{align*}
		u_a+u_{a+1}, \quad -u_a-u_{a+1}, \quad u_1-u_2
	\end{align*}
	by Fact \ref{fact:WeylGroupJordan}.
	The elements belong to $\lie{c} \cap (\RR u')^\perp$ and hence to $\lie{h}^\perp$ by Lemma \ref{lem:HperpInC}.
	Therefore we obtain $\lie{c}\cap \lie{h}^\perp \cap \overline{\orbit{O}} \neq \set{0}$.
\end{proof}

\begin{lemma}
	If $\Delta(\lie{g}^{\theta\sigma}, \lie{a'})$ is of type $C_{r}$ and $r\geq 2$, then $(\lie{g}, \lie{h}, \orbit{O})$ is not minimal.
\end{lemma}

\begin{proof}
	Under the assumptions, in the proof of Lemma \ref{lem:NonMinimalRank3}, we may choose the element in $\overline{\orbit{O}}$ to be $u_1 + u_2$ using the Weyl group action.
	Then we have $u_1 + u_2 \in \lie{c} \cap (\RR u')^\perp = \lie{c} \cap \lie{h}^\perp$ by Lemma \ref{lem:HperpInC}.
	This shows the lemma.
\end{proof}

The remaining cases are those in which $\Delta(\lie{g}^{\theta\sigma}, \lie{a'})$ is of type $A_{r-1}$ or $D_r$ ($r \leq 2$).
By the classification (Table \ref{table:JordanAlgebra}), the only possibility is the case of $\lie{g} = \lieSO(p+1, q+1)$ ($p, q \geq 1$).
Note that, if $p = 0$ or $q = 0$, then $(\lie{g}, \lie{h}, \orbit{O})$ is not minimal by Corollary \ref{cor:SufficientSameRank}.
Then we have $r = 2$.

As in Subsection \ref{subsect:JordanAlgebraSO(p,q)}, identify $\lie{g}_1$ with $\RR^{p,q}$.
Retain the notation $I_{p,q}$, $B$ and $\langle\cdot, \cdot \rangle$ from Subsection \ref{subsect:JordanAlgebraSO(p,q)}.
By Proposition \ref{prop:TripotentSO(p,q)}, swapping $p$ and $q$ if necessary, we may assume that $u = u' = e_1$, and $u_1 = (e_1 + e_{p+1})/2, u_2 = (e_1 - e_{p+1})/2$.
Then we have $\vartheta' = \vartheta$.
Here, we do not assume that $\lie{g}_1^{\vartheta}$ is simple.

\begin{lemma}\label{lem:Rank2Orthogonal}
	One has $\lie{g}_1^{-\vartheta} = \lie{h}_1^{-\vartheta}$.
\end{lemma}

\begin{proof}
	Assume that $\lie{g}_1^{-\vartheta} \neq \lie{h}_1^{-\vartheta}$.
	Recall that the bilinear form $B$ is definite on $\lie{g}_1^{-\vartheta}$ by Lemma \ref{lem:InvolutionSO(1,k)}.
	By the action of $K_0$, we may assume that $\lie{h}_1 = \spn{\RR}{e_1, \ldots, e_{k-1}}$.
	Then we have
	\begin{align*}
		\lie{g}_1 \cap \lie{h}^\perp = \spn{\RR}{e_k, \ldots, e_{p+q}}.
	\end{align*}
	This implies that the bilinear form $B$ is indefinite on $\lie{g}_1 \cap \lie{h}^\perp$.
	Let $\lie{h'}$ be the subalgebra of $[\lie{g}_0, \lie{g}_0] \simeq \lieSO(p, q)$ isomorphic to $\lieSO(p - (k-1), q)$.
	This subalgebra is the Lie algebra of the isometry group of $\lie{g}_1\cap \lie{h}^\perp$.
	Then $\lie{h'}$ is $\theta$-stable and acts on $\lie{h}_1$ trivially.
	Since $\lie{h}$ is generated by $\lie{h}_1$ and $\theta(\lie{h}_1)$, we have $\lie{h'}\subset \lieCent_{\lie{g}}(\lie{h})$.
	This contradicts the assumption that $\lieCent_{\lie{g}}(\lie{h})$ is compact.
\end{proof}

\begin{lemma}
	Under the above assumptions, $(\lie{g}, \lie{h}, \orbit{O})$ is not minimal.
\end{lemma}

\begin{proof}
	By Lemma \ref{lem:Rank2Orthogonal}, we have $p+1 = k$ and $\lieNorm_{\lie{g}}(\lie{h}) \simeq \lieSO(k, 1) \oplus \lieSO(0, q)$.
	This implies that $(\lie{g}, \lieNorm_{\lie{g}}(\lie{h}))$ is a symmetric pair.
	By Lemma \ref{lem:SufficientSymmetric}, $(\lie{g}, \lie{h}, \orbit{O})$ is not minimal.
\end{proof}

\begin{theorem} \label{thm:NonMinimalTypeASimple}
	Let $(\lie{g}, \lie{h}, \orbit{O})$ be an NDD tuple such that $(\lie{g}, \lie{h})$ is of type $A_1$.
	If $\lie{g}$ is simple, then $(\lie{g}, \lie{h}, \orbit{O})$ is not minimal.
\end{theorem}

Next, we shall consider the case in which $\lie{g}$ is not simple.
Let $\lie{g} = \lie{g}^1\oplus \lie{g}^2 \oplus \cdots \oplus \lie{g}^m$ ($m \geq 2$) be the decomposition of $\lie{g}$ into simple ideals.
Then $\lie{g}_1 = \lie{g}_1^1 \oplus \lie{g}_1^2 \oplus \cdots \oplus \lie{g}_1^m$ is the decomposition of $\lie{g}_1$ into simple Jordan triple systems.
Write $p_i\colon \lie{g} \rightarrow \lie{g}^i$ for the projection onto the $i$-th component.
By Definition \ref{def:typeA}, we have $\lie{g}_1^i \neq 0$ and $p_i(\lie{h}) \neq 0$ for any $i$.

\begin{lemma} \label{lem:SL2DirectSum}
	Assume that $\lie{h} \simeq \lieSL(2,\RR)$ and every simple factor of $\lie{g}$ is isomorphic to $\lieSL(2,\RR)$.
	Then $(\lie{g}, \lie{h}, \orbit{O})$ is not minimal for any $\orbit{O}$.
\end{lemma}

\begin{proof}
	We may assume that $\lie{g} = \lieSL(2,\RR)^{\oplus m}$ and $\lie{h}$ is the diagonal subalgebra of $\lie{g}$.
	As a $G$-invariant bilinear form on $\lie{g}$, we take the Killing form.
	Let $\set{h_0, x_0, y_0}$ be an $\lieSL_2$-triple in $\lieSL(2,\RR)$ such that $2h = (h_0, h_0, \ldots, h_0)$.
	By Lemma \ref{lem:ExistenceWeightVector}, $\lie{g}_1 \cap \lie{h}^\perp \cap \overline{\orbit{O}}$ contains
	a nonzero element $X = (a_1 x_0, a_2 x_0, \ldots, a_m x_0)$ for some $a_i \in \RR$.
	Since $X \in \lie{h}^\perp$, we have $\sum_{i=1}^m a_i = 0$.
	In particular, there exist $i, j$ such that $a_i > 0$ and $a_j < 0$.
	We may assume that $a_1 > 0$ and $a_2 < 0$.

	By acting with $G_0$ and passing to a suitable limit, we have
	\begin{align*}
		X' \coloneq (x_0, -x_0, 0, \ldots, 0) \in \lie{g}_1\cap \overline{\orbit{O}}\cap \lie{h}^\perp.
	\end{align*}
	Since $\set{[X', -\theta(X')], X', -\theta(X')}$ is an $\lieSL_2$-triple, $(\lie{g}, \lie{h}, \orbit{O})$ is not minimal.
\end{proof}

\begin{theorem} \label{thm:NonMinimalTypeASemisimple}
	Let $(\lie{g}, \lie{h}, \orbit{O})$ be an NDD tuple such that $(\lie{g}, \lie{h})$ is of type $A_1$.
	If $\lie{g}$ is not simple, then $(\lie{g}, \lie{h}, \orbit{O})$ is not minimal.
\end{theorem}

\begin{proof}
	By Lemma \ref{lem:ExistenceWeightVector}, there exists a non-zero element $X \in \lie{g}_1 \cap \lie{h}^\perp \cap \overline{\orbit{O}}$.
	If $p_i(X) = 0$ for some $i$, then $(\lie{g}, \lie{h}, \orbit{O})$ is not minimal by Proposition \ref{prop:BasicMinimal} (6).
	Hence we may assume that $p_i(X) \neq 0$ for any $i$.

	First, assume that $k \geq 3$, i.e., $\lie{h}$ is not isomorphic to $\lieSL(2,\RR)$.
	Note that $\lie{g}^i$ is not isomorphic to $\lieSL(2,\RR)$ for any $i$ since $p_i(\lie{h}) \neq 0$.
	By Lemma \ref{lem:NontrivialTripotent}, the assumption of Lemma \ref{lem:JordanAnyInnerProduct} is satisfied for each simple factor.
	Hence we can take an element $X' \in \lie{c} \cap \overline{\orbit{O}}$ such that $(X', -\theta(u')) = 0$.
	By Lemma \ref{lem:HperpInC}, this implies that $X' \in \lie{h}^\perp$.
	In particular, we have  $\lie{c} \cap \lie{h}^\perp \cap \overline{\orbit{O}} \neq \set{0}$.
	By Lemma \ref{lem:HperpInC}, again, we have $\lie{c} = (\lie{c}\cap \lie{h}) \oplus (\lie{c} \cap \lie{h}^\perp)$.
	Therefore, $\lie{c}$ satisfies the assumptions of Lemma \ref{lem:ReductionRealStronglyOrthogonal}, and hence $(\lie{g}, \lie{h}, \orbit{O})$ is not minimal.

	Next, assume that $k = 2$, i.e., $\lie{h} \simeq \lieSL(2,\RR)$.
	If every simple factor of $\lie{g}$ is isomorphic to $\lieSL(2,\RR)$, then $(\lie{g}, \lie{h}, \orbit{O})$ is not minimal by Lemma \ref{lem:SL2DirectSum}.
	If the involution $\vartheta'(\cdot) = \jprod{u', \cdot, u'}$ is not the identity, $(\lie{g}, \lie{h}, \orbit{O})$ is not minimal by the same argument as in the first case.
	
	Hence we may assume that $\lie{g}$ has a simple factor not isomorphic to $\lieSL(2,\RR)$, and $\jprod{u', \cdot, u'}$ is the identity, which implies that every simple factor of $\lie{g}_1$ is Euclidean by Proposition \ref{prop:EuclideanJordanRoot}.
	By Fact \ref{fact:RootSystemJordan}, $u'$ is a $K_0$-fixed vector.
	Take $k \in K_0$ such that $\Ad(k)X \in \lie{c}$ by Proposition \ref{prop:K0Orbit}.
	Then we have $(\Ad(k)X, -\theta(u')) = (X, -\theta(u')) = 0$.
	This implies that $\Ad(k)X \in \lie{c} \cap \lie{h}^\perp$, and hence $\lie{c}\cap \lie{h}^\perp \cap \overline{\orbit{O}} \neq \set{0}$.
	Therefore, $\lie{c}$ satisfies the assumptions of Lemma \ref{lem:ReductionRealStronglyOrthogonal}, and hence $(\lie{g}, \lie{h}, \orbit{O})$ is not minimal.
\end{proof}

\section{Type \texorpdfstring{$BC'_1$}{BC'1}} \label{section:TypeBC1}

In this section, we deal with the $BC'_1$ case defined in Definition \ref{def:typeBC}.
The aim of this part is to prove that, if $(\lie{g}, \lie{h})$ is of type $BC'_1$, then the tuple $(\lie{g}, \lie{h}, \orbit{O})$ is never minimal.
If $\lie{h}$ is isomorphic to $\lieSp(1,k)$ or $\lieF_{4(-20)}$, then non-minimality is readily verified by simple arguments in Subsection \ref{subsect:EmbeddingBasic}.

In the case where $\lie{h} \simeq \lieSU(1,k)$, there are finitely many possible series for $\lie{g}$, but many candidates occur, especially among exceptional Lie algebras.
For exceptional $\lie{g}$, we establish non-minimality by giving a computer-assisted classification of all embeddings $\lie{h}\hookrightarrow \lie{g}$.
The proof of Theorem \ref{thm:MainTargetMinimal} is completed in this section.

\subsection{Basic properties} \label{subsect:EmbeddingBasic}

In this subsection, we establish several basic properties of embeddings $\lie{h}\hookrightarrow\lie{g}$ of type $BC'_1$ in preparation for classifying them.
Let $(\lie{g}, \lie{h})$ be a pair of type $BC'_1$ and fix $h \in \lie{h}^{-\theta}$ as in Definition \ref{def:typeBC}.
Set $\lie{a}_H \coloneq \RR h$.
Then $\lie{g}$ is simple and we have
\begin{align*}
	\lie{g} = \lie{g}_{-2} \oplus \lie{g}_{-1} \oplus \lie{g}_0 \oplus \lie{g}_1 \oplus \lie{g}_2.
\end{align*}
Note that $\lie{h}$ is isomorphic to either $\lieSU(1, k)$ ($k \geq 2$), $\lieSp(1, k)$ ($k \geq 2$) or $\lieF_{4(-20)}$.

\begin{lemma} \label{lem:NonMinimalInvariant}
	Let $(\lie{g}, \lie{h}, \orbit{O})$ be an NDD tuple.
	Assume that there exists a non-zero element $X \in \lie{g}_1 \cap \overline{\orbit{O}} \cap \lie{h}^\perp$ such that $\lie{k}_H \cap \lie{g}_0^X \neq 0$.
	Then $(\lie{g}, \lie{h}, \orbit{O})$ is not minimal.
\end{lemma}

\begin{proof}
	Take $0\neq v \in \lie{k}_H \cap \lie{g}_0^X$.
	Then $\lie{g'} \coloneq \lieCent_{\lie{g}}(v)$ is a $\theta$-stable proper subalgebra of $\lie{g}$ containing $X$ and $\lie{a}_H$.
	By definition, we have
	\begin{align*}
		\lie{g'} = \lie{h}^v \oplus (\lie{h}^\perp)^v = (\lie{g'} \cap \lie{h}) \oplus (\lie{g'} \cap \lie{h}^\perp).
	\end{align*}
	Hence $(\lie{g}, \lie{h}, \orbit{O})$ is not minimal by Proposition \ref{prop:NonMinimalSubalgebra}.
\end{proof}

We shall see that there are no minimal NDD tuples $(\lie{g}, \lie{h}, \orbit{O})$ when $\lie{h} \simeq \lieSp(1, k)$ or $\lieF_{4(-20)}$.
Assume that $\lie{h} \simeq \lieSp(1, k)$ ($k \geq 2$).
Then we have
\begin{align*}
	\lie{h}_0 \simeq \RR h \oplus \lieSp(1) \oplus \lieSp(k-1).
\end{align*}
We denote by $\lie{k}'$ the third factor $\lieSp(k-1)$.

\begin{lemma} \label{lem:SpInvariant}
	Let $V$ be an irreducible $\lie{h}$-module over $\RR$ such that the $h$-eigenspace decomposition of $V$ is written as $V = V_{-1} \oplus V_0 \oplus V_1$ with $V_1 \neq 0$.
	Then any vector in $V_1$ is $\lie{k}'$-invariant.
\end{lemma}

\begin{proof}
	It is enough to show the same assertion for an irreducible $\lie{h}$-module $V$ over $\CC$.
	Fix a Cartan subalgebra $\lieC{t}$ of $\lieSp(1+k, \CC)$ containing $h$, and take its coordinate system $\set{\varepsilon_1, \ldots, \varepsilon_{k+1}}$ such that
	\begin{align*}
		\Delta^+(\lieC{h}, \lieC{t}) &= \set{\varepsilon_i \pm \varepsilon_j : 1\leq i < j \leq k+1} \cup \set{2\varepsilon_i : 1 \leq i \leq k+1},\\
		\varepsilon_i(h) &= \begin{cases}
			1 & (i = 1,2), \\
			0 & (3 \leq i \leq k+1).
		\end{cases}
	\end{align*}
	Then a dominant integral weight in $\lieC{t}^*$ is of the form $\sum_{i=1}^{k+1} m_i \varepsilon_i$ with $m_1 \geq m_2 \geq \cdots \geq m_{k+1} \geq 0$.

	Let $\lambda$ be the highest weight of $V$.
	By assumption, we have $\lambda(h) = 1$.
	This implies that $\lambda = \varepsilon_1$; in other words, $V$ is isomorphic to the natural representation $\CC^{2k+2}$ of $\lieC{h}$.
	Hence any vector in $V_1 \simeq \CC^2$ is $\lie{k}'$-invariant.
\end{proof}

\begin{theorem} \label{thm:NonMinimalSp}
	Assume that $\lie{h} \simeq \lieSp(1, k)$ ($k \geq 2$).
	Then there exists no minimal NDD tuple $(\lie{g}, \lie{h}, \orbit{O})$.
\end{theorem}

\begin{proof}
	Let $(\lie{g}, \lie{h}, \orbit{O})$ be a minimal NDD tuple.
	By Lemma \ref{lem:ReductionLemmaStep1}, there exists a non-zero element $X \in \lie{g}_1 \cap \overline{\orbit{O}} \cap \lie{h}^\perp$.
	By Corollary \ref{cor:ReductionLemmaStep3}, we have $[X, \lie{h}_1] = 0$.
	By the grading, we have $[X, \lie{h}_2] = 0$.
	This implies that every irreducible component in the $\lie{h}$-submodule $V$ generated by $X$ satisfies the assumption of Lemma \ref{lem:SpInvariant}.
	Combining Lemmas \ref{lem:NonMinimalInvariant} and \ref{lem:SpInvariant}, we conclude that $(\lie{g}, \lie{h}, \orbit{O})$ is not minimal.
\end{proof}

Next, we assume that $\lie{h} \simeq \lieF_{4(-20)}$.
By considering the analogue of Lemma \ref{lem:SpInvariant} for this $\lie{h}$, we obtain the following stronger result.

\begin{lemma} \label{lem:NonIrreducibleF4}
	There exists no irreducible $\lie{h}$-module $V$ over $\RR$ such that the $h$-eigenspace decomposition of $V$ is written as $V = V_{-1} \oplus V_0 \oplus V_1$ with $V_1 \neq 0$.
\end{lemma}

\begin{proof}
	Fix a Cartan subalgebra $\lieC{t}$ of $\lieC{h}$ containing $h$, and take its coordinate system $\set{\varepsilon_1, \ldots, \varepsilon_4}$ such that
	\begin{align*}
		&\alpha_1 \coloneq \varepsilon_2 - \varepsilon_3, \quad \alpha_2 \coloneq \varepsilon_3 - \varepsilon_4, \\
		&\alpha_3 \coloneq \varepsilon_4, \quad \alpha_4 \coloneq \frac{1}{2}(\varepsilon_1 - \varepsilon_2 - \varepsilon_3 - \varepsilon_4)
	\end{align*}
	are simple roots in $\Delta(\lieC{h}, \lieC{t})$ and $\alpha_4(h) = 1, \alpha_i(h) = 0$ for $i = 1, 2, 3$.
	Then we have $\varepsilon_1(h) = 2$ and $\varepsilon_i(h) = 0$ for $i = 2, 3, 4$.
	The fundamental weights $\omega_1, \ldots, \omega_4$ corresponding to $\alpha_1, \ldots, \alpha_4$ are given by
	\begin{align*}
		\omega_1 &= \varepsilon_1 + \varepsilon_2, \\
		\omega_2 &= 2\varepsilon_1 + \varepsilon_2 + \varepsilon_3, \\
		\omega_3 &= \frac{1}{2}(3\varepsilon_1 + \varepsilon_2 + \varepsilon_3 + \varepsilon_4), \\
		\omega_4 &= \varepsilon_1.
	\end{align*}
	Hence any non-zero dominant integral weight $\lambda = \sum_{i=1}^4 m_i \omega_i$ with $m_i \in \ZZ_{\geq 0}$ satisfies $\lambda(h) = 2m_1 + 4m_2 + 3m_3 + 2m_4 > 1$.
	This shows the assertion.
\end{proof}

The following theorem is shown by the same argument as in the proof of Theorem \ref{thm:NonMinimalSp} using Lemma \ref{lem:NonIrreducibleF4} instead of Lemma \ref{lem:SpInvariant}.

\begin{theorem} \label{thm:NonMinimalF4}
	Assume that $\lie{h} \simeq \lieF_{4(-20)}$.
	Then there exists no minimal NDD tuple $(\lie{g}, \lie{h}, \orbit{O})$.
\end{theorem}

We shall consider the case where $\lie{h}$ is isomorphic to $\lieSU(1, k)$ ($k \geq 2$).
There are too many possible pairs $(\lie{g},\lie{h})$ to classify directly.
We give an additional condition for the pair $(\lie{g}, \lie{h})$ to have a minimal NDD tuple $(\lie{g}, \lie{h}, \orbit{O})$.

\begin{lemma} \label{lem:SUTopGradingIrr}
	$\lie{g}_2$ is an irreducible $\lie{g}_0$-module.
	In particular, $\lie{g}_2$ is a simple Jordan triple system.
\end{lemma}

\begin{proof}
	Let $\lie{g}_2 = \bigoplus_{i=1}^n V^i$ be the decomposition into irreducible $\lie{g}_0$-modules.
	Let $\lie{j}^i$ be the $\lie{g}$-submodule generated by $V^i$.
	Since the $V^i$ are annihilated by $\lie{g}_1\oplus \lie{g}_2$ and $\lie{g}_0$-stable, we have $\lie{j}^i \cap \lie{g}_2 = V^i$.
	This implies that the $\lie{j}^i$ are non-zero ideals of $\lie{g}$.
	Since $\lie{g}$ is simple, we have $n=1$ and hence $\lie{g}_2$ is irreducible.
\end{proof}

By Lemma \ref{lem:SUTopGradingIrr}, the subalgebra generated by $\lie{g}_2$ and $\theta(\lie{g}_2)$ is a simple Lie algebra with the $3$-grading induced from the grading $\lie{g} = \bigoplus_{i=-2}^2 \lie{g}_i$.
Hence the results in Subsection \ref{subsect:JordanAlgebra} can be applied to $\lie{g}_2$.

\begin{lemma} \label{lem:SUTopGrading}
	Assume that there exists a minimal NDD tuple $(\lie{g}, \lie{h}, \orbit{O})$.
	Then one has $\lie{g}_2 = \lie{h}_2$.
\end{lemma}

\begin{proof}
	By Lemma \ref{lem:ReductionLemmaStep1}, we have $\overline{\orbit{O}} \cap \lie{g}_1 \neq \set{0}$.
	Let $X$ be a non-zero element in $\overline{\orbit{O}} \cap \lie{g}_1$.
	By Lemma \ref{lem:ReductionLemmaStep3}, $\lie{g}_2^{\theta(X)}$ has codimension one in $\lie{g}_2$.
	Set $\lie{g}' \coloneq \lie{g}^X \cap \lie{g}^{\theta(X)}$.
	Then $\lie{g}'$ is a $\theta$-stable graded subalgebra of $\lie{g}$ with $\lie{g}'_2 = \lie{g}_2^{\theta(X)}$.
	
	By the definition of $\lie{g}^{\theta(X)}$, $h$ does not belong to $\lie{g'}$.
	This implies that $\lie{g}_2^{\theta(X)}$ does not contain any unitary element in $\lie{g}_2$.
	Since $\overline{\orbit{O}} \cap \lie{g}_1$ is $G_0$-stable, for any $g \in G_0$, the subspace $\Ad(g)\lie{g}_2^{\theta(X)}$ does not contain any unitary element in $\lie{g}_2$.
	By Proposition \ref{prop:CodimensionOneJordan}, we obtain $\dim(\lie{g}_2) = 1$ and hence $\lie{g}_2 = \lie{h}_2$.
\end{proof}

By Lemma \ref{lem:SUTopGrading}, we may assume that $\lie{g}_2 = \lie{h}_2$.
In other words, $(\lie{g}, \lie{h})$ is of type $BC_1$ in the sense of Definition \ref{def:typeBC}.

\begin{lemma} \label{lem:G2RootSpace}
	Let $\lie{a}$ be a maximal abelian subspace of $\lie{g}^{-\theta}$ containing $\lie{a}_H$.
	Fix a positive system $\Delta^+(\lie{g}, \lie{a})$ containing $\Delta(\lie{g}_1\oplus \lie{g}_2, \lie{a})$.
	Then $\lie{g}_2$ is the root space in $\lie{g}$ corresponding to the highest restricted root.
\end{lemma}

\begin{proof}
	Since $\dim(\lie{g}_2) = \dim(\lie{h}_2) = 1$, the $\lie{a}$-stable subspace $\lie{g}_2$ is a restricted root space with restricted root $\alpha \in \lie{a}^*$.
	By the grading, it is clear that $\alpha$ is the highest restricted root.
\end{proof}

Table \ref{table:MultiplicityHighestRoot} lists the real simple Lie algebras together with the multiplicity of the highest restricted root.
See \cite[Appendices C.3 and C.4]{Kn02}.
We omit complex Lie algebras since $\lie{g}_2$ cannot be one-dimensional.
Moreover, we exclude the case $\rank_\RR(\lie{g})=1$, as Corollary \ref{cor:SufficientSameRank} shows that no NDD tuple $(\lie{g}, \lie{h}, \orbit{O})$ can be minimal in this case.

\begin{table}[htbp]
	\centering
	\begin{tabular}{c|c}
		$\lie{g}$ & $\dim(\lie{g}_{2})$ \\
		\hline
		$\lieSL(n, \mathbb{R})$ ($n \geq 3$) & $1$ \\
		$\lieSU(p, q)$ ($2 \leq p \leq q$) & $1$ \\
		$\lieSL(n, \HH)$ ($n \geq 2$) & $4$ \\
		$\lieSO(p, q)$ ($2 \leq p \leq q$) & $1$ \\
		$\lieSO^*(2n)$ ($n \geq 4$) & $1$ \\
		$\lieSp(n, \mathbb{R})$ ($n \geq 3$) & $1$ \\
		$\lieSp(p, q)$ ($2 \leq p \leq q$) & $3$ \\
		$\lieE_{6(6)}$ & $1$ \\
		$\lieE_{6(2)}$ & $1$ \\
		$\lieE_{6(-14)}$ & $1$ \\
		$\lieE_{6(-26)}$ & $8$ \\
		$\lieE_{7(7)}$ & $1$ \\
		$\lieE_{7(-5)}$ & $1$ \\
		$\lieE_{7(-25)}$ & $1$ \\
		$\lieE_{8(8)}$ & $1$ \\
		$\lieE_{8(-24)}$ & $1$ \\
		$\lieG_{2(2)}$ & $1$ \\
		$\lieF_{4(4)}$ & $1$ \\
	\end{tabular}
	\caption{Real simple Lie algebras $\lie{g}$ and the multiplicity of the highest restricted root}
	\label{table:MultiplicityHighestRoot}
\end{table}

To describe the embedding of $\lieSU(1, k)$, we examine some properties of the fundamental representations of $\lieSU(1, k)$.

\begin{lemma} \label{lem:FundamentalRep}
	Let $V$ be a non-trivial irreducible representation of $\lie{h} = \lieSU(1, k)$ over a field $\KK \in \set{\RR, \CC, \HH}$.
	Assume that the eigenspace decomposition of $V$ with respect to the action of $h$ is written as $V = V_{-1} \oplus V_0 \oplus V_1$, where $V_i$ is the eigenspace of $h$ with eigenvalue $i$.
	\begin{enumparen}
		\item If $\KK = \CC$, then $V$ is isomorphic to $\bigwedge\nolimits^j \CC^{k+1}$ for some $0 < j \leq k$.
		\item If $\KK = \CC$ and $\dim_{\CC}(V_1) = 1$, then one has $j = 1 \text{or} k$ in (1).
		If $\KK = \CC$ and $\dim_{\CC}(V_1) = 2$, then one has $(j, k) = (2, 3)$.
		\item $\KK = \RR$ and $\dim_{\RR}(V_1) = 1$ cannot happen.
		\item If $\KK = \HH$ and $\dim_{\HH}(V_1) = 1$, then there exist two possible cases:
		\begin{align*}
			V&\simeq \HH\otimes_{\CC}\CC^{k+1}, \\
			V&\simeq \HH^3 \quad \text{under the isomorphism $\lieSU(1, 3) \simeq \lieSO^*(6)$},
		\end{align*}
		where the second case occurs when $k = 3$.
		\item $\bigwedge\nolimits^2 \CC^4$ has no real form.
	\end{enumparen}
\end{lemma}

\begin{proof}
	Realize $\lie{h}_{\CC}$ as the standard matrix Lie algebra $\lie{sl}(k+1, \CC)$ such that $h$ is represented by
	\begin{align*}
		\begin{pmatrix}
			1 & 0 & \cdots & 0 & 0 \\
			0 & 0 & \cdots & 0 & 0 \\
			\vdots & \vdots & \ddots & \vdots & \vdots \\
			0 & 0 & \cdots & 0 & 0 \\
			0 & 0 & \cdots & 0 & -1
		\end{pmatrix}.
	\end{align*}
	Consider the standard Borel subalgebra and Cartan subalgebra of $\lie{sl}(k+1, \CC)$ consisting of upper triangular matrices and diagonal matrices, respectively.
	Then the highest weight of $V$ is a fundamental weight $\omega_j$ ($1 \leq j \leq k$) by the assumption $V = V_1 \oplus V_0 \oplus V_{-1}$.
	In other words, for each $V$, there exists $0 < j \leq k$ such that $V \simeq \bigwedge\nolimits^j \CC^{k+1}$.
	This shows (1).

	For each $V = \bigwedge\nolimits^j \CC^{k+1}$ with $0 < j \leq k$, we have $V_1 \simeq \bigwedge\nolimits^{j-1} \CC^{k-1}$.
	Hence $\dim_{\CC}(V_1) \in \set{1, 2}$ if and only if $j \in \set{1, k}$, or $(j, k) = (2, 3)$.
	This shows (2).

	To show (3), assume that $\KK = \RR$ and $\dim_{\RR}(V_1) = 1$.
	Then $V\otimes_{\RR} \CC$ is irreducible over $\CC$ and $\dim_{\CC}(V_1) = 1$.
	By (2), we have $V\otimes_{\RR} \CC \simeq \CC^{k+1}$ or $(\CC^{k+1})^*$.
	By definition, $V\otimes_{\RR}\CC$ has a real form.
	This happens only if $\lieSU(1, k) \simeq \lieSL(1+k, \RR)$.
	This contradicts the assumption $k \geq 2$.

	We shall show (4).
	Take an irreducible subrepresentation $W$ of $V$ over $\RR$.
	Since $\HH\cdot W = V$, any irreducible subrepresentation of $V$ over $\RR$ is isomorphic to $W$.
	Hence $V$ is completely reducible over $\RR$.
	By (3), we have $\dim_{\RR}(W \cap V_1) = 2$ or $4$.
	
	Since $\HH\cdot W = V$, we have
	\begin{align*}
		V \simeq \Hom_{\RR, \lie{h}}(W, V) \otimes_{\End_{\RR, \lie{h}}(W)} W.
	\end{align*}
	Since $V$ is irreducible over $\HH$, the space $\Hom_{\RR, \lie{h}}(W, V)$ is one-dimensional over $\HH$.
	Comparing the dimensions, we have
	\begin{align*}
		\End_{\RR, \lie{h}}(W) \simeq \begin{cases}
			\CC & (\dim_{\RR}(W \cap V_1) = 2), \\
			\HH & (\dim_{\RR}(W \cap V_1) = 4).
		\end{cases}
	\end{align*}
	This implies that $W$ has a complex structure.
	By (2), we obtain
	\begin{align*}
		W \simeq \begin{cases}
			\CC^{k+1} & (\dim_{\RR}(W \cap V_1) = 2), \\
			\bigwedge\nolimits^2 \CC^4 & (\dim_{\RR}(W \cap V_1) = 4).
		\end{cases}
	\end{align*}
	Note that $\CC^{k+1}$ is isomorphic to $(\CC^{k+1})^* \simeq \overline{\CC^{k+1}}$ as a representation over $\RR$.
	Therefore we have shown (4).

	(5) follows from the proof of (4) above.
	In fact, if $\bigwedge\nolimits^2 \CC^4$ has a real form, then the real form is isomorphic to $\CC^4$ by the proof of (4), which is a contradiction.
	Note that any finite-dimensional $\HH$-representation of a simple real Lie algebra that is irreducible over $\CC$ cannot admit a real form.
\end{proof}

Classifying embeddings of $\lie{h}$ is not easy if one only uses the fact that they are embeddings of $\lieSU(1,k)$.
However, by using the fact that $\lieC{h}$ is a regular subalgebra of $\lieC{g}$, as shown in the following lemma, the classification can be reduced to computations in root systems.

\begin{lemma} \label{lem:EmbeddingSU}
	Assume that $\lie{h} \simeq \lieSU(1, k)$ ($k \geq 2$).
	Then $\lie{g}$ has a compact Cartan subalgebra contained in $\lieNorm_{\lie{g}}(\lie{h})$.
	In particular, $\lieC{h}$ is a regular subalgebra of $\lieC{g}$.
\end{lemma}

\begin{proof}
	Decompose the $\lieC{h}$-module $\lieC{g}$ into irreducible components as
	\begin{align*}
		\lieC{g} = \lieC{h} \oplus \lieCent_{\lieC{g}}(\lieC{h}) \oplus \bigoplus_{i=1}^m V_i,
	\end{align*}
	where $V_i$ are irreducible $\lieC{h}$-modules with non-trivial $\lieC{h}$-action.
	Since $\lie{g}_2 = \lie{h}_2$, we have 
	\begin{align*}
		\bigoplus_{i=1}^m V_i \subset (\lieC{g})_{-1} \oplus (\lieC{g})_0 \oplus (\lieC{g})_1.
	\end{align*}

	Fix a compact Cartan subalgebra $\lie{t}$ of $\lie{h}$.
	By Lemma \ref{lem:FundamentalRep}, for each $V_i$, there exists $0 < j \leq k$ such that $V_i \simeq \bigwedge\nolimits^j \CC^{k+1}$.
	In particular, none of the $V_i$ has zero weight, and we have $V_i^{\lie{t}} = 0$.
	This implies that
	\begin{align}
		(\lieC{g})^{\lie{t}} = (\lieC{h})^{\lie{t}} \oplus \lieCent_{\lieC{g}}(\lieC{h}) \oplus \bigoplus_{i=1}^m V_i^{\lie{t}} = \lieC{t} \oplus \lieCent_{\lieC{g}}(\lieC{h}). \label{eqn:CentralizerCartan}
	\end{align}
	Since $(\lie{g}, \lie{h})$ is of type $BC_1$, $\lieCent_{\lie{g}}(\lie{h})$ is compact.
	Take a maximal torus $\lie{t'}$ of $\lieCent_{\lie{g}}(\lie{h})$.
	Then $\lie{t}\oplus \lie{t'}$ is a compact Cartan subalgebra of $\lie{g}$ by \eqref{eqn:CentralizerCartan}.
\end{proof}

Combining Lemma \ref{lem:EmbeddingSU} and Table \ref{table:MultiplicityHighestRoot}, we obtain the following corollary.

\begin{corollary} \label{cor:EmbeddingSU}
	If $\lie{h} \simeq \lieSU(1, k)$ ($k \geq 2$), then $\lie{g}$ is isomorphic to one of the following Lie algebras:
	\begin{align*}
		&\lieSU(p, q) \quad (p, q \geq 1, p+q\geq 3), \\
		&\lieSO(p, q) \quad (p, q \geq 2, p+q\geq 7, p\text{:even}), \\
		&\lieSO^*(2n) \quad (n \geq 5), \\
		&\lieSp(n,\RR) \quad (n \geq 2), \\
		&\lieE_{6(2)}, \quad \lieE_{6(-14)}, \\
		&\lieE_{7(7)}, \quad \lieE_{7(-5)}, \quad \lieE_{7(-25)}, \\
		&\lieE_{8(8)}, \quad \lieE_{8(-24)}, \\
		&\lieF_{4(4)}, \\
		&\lieG_{2(2)}.
	\end{align*}
\end{corollary}

\subsection{Embeddings of \texorpdfstring{$\lieSU(1, k)$}{su(1,k)}: Classical \texorpdfstring{$\lie{g}$}{g}}

In this subsection, we concentrate on the case where $\lie{h}$ is isomorphic to $\lieSU(1, k)$ ($k \geq 2$) and $\lie{g}$ is classical.
By Corollary \ref{cor:EmbeddingSU}, the possible choices of $\lie{g}$ are listed as follows:
\begin{align*}
	&\lieSU(p, q) \quad (p, q \geq 2), \\
	&\lieSO(p, q) \quad (p, q \geq 2, p+q\geq 7, p\text{:even}), \\
	&\lieSO^*(2n) \quad (n \geq 5), \\
	&\lieSp(n,\RR) \quad (n \geq 2).
\end{align*}
For simplicity, we consider $\lieU(p, q)$ instead of $\lieSU(p, q)$.
Although $(\lie{g}, \lie{h})$ is not of type $BC_1$ if $\lie{g} = \lieU(p, q)$, this replacement does not affect the classification of embeddings $\lie{h}\hookrightarrow \lie{g}$.
If $\lie{g}$ has real rank one, then any NDD tuple $(\lie{g}, \lie{h}, \orbit{O})$ is not minimal by Corollary \ref{cor:SufficientSameRank}, so we omit the case.

Let $V\coloneq \KK^n$ be the natural representation of $\lie{g}$, where $\KK = \RR$, $\CC$ or $\HH$.
Then $\lie{g}$ is isomorphic to the Lie algebra of the isometry group of $V$ with respect to some non-degenerate bilinear or sesquilinear form $B(\cdot, \cdot)$.
Since $h$ is the coroot of the highest restricted root, $V$ has the eigenspace decomposition
\begin{align}
	V &= V_{-1} \oplus V_0 \oplus V_1, \\
	\dim_{\KK}(V_1) &= \begin{cases}
		2 & (\lie{g} = \lieSO(p, q)) \\
		1 & (\text{otherwise})
	\end{cases} \label{eqn:eigenspacehSU}
\end{align}
with respect to the action of $h$.
Then the restriction of $B(\cdot, \cdot)$ to $V_1\oplus V_{-1}$ is non-degenerate.

Set $W \coloneq (V^{\lie{h}})^\perp$.
Then $V = W \oplus V^{\lie{h}}$.
By \eqref{eqn:eigenspacehSU} and Lemma \ref{lem:FundamentalRep}, $W$ is irreducible over $\KK$.

\begin{lemma} \label{lem:EmbeddingSUtoSp}
	$\lie{g}$ is not isomorphic to $\lieSp(n, \RR)$.
\end{lemma}

\begin{proof}
	By Lemma \ref{lem:FundamentalRep}, $\KK = \RR$ and $\dim_{\RR}(V_1) = 1$ cannot occur.
	This implies that $\lie{g}$ is not isomorphic to $\lieSp(n, \RR)$.
\end{proof}

\begin{lemma} \label{lem:EmbeddingSUtoSU}
	Suppose that $\lie{g} = \lieU(p, q)$.
	Then one has $\lieNorm_{\lie{g}}(\lie{h}) = \lieU(1, q) \oplus \lieU(p-1, 0)$ or $\lieU(p, 1) \oplus \lieU(0, q-1)$ up to automorphisms of $\lie{g}$.
	In particular, $\lieNorm_{\lie{g}}(\lie{h})$ is a symmetric subalgebra of $\lie{g}$.
\end{lemma}

\begin{proof}
	We have
	\begin{align*}
		\lie{h} \subset \lieU(W) \oplus \lieU(V^{\lie{h}}) \subset \lieU(V) = \lie{g}.
	\end{align*}
	Since $\lieCent_{\lie{g}}(\lie{h})$ is compact, $B$ is definite on $V^{\lie{h}}$.
	By Lemma \ref{lem:FundamentalRep}, $W$ is isomorphic to $\CC^{k+1}$ or $(\CC^{k+1})^*$ as an $\lie{h}$-module over $\CC$.
	Hence we have $\lie{h} = \lieSU(W)$.
	Therefore, we obtain $\lieNorm_{\lie{g}}(\lie{h}) = \lieU(W) \oplus \lieU(V^{\lie{h}})$.
\end{proof}

\begin{lemma} \label{lem:EmbeddingSUtoSOstar}
	Suppose that $\lie{g} = \lieSO^*(2n)$.
	Then one has
	\begin{align*}
		\lieNorm_{\lie{g}}(\lie{h}) &= \lieU(1, n-1) \quad (n = k+1), \text{ or }\\
		\lieNorm_{\lie{g}}(\lie{h}) &= \lieU(1, n-2) \oplus \lieSO^*(2) \quad (n = k+2),
	\end{align*}
	up to automorphisms of $\lie{g}$.
\end{lemma}

\begin{proof}
	We have
	\begin{align*}
		\lie{h} \subset \lieSO^*(W) \oplus \lieSO^*(V^{\lie{h}}) \subset \lieSO^*(V) = \lie{g}.
	\end{align*}
	Since $\lieCent_{\lie{g}}(\lie{h})$ is compact, we have $\dim_{\HH}(V^{\lie{h}}) \leq 1$.
	By Lemma \ref{lem:FundamentalRep}, $W$ is isomorphic to one of the following representations:
	\begin{align*}
		\HH\otimes_{\CC} \CC^{k+1}, \quad \HH^3 \left(\simeq \bigwedge\nolimits^2 \CC^4 \right).
	\end{align*}
	If $W \simeq \HH^3$, then we have $n \leq 4$, and this contradicts the assumption $n \geq 5$.

	Assume that $W \simeq \HH\otimes_{\CC} \CC^{k+1}$.
	Then the embedding is
	\begin{align*}
		\lieSU(1, k) \simeq \lie{h} \hookrightarrow \lieSO^*(W) \hookrightarrow \lieSO^*(V).
	\end{align*}
	If $V^{\lie{h}} = 0$, then we have $n=k+1$ and $\lieNorm_{\lie{g}}(\lie{h}) \simeq \lieU(1, n-1)$, which is a symmetric subalgebra of $\lie{g}$.
	If $\dim_{\HH}(V^{\lie{h}}) = 1$, then we have $n = k+2$ and $\lieNorm_{\lie{g}}(\lie{h}) \simeq \lieU(1, n-2) \oplus \lieSO^*(2)$, which is a symmetric subalgebra of $\lieSO^*(W) \oplus \lieSO^*(V^{\lie{h}})\subset \lie{g}$.
	This shows the assertion.
\end{proof}

\begin{lemma} \label{lem:EmbeddingSUtoSO}
	Suppose that $\lie{g} = \lieSO(p, q)$.
	Then, up to automorphisms of $\lie{g}$, one has
	\begin{align*}
		\lieNorm_{\lie{g}}(\lie{h}) &= \lieU(p/2, 1) \oplus \lieSO(0, q - 2),  \\
		\lieNorm_{\lie{g}}(\lie{h}) &= \lieU(1, q/2) \oplus \lieSO(p - 2, 0) \quad (q\text{: even}), \text{ or } \\
		\lieNorm_{\lie{g}}(\lie{h}) &= \lieU(1, k) \oplus \lieSO(0, q - 2k) \quad (p=2, 2k \leq q).
	\end{align*}
\end{lemma}

\begin{proof}
	We have
	\begin{align*}
		\lie{h} \subset \lieSO(W) \oplus \lieSO(V^{\lie{h}}) \subset \lieSO(V) = \lie{g}.
	\end{align*}
	Since $\lieCent_{\lie{g}}(\lie{h})$ is compact, $B$ is definite on $V^{\lie{h}}$.
	Consider the representation $W\otimes_{\RR} \CC$ over $\CC$.
	Since $\dim_{\RR}(V_1) = 2$, Lemma \ref{lem:FundamentalRep} (2) implies that $W\otimes_{\RR} \CC$ is reducible or isomorphic to $\bigwedge\nolimits^2 \CC^{4}$.
	By Lemma \ref{lem:FundamentalRep} (5), $\bigwedge\nolimits^2 \CC^{4}$ has no real form, and hence $W\otimes_{\RR} \CC$ is reducible.
	In other words, $W$ has a complex structure.
	By Lemma \ref{lem:FundamentalRep} (2), $W$ is isomorphic to $\CC^{k+1}$ or $(\CC^{k+1})^*$.

	Since $((W \otimes_\RR W)^*)^{\lie{h}}$ is spanned by the symmetric bilinear form $B|_W$ and some skew-symmetric bilinear form, the complex structure on $W$ is compatible with $B|_W$.
	Hence we have $\lie{h} = \lieSU(W)$ and $(\lieSO(W), \lieU(W))$ is a symmetric pair.
	Therefore, we obtain $\lieNorm_{\lie{g}}(\lie{h}) = \lieU(W) \oplus \lieSO(V^{\lie{h}})$.
\end{proof}

We have thus proved non-minimality in the cases where $\lie{g}=\lieSp(n,\RR)$ or $\lie{g}=\lieSU(p,q)$.
The cases where $\lie{g}=\lieSO(p,q)$ or $\lie{g}=\lieSO^*(2n)$ will be treated in the next two subsections.

\subsection{Explicit computation for \texorpdfstring{$\lie{g} = \lieSO(p,q)$}{g=SO(p,q)}}

 We shall show that $(\lie{g}, \lie{h}, \orbit{O})$ is not minimal using Lemma \ref{lem:ReductionRealStronglyOrthogonal} when $\lie{g}$ is isomorphic to $\lieSO(p,q)$.
We check the assumptions of Lemma \ref{lem:ReductionRealStronglyOrthogonal} directly.
By Lemma \ref{lem:EmbeddingSUtoSO}, it is enough to consider the following cases.
\begin{enumerate}[label=(\alph*)]
	\item $(\lie{g}, \lieNorm_{\lie{g}}(\lie{h})) = (\lieSO(p, q), \lieU(p/2, 1)\oplus \lieSO(0, q-2))$ ($p\geq 4, q \geq 4$, $p$: even)
	\item $(\lie{g}, \lieNorm_{\lie{g}}(\lie{h})) = (\lieSO(p, 3), \lieU(p/2, 1))$ ($p\geq 4$, $p$: even)
	\item $(\lie{g}, \lieNorm_{\lie{g}}(\lie{h})) = (\lieSO(2, q), \lieU(1, k)\oplus \lieSO(0, q-2k))$ ($q > 2k \geq 4$)
\end{enumerate}
Note that, if $q = 2k$ in case (c), then $(\lie{g}, \lieNorm_{\lie{g}}(\lie{h}))$ is a symmetric pair and hence $(\lie{g}, \lie{h}, \orbit{O})$ is not minimal for any $\orbit{O}$ by Lemma \ref{lem:SufficientSymmetric}.

Let $V \coloneq \RR^{p+q}$ with the standard inner product of signature $(p, q)$ ($p\geq 4, q \geq 3$, $p$: even).
Let $\set{e_1, \ldots, e_{p+q}}$ be the standard basis of $V$.
Set
\begin{align*}
	V_1 &\coloneq \spn{\RR}{e_1 + e_{p+q}, e_2 - e_{p+q-1}}, \\
	V_{-1} &\coloneq \spn{\RR}{e_1 - e_{p+q}, e_2 + e_{p+q-1}}, \\
	W_0 &\coloneq \spn{\RR}{e_3, \ldots, e_{p}}, \\
	W &\coloneq V_{-1} \oplus W_0 \oplus V_{1}, \\
	W' &\coloneq \spn{\RR}{e_{p+1}, \ldots, e_{p+q-2}}, \\
	V_0 &\coloneq W_0 \oplus W'.
\end{align*}
Then $V_1$ and $V_{-1}$ are isotropic subspaces of $V$ and we have $V = W \oplus W'$.
Define a complex structure on $W$ by
\begin{align*}
	J(e_{2i-1}) &= e_{2i} &&(1\leq i \leq p/2), \\
	J(e_{2i}) &= -e_{2i-1} &&(1\leq i \leq p/2), \\
	J(e_{p+q-1}) &= e_{p+q}, \\
	J(e_{p+q}) &= -e_{p+q-1}.
\end{align*}
$J$ is skew-symmetric on $W$ and hence $W$ is a complex vector space with a Hermitian form of signature $(p/2, 1)$.
Set $\lie{g} = \lieSO(V)$ and $\lie{h} = \lieSU(W)$.

In this setting, we may choose $h \in \lie{h}$ such that $h$ acts on $V_{\pm 1}$ by $\pm 1$ and on $V_0$ by $0$.
The Levi subalgebra $\lie{l} \coloneq \lieCent_{\lie{g}}(h)$ is $\lieGL(V_1) \oplus \lieSO(V_0)$.
We identify the $\lie{l}$-module $\lie{g}_1$ with $\Hom_{\RR}(V_{-1}, V_0)$.
Then we have
\begin{align*}
	\lie{g}_1 \cap \lie{h} &= \Hom_{\CC}(V_{-1}, W_0), \\
	\lie{g}_1 \cap \lie{h}^\perp &= \Hom_{\CC}(V_{-1}, \overline{W_0}) \oplus \Hom_{\RR}(V_{-1}, W').
\end{align*}

First, we consider the case of $q \geq 4$.
Define $T_1, \ldots, T_4 \in \Hom_{\RR}(V_{-1}, V_0)$ by
\begin{align*}
	T_1(e_1 - e_{p+q}) &= e_3 + e_{p+1}, & T_1(e_2 + e_{p+q-1}) &= 0, \\
	T_2(e_1 - e_{p+q}) &= e_{3} - e_{p+1}, & T_2(e_2 + e_{p+q-1}) &= 0, \\
	T_3(e_1 - e_{p+q}) &= 0, & T_3(e_2 + e_{p+q-1}) &= -e_4 - e_{p+2}, \\
	T_4(e_1 - e_{p+q}) &= 0, & T_4(e_2 + e_{p+q-1}) &= -e_4 + e_{p+2},
\end{align*}
and set $\lie{c} \coloneq \spn{\RR}{T_1, T_2, T_3, T_4}$.

\begin{lemma} \label{lem:SOpq4Basic}
	Let $a, b, c, d \in \RR$.
	\begin{enumparen}
		\item There exists a maximal abelian subspace $\lie{a}$ of $\lie{g}^{-\theta}$ containing $h$ such that $T_1, \ldots, T_4$ are root vectors with respect to $\lie{a}$ and the roots are strongly orthogonal.
		\item $aT_1 + bT_2 + cT_3 + dT_4 \in \lie{h}^\perp$ if and only if $a + b = c + d$.
		\item $aT_1 + bT_2 + cT_3 + dT_4 \in \lie{h}$ if and only if $a = b$, $c = d$ and $a + c = 0$.
		\item $\lie{c} = (\lie{c} \cap \lie{h}) \oplus (\lie{c} \cap \lie{h}^\perp)$.
	\end{enumparen}
\end{lemma}

\begin{proof}
	The assertions are straightforward from the definitions of $T_1, \ldots, T_4$.
\end{proof}

We show that there is no minimal tuple by applying Lemma \ref{lem:ReductionRealStronglyOrthogonal}.
The condition $\lie{c}=(\lie{c} \cap \lie{h}) \oplus (\lie{c} \cap \lie{h}^\perp)$ has already been verified in Lemma \ref{lem:SOpq4Basic}.
It remains to verify the other condition, namely $\overline{\orbit{O}} \cap \lie{h}^\perp \cap \lie{c} \neq \set{0}$.
Note that $\lie{c}$ is contained in $\Nilpotent(\lie{g}, H)$ since $\lie{c} \subset \lie{g}_1$.

\begin{lemma} \label{lem:SOpq4}
	Let $\orbit{O}$ be a nilpotent orbit in $\lie{g}$.
	Assume that there exists $T \in \lie{g}_1 \cap \lie{h}^\perp \cap \overline{\orbit{O}}$ such that the component of $T$ in $\Hom_{\CC}(V_{-1}, \overline{W_0})$ is non-zero.
	Then one of $T_1 + T_3$ and $T_1 + T_4$ belongs to $\overline{\orbit{O}} \cap \lie{h}^\perp$.
\end{lemma}

\begin{proof}
	By assumption, we have
	\begin{align*}
		T(e_1 - e_{p+q}) &= v + w, \\
		T(e_2 + e_{p+q-1}) &= -J(v) + w',
	\end{align*}
	for some $0\neq v \in W_0$, $w \in W'$, and $w' \in W'$.
	Recall that $\overline{\orbit{O}}$ is a $G$-stable closed cone.
	By the action of $\LieU(W_0)$ and scalar multiplication, we may assume that $v = e_3$.
	Then we have $T(e_2 + e_{p+q-1}) = -J(e_3) + w' = -e_4 + w'$.
	By the action of $\LieSO(W')$, we may assume that $w = ae_{p+1}$ and $w' = be_{p+1} + ce_{p+2}$ for some $a \in \RR_{\geq 0}$, $b, c \in \RR$.
	
	Define an element $h' \in \lieGL(V_1) \oplus \lieSO(V_0)$ by
	\begin{align*}
		h' \cdot (e_1 - e_{p+q}) &= e_1 - e_{p+q}, \\
		h'\cdot (e_2 + e_{p+q-1}) &= 2(e_2 + e_{p+q-1}), \\
		h'\cdot (e_3 + e_{p+1}) &= e_3 + e_{p+1}, \\
		h'\cdot (e_3 - e_{p+1}) &= -(e_3 - e_{p+1}), \\
		h'\cdot (-e_4 + e_{p+2}) &= 2\varepsilon (-e_4 + e_{p+2}), \\
		h'\cdot (-e_4 - e_{p+2}) &= -2\varepsilon(-e_4 - e_{p+2}), \\
		h'\cdot v &= 0 \quad (v \in (V_{-1} \oplus V_1 \oplus \spn{\RR}{e_3, e_4, e_{p+1}, e_{p+2}})^\perp),
	\end{align*}
	where $\varepsilon = 1$ if $c \geq 0$ and $\varepsilon = -1$ if $c < 0$.
	Set $T'\coloneq \lim_{t\to \infty} (\exp(th')T)$.
	Then we have
	\begin{align*}
		T'(e_1 - e_{p+q}) &= \lim_{t\to \infty} e^{-t}\exp(th')(e_3 + ae_{p+1})\\ 
		&= \frac{1+a}{2}(e_3 + e_{p+1}), \\
		T'(e_2 + e_{p+q-1}) &= \lim_{t\to \infty} e^{-2t}\exp(th')(-e_4 + be_{p+1} + ce_{p+2}) \\
		&= \begin{cases}
			\frac{1+c}{2}(-e_4 + e_{p+2}) & (c \geq 0), \\
			\frac{1-c}{2}(-e_4 - e_{p+2}) & (c < 0).
		\end{cases}
	\end{align*}
	Hence one of $T_1 + T_3$ and $T_1 + T_4$ belongs to $\LieGL(V_1)_o T' \subset \overline{\orbit{O}}$.
	This shows the assertion.
\end{proof}

\begin{theorem} \label{thm:SOpq4}
	If $\lie{g} \simeq \lieSO(p, q)$ ($p\geq 4, q \geq 4$, $p$: even) and $\lie{h} \simeq \lieSU(p/2, 1)$, then $(\lie{g}, \lie{h}, \orbit{O})$ is not minimal for any $\orbit{O}$.
\end{theorem}

\begin{proof}
	To apply Lemma \ref{lem:ReductionRealStronglyOrthogonal}, we shall show that $\lie{c} \cap \lie{h}^\perp \cap \overline{\orbit{O}} \neq \set{0}$.
	If the assumption of Lemma \ref{lem:SOpq4} holds, then we have nothing to prove.
	Assume that there exists $0\neq T \in \lie{g}_1 \cap \lie{h}^\perp \cap \overline{\orbit{O}}$ such that $T\in \Hom_{\RR}(V_{-1}, W')$.

	Recall that $\Hom_{\RR}(V_{-1}, W')$ is a $\LieGL(V_1)_o\times \LieSO(W')$-stable subspace in $\lie{g}_1\cap \lie{h}^\perp$ and the bilinear form on $W'$ is definite.
	Hence there exists $g \in \LieGL(V_1)_o\times \LieSO(W')$ such that $(gT)(e_1 - e_{p+q}) = ae_{p+1}$ and $(gT)(e_2 + e_{p+q-1}) = be_{p+2}$ for some $a, b \in \RR$.
	Then we have $0\neq gT \in \lie{c} \cap \lie{h}^\perp \cap \overline{\orbit{O}}$.
	We have shown the theorem.
\end{proof}

Next, we assume that $q = 3$.
In this case, Lemma \ref{lem:SOpq4} fails since $\dim(W') = 1$.
Instead, we prove an analogous lemma and apply Lemma \ref{lem:ReductionRealStronglyOrthogonal}.
Since the argument follows almost the same outline, we omit the details.

Define $T_1, \ldots, T_3 \in \Hom_{\RR}(V_{-1}, V_0)$ by
\begin{align*}
	T_1(e_1 - e_{p+q}) &= e_3 + e_{p+1}, & T_1(e_2 + e_{p+q-1}) &= 0, \\
	T_2(e_1 - e_{p+q}) &= e_{3} - e_{p+1}, & T_2(e_2 + e_{p+q-1}) &= 0, \\
	T_3(e_1 - e_{p+q}) &= 0, & T_3(e_2 + e_{p+q-1}) &= -e_4,
\end{align*}
and set $\lie{c} \coloneq \spn{\RR}{T_1, T_2, T_3}$.
Then $\lie{c}$ satisfies properties analogous to those in Lemma \ref{lem:SOpq4Basic}.
In particular, there exists a maximal abelian subspace $\lie{a}$ of $\lie{g}^{-\theta}$ containing $h$ such that
$T_1, \ldots, T_3$ are root vectors with respect to $\lie{a}$ and the roots are strongly orthogonal.

For $0 \neq T \in \lie{g}_1 \cap \lie{h}^\perp$, there are four possibilities:
\begin{enumparen}
	\item $B$ is positive definite on $\Im(T)$.
	\item $B$ is negative definite on $\Im(T)$ and $\dim(\Im(T)) = 1$.
	\item $B$ has signature $(1, 1)$ on $\Im(T)$.
	\item $B$ is degenerate and positive semi-definite on $\Im(T)$.
\end{enumparen}
We have $\dim(\Im(T)) = 2$ if and only if the component of $T$ in $\Hom_{\CC}(V_{-1}, \overline{W_0})$ is non-zero.

\begin{lemma} \label{lem:SOpq3-degenerate}
	Let $\orbit{O}$ be a nilpotent orbit in $\lie{g}$.
	Assume that there exists $0 \neq T \in \lie{g}_1 \cap \lie{h}^\perp \cap \orbit{O}$ such that $\dim(\Im(T)) = 2$.
	Then one has $\lie{c} \cap \lie{h}^\perp \cap \orbit{O} \neq \set{0}$.
\end{lemma}

\begin{proof}
	Since the projection from $\Im(T)$ to $W'$ has a non-trivial kernel, there exists $0\neq v \in V_{-1}$ such that $T(v) \in W_0$.
	By the action of $\LieGL(V_1, J)\times \LieU(W_0) \subset \LieU(W)$, we may assume that $T(e_2 + e_{p+q-1}) = -e_4$, where $\LieGL(V_1, J)$ is the general linear group of the complex vector space $(V_1, J)$.
	Then we have $T(e_1 - e_{p+q}) = -J(e_4) + ae_{p+1} = e_3 + ae_{p+1}$ for some $a \in \RR$.
	This shows the assertion.
\end{proof}

\begin{theorem} \label{thm:SOpq3}
	If $\lie{g} \simeq \lieSO(p, 3)$ and $\lie{h} \simeq \lieSU(p/2, 1)$, then $(\lie{g}, \lie{h}, \orbit{O})$ is not minimal for any $\orbit{O}$.
\end{theorem}

\begin{proof}
	The proof is the same as the proof of Theorem \ref{thm:SOpq4} except that we use Lemma \ref{lem:SOpq3-degenerate} instead of Lemma \ref{lem:SOpq4}.
\end{proof}

We shall consider the remaining case where $\lie{g} \simeq \lieSO(2, q)$ ($q \geq 5$).
Let $V \coloneq \RR^{2+q}$ with the standard inner product of signature $(2, q)$ ($q \geq 5$).
Let $2\leq k < q/2$.
Let $\set{e_1, \ldots, e_{2+q}}$ be the standard basis of $V$.
Set
\begin{align*}
	V_1 &\coloneq \spn{\RR}{e_1 + e_{2+q}, e_2 - e_{2+q-1}}, \\
	V_{-1} &\coloneq \spn{\RR}{e_1 - e_{2+q}, e_2 + e_{2+q-1}}, \\
	W_0 &\coloneq \spn{\RR}{e_3, \ldots, e_{2k}}, \\
	W &\coloneq V_{-1} \oplus W_0 \oplus V_{1}, \\
	W' &\coloneq \spn{\RR}{e_{2k+1}, \ldots, e_{2+q-2}}, \\
	V_0 &\coloneq W_0 \oplus W'.
\end{align*}
Then $V_1$ and $V_{-1}$ are isotropic subspaces of $V$ and we have $V = W \oplus W'$.
Define a complex structure on $W$ by
\begin{align*}
	J(e_{2i-1}) &= e_{2i} &&(1\leq i \leq k), \\
	J(e_{2i}) &= -e_{2i-1} &&(1\leq i \leq k), \\
	J(e_{2+q-1}) &= e_{2+q}, \\
	J(e_{2+q}) &= -e_{2+q-1}.
\end{align*}
$J$ is skew-symmetric on $W$ and hence $W$ is a complex vector space with a Hermitian form of signature $(1, k)$.
Set $\lie{g} = \lieSO(V)$ and $\lie{h} = \lieSU(W)$.

Take $h \in \lie{h}$ and identify the $\lie{l}$-module $\lie{g}_1$ with $\Hom_{\RR}(V_{-1}, V_0)$ as in the previous cases.
Define $T_1, \ldots, T_3 \in \Hom_{\RR}(V_{-1}, V_0)$ by
\begin{align*}
	T_1(e_1 - e_{2+q}) &= e_3, & T_1(e_2 + e_{2+q-1}) &= 0, \\
	T_2(e_1 - e_{2+q}) &= 0, & T_2(e_2 + e_{2+q-1}) &= e_4, \\
	T_3(e_1 - e_{2+q}) &= e_{2k+1}, & T_3(e_2 + e_{2+q-1}) &= 0
\end{align*}
and set $\lie{c} \coloneq \spn{\RR}{T_1, T_2}$.
Then $\lie{c}$ satisfies properties analogous to those in Lemma \ref{lem:SOpq4Basic}.
In particular, there exists a maximal abelian subspace $\lie{a}$ of $\lie{g}^{-\theta}$ containing $h$ such that
$T_1, \ldots, T_3$ are root vectors with respect to $\lie{a}$ and $\set{T_1, T_2}$ is a $\lie{q}$-orthogonal system defined in \ref{subsect:ReductionStronglyOrthogonal}.

\begin{lemma} \label{lem:SOpq2-nondegenerate}
	Let $T \in \lie{g}_1$ such that $\dim(\Im(T)) = 2$.
	Then there exists $g \in G$ such that $\Ad(g)T \in \lie{c} \cap \lie{h}^\perp$.
\end{lemma}

\begin{proof}
	Since the bilinear form on $V_0$ is definite, the element $g$ in the assertion can be taken from $\LieGL(V_1)_o\times \LieSO(V_0)$.
	Note that, for any $S \in \lie{c}$, we have $S \in \lie{h}^\perp$ if and only if $S \in \RR(T_1 - T_2)$.
\end{proof}

\begin{lemma} \label{lem:SOpq2-degenerate}
	Let $T \in \lie{g}_1$ such that $\dim(\Im(T)) = 1$.
	Then there exists $g \in G$ such that $\Ad(g)T \in \RR T_3 \subset \lie{h}^\perp$.
\end{lemma}

\begin{proof}
	Since the bilinear form on $V_0$ is definite, the element $g$ in the assertion can be taken from $\LieGL(V_1)_o\times \LieSO(V_0)$.
\end{proof}

\begin{theorem} \label{thm:SOpq2}
	If $\lie{g} \simeq \lieSO(2, q)$ ($q \geq 5$) and $\lie{h} \simeq \lieSU(1, k)$ ($2\leq k\leq q/2$), then $(\lie{g}, \lie{h}, \orbit{O})$ is not minimal for any $\orbit{O}$.
\end{theorem}

\begin{proof}
	If $2k = q$, then $(\lie{g}, \lieNorm_{\lie{g}}(\lie{h}))$ is a symmetric pair and hence $(\lie{g}, \lie{h}, \orbit{O})$ is not minimal for any $\orbit{O}$ by Lemma \ref{lem:SufficientSymmetric}.
	If $\orbit{O}$ contains an element $T \in \lie{g}_1$ with $\rank(T) = 2$, the assertion follows from Lemmas \ref{lem:ReductionRealStronglyOrthogonal} and \ref{lem:SOpq2-nondegenerate}.
	If not, the assertion follows from Lemmas \ref{lem:SufficientMaxAbelian} and \ref{lem:SOpq2-degenerate}.
\end{proof}

\subsection{Minimality} \label{subsect:AssociatedPair}

Let $(\lie{g}, \lie{h})$ be a pair of type $BC_1$ with $\lie{h} \simeq \lieSU(1, k)$ ($k \geq 2$).
We shall give a necessary condition for minimality.

Fix $h \in \lie{h}^{-\theta}$ as in Definition \ref{def:typeBC}.
Let $\sigma$ be the involution of $\lie{g}$ such that $\lie{g}^\sigma = \lie{g}_{-2}\oplus\lie{g}_0 \oplus \lie{g}_2 \subset \lie{g}$.
Then $\lie{h}$ is $\sigma$-stable and we have $\lie{g}^{\theta\sigma} \supset \lie{g}^{-\theta, -\sigma} = (\lie{g}_{-1}\oplus \lie{g}_1) \cap \lie{g}^{-\theta}$.

Fix an element $Z \in \lie{g}_2$ such that $\set{h, Z, -\theta(Z)}$ forms an $\lie{sl}_2$-triple.
Then $Z + \theta(Z) \in \lie{k}^\sigma$.
Set $J\coloneq Z + \theta(Z)$.
Since $\dim_\RR(\lie{g}_2) = 1$, we have
\begin{align}
	\lie{k}^\sigma = ((\lie{g}_2 \oplus \lie{g}_{-2}) \cap \lie{k}) \oplus (\lie{g}_0\cap \lie{k}) = \RR J \oplus (\lie{k}\cap \lie{g}_0). \label{eqn:KsigmaDecomp}
\end{align}
This implies that $J \in \lieCent(\lie{k}^\sigma)$.

\begin{lemma}\label{lem:ComplexStructure}
	$\ad(J)$ gives a $\lie{k}^{\sigma}$-invariant complex structure on $\lie{g}_1\oplus \lie{g}_{-1}$.
\end{lemma}

\begin{proof}
	The invariance of $\ad(J)$ follows from $J \in \lieCent(\lie{k}^\sigma)$.
	Since $[Z, \lie{g}_1] = 0$ and $[\theta(Z), \lie{g}_{-1}] = 0$, we have
	\begin{align*}
		\ad(J)^2(X) &= [Z, [\theta(Z), X]] = [[Z, \theta(Z)], X] = [-h, X] = -X
	\end{align*}
	for any $X \in \lie{g}_1$.
	Similarly, we have $\ad(J)^2(Y) = -Y$ for any $Y \in \lie{g}_{-1}$.
	Hence $\ad(J)$ defines a complex structure on $\lie{g}_1 \oplus \lie{g}_{-1}$.
\end{proof}

Take a maximal torus $\lie{t}'$ in $(\lie{k}_H\cap \lie{g}_0) \oplus \lieCent_{\lie{g}}(\lie{h}) \subset \lie{k}^\sigma$.
Then $\lie{t}\coloneq \lie{t}'\oplus \RR J$ is a compact Cartan subalgebra of $\lie{g}$ by Lemma \ref{lem:EmbeddingSU}, and $\lie{t}$ is contained in $\lie{k}^{\sigma} \subset \lie{g}^{\theta\sigma}$.
Set $\lie{p} \coloneq \lie{g}^{-\theta, -\sigma}_{\CC}$ and let $\lie{p} = \lie{p}_+\oplus \lie{p}_-$ be the eigenspace decomposition of $\lie{p}$ with respect to the action of $\sqrt{-1}(Z+\theta(Z))$, where $\lie{p}_+$ and $\lie{p}_-$ are the eigenspaces with eigenvalues $1$ and $-1$, respectively.
Take a set of strongly orthogonal roots $\set{\gamma_1, \ldots, \gamma_r}$ in $\Delta(\lie{p}_+\cap (\lieC{h})^\perp, \lieC{t})$ with $r$ as large as possible.
Take root vectors $X_{\gamma_i} \in (\lieC{g})_{\gamma_i}$ for $1 \leq i \leq r$.

Set $\lie{c} \coloneq \RR h \oplus \lie{t'}$ and, for $1\leq i\leq r$, define $\gamma^{\pm}_i \in \lie{c}^*$ by
\begin{align*}
	\gamma^{+}_i(h) &= 1, \quad \gamma^{+}_i|_{\lie{t'}} = \gamma_i|_{\lie{t'}}, \\
	\gamma^{-}_i(h) &= -1, \quad \gamma^{-}_i|_{\lie{t'}} = \gamma_i|_{\lie{t'}}.
\end{align*}
Then we have
\begin{align*}
	\overline{(\lieC{g})_{\gamma^{\pm}_i}} = (\lieC{g})_{-\gamma^{\mp}_i}, \quad \theta((\lieC{g})_{\gamma^{\pm}_i}) = (\lieC{g})_{\gamma^{\mp}_i}.
\end{align*}
Let $\lieC{g}'$ be the subalgebra of $\lieC{g}$ generated by
\begin{align*}
	(\lieC{g})_{\gamma^{\pm}_i}, (\lieC{g})_{-\gamma^{\pm}_i} \ (1\leq i \leq r) \text{ and } \lieC{c},
\end{align*}
and set $\lie{g'} \coloneq \lie{g}\cap \lieC{g}'$, which is a real form of $\lieC{g}'$.
Then $\lie{g'}$ is a $\theta$-stable subalgebra of $\lie{g}$ containing $h$.

We shall discuss when $\lie{g'}$ satisfies the assumptions of Proposition \ref{prop:NonMinimalSubalgebra}.
Set
\begin{align*}
	\lie{a}' \coloneq \bigoplus_{i=1}^r \RR(X_{\gamma_i} + \overline{X_{\gamma_i}}) \subset \lie{g}^{-\theta, -\sigma}.
\end{align*}
Then $\lie{a'}$ is an abelian subspace of $\lie{g}^{-\theta, -\sigma}$, and we have
\begin{align}
	\lie{a'} \subset \lie{g'} \cap \lie{h}^\perp, \quad \ad(J)\lie{a'} = \bigoplus_{i=1}^r \RR\sqrt{-1}(X_{\gamma_i} - \overline{X_{\gamma_i}}) \subset \lie{g'} \cap \lie{h}^\perp. \label{eqn:SmallSubalgebraA}
\end{align}

\begin{lemma} \label{lem:SmallSubalgebraStronglyOrthogonal}
	$\lie{g'}$ satisfies the following properties.
	\begin{enumparen}
		\item $X_{\gamma_i} + \overline{X_{\gamma_i}}, \sqrt{-1}(X_{\gamma_i} - \overline{X_{\gamma_i}}) \in \lie{g'}$.
		\item Let $\lie{g}'_{ss}$ be the semisimple part of $\lie{g'}$. Then one has
		\begin{align*}
			r \leq \rank_{\RR}(\lie{g}'_{ss}) \leq \rank(\lie{g}'_{ss}) \leq r + 1.
		\end{align*}
		\item $\lie{g'} = (\lie{g'} \cap \lie{h}) \oplus (\lie{g'} \cap \lie{h}^\perp)$.
	\end{enumparen}
\end{lemma}

\begin{proof}
	The assertion (1) is clear from the definition of $\lie{g'}$.
	By definition, the weight $\gamma^+_i - \gamma^-_i$ does not depend on $i$.
	Hence we have
	\begin{align*}
		r = \dim_{\RR}(\lie{a'}) &\leq \rank_{\RR}(\lie{g}'_{ss}) \leq \rank(\lie{g}'_{ss}) \\
		&= \dim_{\RR}\mathrm{span}_{\RR}\set{\gamma^+_1, \ldots, \gamma^+_r, \gamma^-_1, \ldots, \gamma^-_r} \leq r + 1.
	\end{align*}
	This shows (2).

	Since $\lie{c}$ is a Cartan subalgebra contained in $\lieNorm_{\lie{g}}(\lie{h})$, any root space with respect to $\lieC{c}$ is contained in $\lieC{h}$ or $\lieC{h}^\perp$.
	This proves assertion (3).
\end{proof}

\begin{corollary} \label{cor:SmallG'}
	If $\lie{g'} = \lie{g}$, then one has
	\begin{align*}
		r \leq \rank_{\RR}(\lie{g}) \leq \rank(\lie{g}) \leq r + 1.
	\end{align*}
\end{corollary}

\begin{lemma} \label{lem:SufficientForSecondReduction}
	Let $\orbit{O}$ be a nilpotent orbit in $\lie{g}$ with $\overline{\orbit{O}} \cap \Nilpotent(\lie{h}^\perp, H) \neq \set{0}$.
	Assume that there exists a compact subgroup $K_0$ of $C_K(h)$ such that $\lie{g}_1$ is $K_0$-stable and $\exp(\RR \ad(J))\Ad(K_0) \lie{a'} \supset \lie{g}^{-\theta, -\sigma} \cap \lie{h}^\perp$.
	Then one has $\overline{\orbit{O}} \cap \lie{h}^\perp \cap \lie{g}'_1 \neq \set{0}$.
\end{lemma}

\begin{proof}
	By Lemma \ref{lem:ExistenceWeightVector}, there exists a non-zero element $X \in \overline{\orbit{O}} \cap \lie{g}_1$.
	Then, we have $X - \theta(X) \in \lie{g}^{-\theta, -\sigma} \cap \lie{h}^\perp$.
	By assumption, there exist $k \in K_0$ and $s \in \RR$ such that $\exp(s\ad(J))\Ad(k)(X - \theta(X)) \in \lie{a'}$.
	By \eqref{eqn:SmallSubalgebraA}, we have
	\begin{align*}
		\Ad(k)X &\in \exp(\RR \ad(J))\lie{a'} + \ad(h)\exp(\RR \ad(J))\lie{a'} \\
		&\subset \lie{a'} + \ad(J)(\lie{a'}) + \ad(h)\lie{a'} + \ad(h)\ad(J)(\lie{a'}) \subset \lie{g'} \cap \lie{h}^\perp.
	\end{align*}
	This implies that $\overline{\orbit{O}} \cap \lie{h}^\perp \cap \lie{g}'_1 \neq \set{0}$.
\end{proof}

\begin{proposition} \label{prop:NonMinimalSymmetric}
	Suppose that $\lie{g} \neq \lie{g'}$ and $(\lie{g}^{\theta\sigma}, \lieNorm_{\lie{g}}(\lie{h})^{\theta\sigma})$ is a symmetric pair.
	Then $(\lie{g}, \lie{h}, \orbit{O})$ is not minimal for any $\orbit{O}$.
\end{proposition}

\begin{proof}
	Let $\tau$ be an involution of $\lie{g}^{\theta\sigma}$ corresponding to $\lieNorm_{\lie{g}}(\lie{h})^{\theta\sigma}$.
	Note that $\lieNorm_{\lie{g}}(\lie{h})^{\theta\sigma}$ contains $\lie{t}$ and the center of $\lie{g}^{\theta\sigma}$.
	Then $\tau$ is unique, $\tau$ commutes with $\theta$ and $\sigma$, and the restriction of the form $(\cdot, \cdot)$ to $\lie{g}^{\theta\sigma}$ is $\tau$-invariant.

	Set $\lie{s} \coloneq \lie{g}^{\theta\sigma, \theta\tau}$.
	Since $\lieCent_{\lie{g}}(\lie{h})$ is compact, we have
	\begin{align*}
		\lie{s}^{-\theta} &= \lie{g}^{-\theta, -\sigma, -\tau} = \lie{g}^{\theta\sigma} \cap \lie{g}^{-\theta} \cap \lie{h}^\perp, \\
		\lie{t}&\subset \lie{g}^{\theta, \sigma, \tau} \subset \lie{s}.
	\end{align*}
	Hence $\lie{s}^{-\theta}$ has an $\lie{s}^{\theta}$-invariant complex structure and we have $\lie{a}' \subset \lie{s}^{-\theta}$.
	In other words, any non-compact simple factor of $\lie{s}$ is of Hermitian type.
	Since $\lie{a'}$ is constructed from a maximal set of strongly orthogonal roots, $\lie{a'}$ is a maximal abelian subspace of $\lie{s}^{-\theta}$ by \cite[Lemma 7.143]{Kn02}.

	Let $K_0$ be the analytic subgroup of $G$ with the Lie algebra $\lieNorm_{\lie{k}}(\lie{h})\cap \lie{g}_0 \subset \lie{g}^{\theta, \sigma, \tau}$, and set $K'\coloneq \exp(\RR J)K_0$.
	By \eqref{eqn:KsigmaDecomp}, we have
	\begin{align*}
		\lie{g}^{\theta, \sigma, \tau} = (\RR (Z+\theta(Z)) \oplus (\lie{k}\cap \lie{g}_0)) \cap \lieNorm_{\lie{g}}(\lie{h})
		= \RR J \oplus (\lieNorm_{\lie{k}}(\lie{h})\cap \lie{g}_0).
	\end{align*}
	Then $\lie{g}_1$ is $K_0$-stable and $K'$ has the Lie algebra $\lie{g}^{\theta, \sigma, \tau}$.
	Since $\lie{a'}$ is a maximal abelian subspace of $\lie{s}^{-\theta}$, we have
	\begin{align*}
		\Ad(K')\lie{a'} = \lie{s}^{-\theta} = \lie{g}^{-\theta, -\sigma} \cap \lie{h}^\perp.
	\end{align*}
	Therefore the assertion follows from Lemma \ref{lem:SufficientForSecondReduction}.
\end{proof}

\begin{proposition} \label{prop:NonMinimalFullRank}
	Suppose that $\lie{g} \neq \lie{g'}$ and $r = \rank_{\RR}(\lie{g}^{\theta\sigma})$.
	Then $(\lie{g}, \lie{h}, \orbit{O})$ is not minimal for any $\orbit{O}$.
\end{proposition}

\begin{proof}
	Let $K_0$ be the analytic subgroup of $G$ with the Lie algebra $\lie{k}\cap \lie{g}_0$, and set $K'\coloneq \exp(\RR J)K_0$.
	Then $\lie{g}_1$ is $K_0$-stable and $K'$ has the Lie algebra $\lie{k}^\sigma$ by \eqref{eqn:KsigmaDecomp}.
	By assumption, $\lie{a'}$ is a maximal abelian subspace of $\lie{g}^{-\theta, -\sigma}$.
	Hence we have
	\begin{align*}
		\Ad(K')\lie{a'} = \lie{g}^{-\theta, -\sigma} \supset \lie{g}^{-\theta, -\sigma} \cap \lie{h}^\perp.
	\end{align*}
	Therefore the assertion follows from Lemma \ref{lem:SufficientForSecondReduction}.
\end{proof}

The conditions in Propositions \ref{prop:NonMinimalSymmetric} and \ref{prop:NonMinimalFullRank} can be verified purely in terms of root systems.
These propositions are closely related to Proposition \ref{prop:ReductionRealStronglyOrthogonal}, but they require no ideal decomposition on $\lie{g}^{\theta\sigma}$ and therefore apply in a different range of cases.

\begin{lemma} \label{lem:NonMinimalSOstar}
	Suppose that $\lie{g} \simeq \lieSO^*(2n)$ ($n \geq 5$).
	Then there exists no minimal NDD tuple $(\lie{g}, \lie{h}, \orbit{O})$.
\end{lemma}

\begin{proof}
	If $(\lie{g}, \lieNorm_{\lie{g}}(\lie{h}))$ is a symmetric pair, then $(\lie{g}, \lie{h}, \orbit{O})$ is not minimal for any $\orbit{O}$ by Lemma \ref{lem:SufficientSymmetric}.
	By Lemma \ref{lem:EmbeddingSUtoSOstar}, we may assume that $\lieNorm_{\lie{g}}(\lie{h}) \simeq \lieU(1, n-2) \oplus \lieSO^*(2)$.
	We shall describe the embedding explicitly.

	Let $F = F_+ \oplus F_- \oplus F_0$ be a complex vector space with a Hermitian form $B$ of signature $(2, n - 2)$ whose restrictions to each of $F_+$ and $F_-$ have signature $(1, 0)$ and whose restriction to $F_0$ has signature $(0, n - 2)$.
	Identify $\lie{g}$ with the Lie algebra of the isometry group of $\HH\otimes_{\CC} F$.
	Fix a subspace $W_0\subset F_0$ such that $\dim(W_0) = n - 3$.
	Set $W'_0 \coloneq W_0^\perp \cap F_0$.

	We may choose the embedding $\lie{h} \hookrightarrow \lie{g}$ and $h \in \lie{h}$ such that
	\begin{align*}
		\lie{g}^{\sigma} &= \lieSO^*(\HH\otimes_{\CC} (F_{+}\oplus F_{-})) \oplus \lieSO^*(\HH\otimes_{\CC}F_0), \\
		\lie{g}^{\theta\sigma} &= \lieU(F), \quad \lie{g}^{\theta, \sigma} = \lieU(F_+ \oplus F_-) \oplus \lieU(F_0), \\
		\lie{h} &= \lieSU(F_{+} \oplus jF_- \oplus W_0).
	\end{align*}
	Then we have
	\begin{align*}
		\lieNorm_{\lie{g}}(\lie{h}) &= \lieU(F_{+} \oplus jF_-\oplus W_0) \oplus \lieSO^*(\HH\otimes_{\CC} W'_0), \\
		\lieNorm_{\lie{g}}(\lie{h})^{\theta\sigma} &= \lieU(F_{+} \oplus W_0) \oplus \lieU(F_-) \oplus \lieU(W'_0).
	\end{align*}
	$\lieNorm_{\lie{g}}(\lie{h})^{\theta\sigma}$ is a Levi subalgebra of $\lie{g}^{\theta\sigma}$.
	Hence we may take a coordinate system $\set{\varepsilon_1, \ldots, \varepsilon_n}$ of $\lieC{t}$ such that
	\begin{align*}
		\Delta(\lieC{g}^{\theta\sigma}, \lieC{t}) &= \set{\pm (\varepsilon_i - \varepsilon_j) : 1\leq i < j \leq n}, \\
		\Delta(\lie{p}_+\cap \lieC{h}, \lieC{t}) &= \set{\varepsilon_1 - \varepsilon_i : 3\leq i \leq n-1}, \\
		\Delta(\lie{p}_+\cap (\lieC{h})^\perp, \lieC{t}) &= \set{ \varepsilon_2 - \varepsilon_i : 3 \leq i \leq n-1} \cup \set{\varepsilon_1 - \varepsilon_n, \varepsilon_2 - \varepsilon_n}.
	\end{align*}
	Since $\set{\varepsilon_2 - \varepsilon_3, \varepsilon_1 - \varepsilon_n}$ is a maximal subset of strongly orthogonal roots in $\Delta(\lie{p}_+\cap (\lieC{h})^\perp, \lieC{t})$, we have $r = 2 = \rank_{\RR}(\lie{g}^{\theta\sigma}) < \rank(\lie{g}) - 1$.
	By Proposition \ref{prop:NonMinimalFullRank}, the assertion follows.
\end{proof}

Summarizing Lemmas \ref{lem:EmbeddingSUtoSp}, \ref{lem:EmbeddingSUtoSU}, \ref{lem:NonMinimalSOstar} and Theorems \ref{thm:SOpq4}, \ref{thm:SOpq3} and \ref{thm:SOpq2}, we obtain the following theorem.

\begin{theorem} \label{thm:NonMinimalClassical}
	If $\lie{g}$ is a classical simple Lie algebra and $\lie{h} \simeq \lieSU(1, k)$ ($k \geq 2$), then there exists no minimal NDD tuple $(\lie{g}, \lie{h}, \orbit{O})$.
\end{theorem}

\subsection{Embeddings of \texorpdfstring{$\lieSU(1, k)$}{su(1,k)}: Exceptional \texorpdfstring{$\lie{g}$}{g}}

We shall consider the case where $\lie{g}$ is an exceptional Lie algebra.
We classify the embeddings of $\lieSU(1, k)$ into $\lie{g}$ for which the pair $(\lie{g}, \lie{h})$ is of type $BC_1$ (Definition \ref{def:typeBC}).
The classification is obtained by implementing the following straightforward algorithm on a computer.
The source code and the resulting classification data are available in \cite{Ki26-code}.
Our classification is up to isomorphism of pairs $(\lie{g}, \lie{h})$.

Before giving the algorithm, we shall state some observations on the embedding of $\lie{h} \simeq \lieSU(1, k)$ into $\lie{g}$.
Let $(\lie{g}, \lie{h})$ be a pair of type $BC_1$ and $\lie{h} \simeq \lieSU(1, k)$.
Let $h \in \lie{h}^{-\theta}$ be as in Definition \ref{def:typeBC} and $\sigma$ be the involution of $\lie{g}$ such that $\lie{g}^\sigma = \lie{g}_{-2} \oplus \lie{g}_0 \oplus \lie{g}_2$.
Take a $\sigma$-stable maximal torus $\lie{t}$ of $\lieNorm_{\lie{g}}(\lie{h})^\theta$.
Then we have
\begin{align*}
	\lie{t} = (\lie{t}\cap (\lie{g}_{-2} \oplus \lie{g}_2)) \oplus (\lie{t}\cap \lie{g}_0),
\end{align*}
and $\lie{t}$ is a compact Cartan subalgebra of $\lie{g}$ by Lemma \ref{lem:EmbeddingSU}.
Note that
\begin{align*}
	\lieNorm_{\lie{g}}(\lie{h})^\sigma = (\lie{g}_{-2}\oplus \RR h \oplus \lie{g}_2) \oplus (\lie{k}_H\cap \lie{g}_0) \oplus \lieCent_{\lie{g}}(\lie{h}),
\end{align*}
where the first summand is isomorphic to $\lieSU(1, 1) \simeq \lie{sl}(2, \RR)$.

Each root in $\Delta(\lieC{g}, \lieC{t})$ is either compact or non-compact.
Note that one of the Vogan diagrams of $\lie{h} \simeq \lieSU(1, k)$ is given by the following diagram.
\begin{align}
	\dynkin[scale=2, labels={\alpha_1, \alpha_2, \alpha_3, \alpha_{k-1}, \alpha_k}]{A}{*oo.oo} \label{eqn:VoganSU}
\end{align}
In the following, we use the same labeling of simple roots for $\lie{h}$ as in the above diagram.

\begin{lemma} \label{lem:SimpleRootEmbeddingSU}
	Let $S \coloneq \set{\alpha_1, \ldots, \alpha_k}$ be a set of simple roots in $\Delta(\lieC{h}, \lieC{t})$ with the Vogan diagram \eqref{eqn:VoganSU}.
	\begin{enumparen}
		\item $S$ is a set of simple roots of a root subsystem of $\Delta(\lieC{g}, \lieC{t})$ of type $A_k$.
		\item $\alpha_1$ is the only non-compact root in $S$.
		\item Any root in $\Delta(\lieC{g}, \lieC{t})$ that is strongly orthogonal to all roots in $S$ is compact.
		\item Any root in $S$ (and hence any root in $\Delta(\lieC{h}, \lieC{t})$) is long in $\Delta(\lieC{g}, \lieC{t})$.
	\end{enumparen}
\end{lemma}

\begin{proof}
	(1) and (2) are clear.
	Since $\lieCent_{\lie{g}}(\lie{h})$ is a compact subalgebra of $\lie{g}$, (3) is satisfied.

	Set $\lie{t'} \coloneq \lie{t} \cap \lie{g}_0$.
	Then $\lie{t'}$ is a maximal torus in $(\lie{g}_0 \cap \lie{k}_H) \oplus \lieCent_{\lie{g}}(\lie{h})$, and $\lie{c} \coloneq \RR h \oplus \lie{t'}$ is a Cartan subalgebra of $\lie{g}$.
	By Lemma \ref{lem:G2RootSpace}, $(\lieC{g})_2$ is a root space of $\lieC{g}$ with respect to $\lieC{c}$.
	Let $\alpha' \in \lieC{c}^*$ be the root such that $(\lieC{g})_2 = (\lieC{g})_{\alpha'}$.
	Then $\alpha'$ is a real root and $2(\alpha', \beta)/(\alpha', \alpha') = \beta(h) \in \set{\pm 1, 0}$ for any $\beta \in \Delta(\lieC{g}, \lieC{c})\backslash\set{\pm \alpha'}$.
	Hence $\alpha'$ is a long root.

	Let $\psi$ be the Cayley transform (see \cite[Chapter VI.7]{Kn02}) by the real root $\alpha'$.
	Then we have $\lie{t} =\psi(\lieC{c}) \cap \lie{g}$ and $\alpha'_1\coloneq \alpha' \circ \psi^{-1}$ is a non-compact long root in $\Delta(\lieC{g}, \lieC{t})$.
	Since $\Delta(\lieC{h}, \lieC{t})$ is of type $A_k$, any root has the same length.
	Hence any root in $\Delta(\lieC{h}, \lieC{t})$ is long in $\Delta(\lieC{g}, \lieC{t})$.
	This shows (4).
\end{proof}

We shall consider the converse of Lemma \ref{lem:SimpleRootEmbeddingSU}.
For a while, let $\lie{g}$ be a simple real Lie algebra with a compact Cartan subalgebra $\lie{t} \subset \lie{g}^\theta$.

\begin{lemma} \label{lem:SimpleRootEmbeddingSUInv}
	Let $S = \set{\alpha_1, \ldots, \alpha_k}$ be a subset of $\Delta(\lieC{g}, \lieC{t})$ satisfying the conclusions of Lemma \ref{lem:SimpleRootEmbeddingSU}.
	Let $\lie{h}$ be the subalgebra whose root system $\Delta(\lieC{h}, \lieC{t})$ has $S$ as a set of simple roots.
	Then $(\lie{g}, \lie{h})$ is a pair of type $BC_1$.
\end{lemma}

\begin{proof}
	We shall verify the conditions (1)--(5) in Definition \ref{def:typeBC}.
	The conditions (1) and (2) are clear.
	By assumption (3) (in Lemma \ref{lem:SimpleRootEmbeddingSU}), $\lieCent_{\lie{g}}(\lie{h})$ is a compact subalgebra of $\lie{g}$.
	This is (3) in Definition \ref{def:typeBC}.

	Let $\varphi \in \Aut(\lieC{g})$ be the Cayley transform (see \cite[Chapter VI.7]{Kn02}) by the non-compact root $\alpha_1$, and set $\lie{c} \coloneq \varphi(\lieC{t}) \cap \lie{g}$.
	Then $\varphi$ normalizes $\lieC{h}$ and $\alpha'_1\coloneq \alpha_1 \circ \varphi^{-1}$ is a unique real root in $\Delta(\lieC{g}, \lieC{c})$ up to sign.
	Since $\alpha'_1$ is a long root, $2(\alpha'_1, \beta')/(\alpha'_1, \alpha'_1) \in \set{\pm 2, \pm 1, 0}$ for any $\beta' \in \Delta(\lieC{g}, \lieC{c})$, and $2(\alpha'_1, \beta')/(\alpha'_1, \alpha'_1) = 2$ only if $\beta' = \alpha'_1$.
	This shows the conditions (4) and (5).
\end{proof}

Our task is to classify the subsets $S\coloneq \set{\alpha_1, \ldots, \alpha_k} \subset \Delta(\lieC{g}, \lieC{t})$ satisfying the conditions in Lemma \ref{lem:SimpleRootEmbeddingSU}.

\begin{algorithm} \label{alg:EmbeddingSU}
	Find and fix a non-compact root $\alpha_1$ in $\Delta(\lieC{g}, \lieC{t})$.
	If there are two lengths of roots, choose $\alpha_1$ to be a long root.
	Let $\Delta = \set{\beta_1, \ldots, \beta_m}$ be an ordered set of compact (long) roots in $\Delta(\lieC{g}, \lieC{t})$.
	Consider $\Delta \cup \set{\alpha_1}$ as an undirected graph in which two vertices $\alpha$ and $\beta$ are joined by an edge labeled $(\alpha, \beta)$ whenever $(\alpha, \beta) \neq 0$.
	Then $S$ is identified with a path that starts at $\alpha_1$ and has negatively labeled edges; moreover, the vertices $\alpha_i$ and $\alpha_j$ are nonadjacent whenever $|i-j|>1$.
	Enumerate all such paths $S = \set{\alpha_1, \ldots, \alpha_k}$ ($k\geq 2$) by depth-first search, and verify that no non-compact root is strongly orthogonal to every root in $S$.
\end{algorithm}

\begin{theorem}
	Algorithm \ref{alg:EmbeddingSU} produces a list of embeddings of $\lieSU(1, k)$ into the simple real Lie algebra $\lie{g}$ for which the corresponding pairs satisfy the conditions in Definition \ref{def:typeBC}.
	The list may contain mutually isomorphic entries, but every such embedding is isomorphic to one appearing in the list.
\end{theorem}

\begin{proof}
	Lemma \ref{lem:SimpleRootEmbeddingSUInv} essentially completes the proof.
	It remains only to justify the constraint that $\alpha_1$ is fixed in Algorithm \ref{alg:EmbeddingSU}, and this follows from the following fact.
	By \cite[Chapter XII. Problem 21]{Kn01}, any two non-compact long roots are conjugate by the group
	\begin{align*}
		\set{g \in \Aut(\Delta(\lieC{g}, \lieC{t})): g\text{ preserves compact roots}}.
	\end{align*}
	This result follows from the fact that $\lieC{g}^{-\theta}$ is irreducible or the direct sum of two irreducible representations of $\lieC{k}$.
	This and Lemmas \ref{lem:SimpleRootEmbeddingSU} and \ref{lem:SimpleRootEmbeddingSUInv} prove the theorem.
\end{proof}

Algorithm \ref{alg:EmbeddingSU} often produces mutually isomorphic entries.
Therefore, from the resulting list, we need to choose a set of representatives under the action of the automorphism group $\Aut(\Delta(\lieC{g}, \lieC{t}))$ preserving compactness.
For simplicity, in our implementation, we first choose representatives under the Weyl group of the root system $\set{\alpha \in \Delta(\lieC{k}, \lieC{t}): (\alpha, \alpha_1) = 0}$ and then eliminate the remaining duplicates by comparing their Vogan diagrams.

In the following, we give the list of embeddings of $\lieSU(1, k)$.
Retain the notation $\lie{g}, \lie{h}, \sigma, \lie{t}$ from the beginning of this subsection.
As in Subsections \ref{subsect:ReductionStronglyOrthogonal} and \ref{subsect:AssociatedPair}, the pair $(\lie{g}^{\theta\sigma}, \lieNorm_{\lie{g}}(\lie{h})^{\theta\sigma})$ plays an important role in establishing the non-minimality of NDD tuples.
We describe the embeddings using Vogan diagrams of $\lie{g}$ and $\lie{g}^{\theta\sigma}$.
Fix a set of simple roots $\set{\beta_1, \ldots, \beta_l}$ of $\Delta(\lieCent_{\lieC{g}}(\lieC{h}), \lieC{t})$.
Then $\set{\alpha_1, \ldots, \alpha_k, \beta_1, \ldots, \beta_l}$ is a set of simple roots of $\Delta(\lieNorm_{\lieC{g}}(\lieC{h}), \lieC{t})$.
The classification of the pairs $(\lie{g},\lie{h})$ below yields the following propositions.

\begin{proposition}
	One has $k+l = \rank(\lie{g})$ or $\rank(\lie{g}) - 1$.
	If $k+l=\rank(\lie{g}) - 1$, then $\lieNorm_{\lieC{g}}(\lieC{h})$ is a Levi subalgebra of $\lieC{g}$ with one-dimensional center.
\end{proposition}
	
There are several ways to realize the Vogan diagram of $\lieNorm_{\lie{g}}(\lie{h})$ in that of $\lie{g}$.
We use the following convention.

\begin{proposition}
	Suppose that $\lie{g}$ is exceptional.
	There exists a unique positive system of $\Delta(\lieC{g}, \lieC{t})$ such that $\alpha_1$ is a lowest root and $\alpha_2, \ldots, \alpha_k, \beta_1, \ldots, \beta_l$ are simple roots.
\end{proposition}

\begin{example}
	We give a concrete example in which $\lie{g}$ is classical rather than exceptional.
	Let $(\lie{g}, \lie{h}) = (\lieSU(4,3), \lieSU(1,3))$.
	Then we have $\lieNorm_{\lie{g}}(\lie{h}) = \lie{s}(\lieU(1,3) \oplus \lieU(3, 0))$.
	The embedding is represented by the following Vogan diagram with an extended vertex $\alpha_1$.
	\begin{align*}
		\dynkin[scale=2, extended, labels={\alpha_1, \alpha_2, \alpha_3, , \beta_1, \beta_2, }, affine mark=*]{A}{oo*ooo}
	\end{align*}
\end{example}

To realize the Vogan diagram of $\lieNorm_{\lie{g}}(\lie{h})^{\theta\sigma}$ in that of $\lie{g}^{\theta\sigma}$, we consider the set of simple roots $\set{\alpha_1 + \alpha_2, \alpha_3, \ldots, \alpha_k}$ of $\Delta((\lieC{h})^{\theta\sigma}, \lieC{t})$.
Here $\alpha_1 + \alpha_2$ is a non-compact root.
Note that $\lie{h}^{\theta\sigma}$ is isomorphic to $\lie{s}(\lieU(1, k-1)\oplus \lieU(1))$.

\begin{proposition} \label{prop:EmbeddingVoganAssociatedPair}
	Suppose that $\lie{g}$ is exceptional.
	There exists a positive system of $\Delta((\lieC{g})^{\theta\sigma}, \lieC{t})$ such that $\alpha_1 + \alpha_2, \alpha_3, \ldots, \alpha_k, \beta_1, \ldots, \beta_l$ are simple roots.
\end{proposition}

\begin{example}
	Let $(\lie{g}, \lie{h}) = (\lieSU(4,3), \lieSU(1,3))$.
	Then we have $\lie{g}^{\theta\sigma} = \lie{s}(\lieU(1,2) \oplus \lieU(3, 1))$ and $\lieNorm_{\lie{g}}(\lie{h})^{\theta\sigma} = \lie{s}(\lieU(1, 2) \oplus \lieU(0, 1) \oplus \lieU(3, 0))$.
	The embedding is represented by the following Vogan diagram.
	\begin{align*}
		\dynkin[scale=2, labels={\alpha_1 + \alpha_2, \alpha_3}]{A}{*o}\quad 
		\dynkin[scale=2, labels={, \beta_1, \beta_2}]{A}{*oo}
	\end{align*}
\end{example}

In Proposition \ref{prop:EmbeddingVoganAssociatedPair}, the positive system may not be unique.
We choose one of them to minimize the number of non-compact roots and realize the Vogan diagram of $\lieNorm_{\lie{g}}(\lie{h})^{\theta\sigma}$ in that of $\lie{g}^{\theta\sigma}$.

\begin{proposition} \label{prop:EmbeddingVoganAssociatedPair2}
	Among the Vogan diagrams provided by Proposition \ref{prop:EmbeddingVoganAssociatedPair}, there exists a unique Vogan diagram of $\Delta((\lieC{g})^{\theta\sigma}, \lieC{t})$ in which each connected component has a unique non-compact simple root.
\end{proposition}

Note that the Vogan diagrams above have information on the semisimple parts of $\lie{g}^{\theta\sigma}$, $\lieNorm_{\lie{g}}(\lie{h})$ and $\lieNorm_{\lie{g}}(\lie{h})^{\theta\sigma}$.
However, since they contain the compact Cartan subalgebra $\lie{t}$, we can recover the three subalgebras from the Vogan diagrams.

\begin{theorem} \label{thm:NonMinimalExceptional}
	If $(\lie{g}, \lie{h})$ is a pair of type $BC_1$ with $\lie{h} \simeq \lieSU(1, k)$ and with exceptional $\lie{g}$, then $(\lie{g}, \lie{h}, \orbit{O})$ is not minimal for any $\orbit{O}$.
\end{theorem}

We shall show Theorem \ref{thm:NonMinimalExceptional} using Propositions \ref{prop:ReductionRealStronglyOrthogonal}, \ref{prop:NonMinimalSymmetric} and \ref{prop:NonMinimalFullRank}.
We use the same notation $r$ and $\lie{p}_+$ as in Subsection \ref{subsect:AssociatedPair}.

\subsubsection{\texorpdfstring{$\lie{g} = \lieE_{6(2)}$}{g=e6(2)}}


We have $\lie{g}^{\theta\sigma} \simeq \lieSU(3, 3) \oplus \lieSL(2, \RR)$.
Then $r$ in Subsection \ref{subsect:AssociatedPair} satisfies
\begin{align*}
	r \leq \rank_{\RR}(\lie{g}^{\theta\sigma}) = 4 < 5 = \rank(\lie{g}) - 1.
\end{align*}
There are two embeddings of $\lieSU(1, k)$ into $\lieE_{6(2)}$.

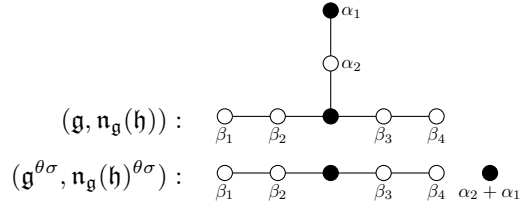
\begin{align*}
	(\lie{g}, \lieNorm_{\lie{g}}(\lie{h})): \quad & \dynkin[scale=2, extended, labels={\alpha_1, \beta_1, \alpha_2, \beta_2, , \beta_3, \beta_4}, affine mark=*]{E}{ooo*oo} \\
	(\lie{g}^{\theta\sigma}, \lieNorm_{\lie{g}}(\lie{h})^{\theta\sigma}): \quad & \dynkin[scale=2, labels={\beta_1, \beta_2, , \beta_3, \beta_4}]{A}{oo*oo} \dynkin[scale=2, labels={\alpha_2 + \alpha_1}]{A}{*}
\end{align*}

In the first embedding, we have $(\lie{g}^{\theta\sigma}, \lieNorm_{\lie{g}}(\lie{h})^{\theta\sigma}) \simeq (\lieSU(3, 3) \oplus \lieSL(2, \RR), \lie{s}(\lieU(3)\oplus \lieU(3)) \oplus \lieSL(2, \RR))$, which is a symmetric pair.
The assumptions of Proposition \ref{prop:NonMinimalSymmetric} hold.

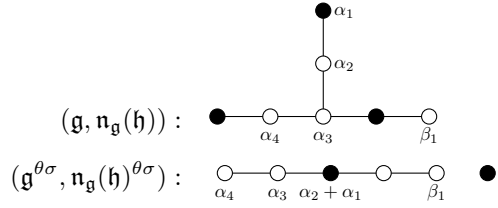
\begin{align*}
	(\lie{g}, \lieNorm_{\lie{g}}(\lie{h})): \quad & \dynkin[scale=2, extended, labels={\alpha_1, , \alpha_2, \alpha_4, \alpha_3, , \beta_1}, affine mark=*]{E}{*ooo*o} \\
	(\lie{g}^{\theta\sigma}, \lieNorm_{\lie{g}}(\lie{h})^{\theta\sigma}): \quad & \dynkin[scale=2, labels={\alpha_4, \alpha_3, \alpha_2 + \alpha_1, , \beta_1}]{A}{oo*oo} \quad \dynkin[scale=2, labels={}]{A}{*}
\end{align*}

In the second embedding, we have $(\lie{g}^{\theta\sigma}, \lieNorm_{\lie{g}}(\lie{h})^{\theta\sigma}) \simeq (\lieSU(3, 3) \oplus \lieSL(2, \RR), \lie{s}(\lieU(3, 1)\oplus \lieU(2)) \oplus \lieSO(2))$, which is a symmetric pair.
The assumptions of Proposition \ref{prop:NonMinimalSymmetric} hold.

\subsubsection{\texorpdfstring{$\lie{g} = \lieE_{6(-14)}$}{g=e6(-14)}}


We have $\lie{g}^{\theta\sigma} \simeq \lieSO^*(10) \oplus \lieSO(2)$.
Then $r$ in Subsection \ref{subsect:AssociatedPair} satisfies
\begin{align*}
	r \leq \rank_{\RR}(\lie{g}^{\theta\sigma}) = 2 < 5 = \rank(\lie{g}) - 1.
\end{align*}
There is one embedding of $\lieSU(1, k)$ into $\lieE_{6(-14)}$.

\begin{align*}
	(\lie{g}, \lieNorm_{\lie{g}}(\lie{h})): \quad & \dynkin[scale=2, extended, labels={\alpha_1, , \alpha_2, \alpha_4, \alpha_3, , \beta_1}, affine mark=*]{E}{*ooooo} \\
	(\lie{g}^{\theta\sigma}, \lieNorm_{\lie{g}}(\lie{h})^{\theta\sigma}): \quad & \dynkin[scale=2, labels={\beta_1, , \alpha_3, \alpha_4, \alpha_2 + \alpha_1}, label directions={,,right}]{D}{oooo*}
\end{align*}

In the embedding, we have $(\lie{g}^{\theta\sigma}, \lieNorm_{\lie{g}}(\lie{h})^{\theta\sigma}) \simeq (\lieSO^*(10) \oplus \lieSO(2), \lieSO^*(6)\oplus \lieU(2) \oplus \lieSO(2))$.
To take a set of strongly orthogonal roots, we fix a coordinate system $\set{\varepsilon_1, \ldots, \varepsilon_6}$ of $\lieC{t}$ such that
\begin{align*}
	\Delta(\lieC{g}^{\theta\sigma}, \lieC{t}) &= \set{\pm \varepsilon_i \pm \varepsilon_j: 1\leq i < j \leq 5}, \\
	\Delta(\lieC{g}^{\theta, \sigma}, \lieC{t}) &= \set{\pm (\varepsilon_i - \varepsilon_j): 1\leq i < j \leq 5}, \\
	\beta_1 = \varepsilon_1 - \varepsilon_2, &\quad \alpha_3 = \varepsilon_3 - \varepsilon_4, \quad \alpha_4 = \varepsilon_4 - \varepsilon_5, \quad \alpha_2 + \alpha_1 = \varepsilon_4 + \varepsilon_5.
\end{align*}
Then we have
\begin{align*}
	\Delta(\lie{p}_+ \cap (\lieC{h})^\perp, \lieC{t}) = \set{\varepsilon_1 + \varepsilon_2} \cup \set{\varepsilon_i + \varepsilon_j: 1 \leq i < 3 \leq j \leq 5}.
\end{align*}
Since $\set{\varepsilon_1 + \varepsilon_3, \varepsilon_2 + \varepsilon_4}$ is a set of strongly orthogonal roots in $\Delta(\lie{p}_+ \cap (\lieC{h})^\perp, \lieC{t})$, we have $r = 2 = \rank_{\RR}(\lie{g}^{\theta\sigma})$.
The assumptions of Proposition \ref{prop:NonMinimalFullRank} hold.

\subsubsection{\texorpdfstring{$\lie{g} = \lieE_{7(7)}$}{g=e7(7)}}


We have $\lie{g}^{\theta\sigma} \simeq \lieSU(4, 4)$.
Then $r$ in Subsection \ref{subsect:AssociatedPair} satisfies
\begin{align*}
	r \leq \rank_{\RR}(\lie{g}^{\theta\sigma}) = 4 < 6 = \rank(\lie{g}) - 1.
\end{align*}
There are two embeddings of $\lieSU(1, k)$ into $\lieE_{7(7)}$.

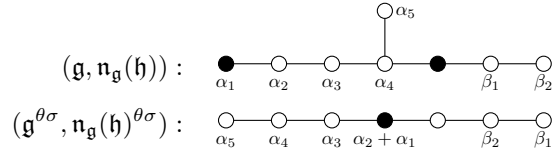
\begin{align*}
	(\lie{g}, \lieNorm_{\lie{g}}(\lie{h})): \quad & \dynkin[scale=2, extended, labels={\alpha_1, \alpha_2, \alpha_5, \alpha_3, \alpha_4, , \beta_1, \beta_2}, affine mark=*]{E}{oooo*oo} \\
	(\lie{g}^{\theta\sigma}, \lieNorm_{\lie{g}}(\lie{h})^{\theta\sigma}): \quad & \dynkin[scale=2, labels={\alpha_5, \alpha_4, \alpha_3, \alpha_2 + \alpha_1, , \beta_2, \beta_1}]{A}{ooo*ooo}
\end{align*}

In the first embedding, we have $(\lie{g}^{\theta\sigma}, \lieNorm_{\lie{g}}(\lie{h})^{\theta\sigma}) \simeq (\lieSU(4, 4), \lie{s}(\lieU(4, 1) \oplus \lieU(3)))$, which is a symmetric pair.
The assumptions of Proposition \ref{prop:NonMinimalSymmetric} hold.

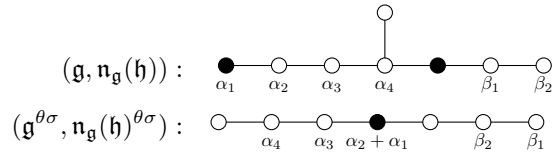
\begin{align*}
	(\lie{g}, \lieNorm_{\lie{g}}(\lie{h})): \quad & \dynkin[scale=2, extended, labels={\alpha_1, \alpha_2, , \alpha_3, \alpha_4, , \beta_1, \beta_2}, affine mark=*]{E}{oooo*oo} \\
	(\lie{g}^{\theta\sigma}, \lieNorm_{\lie{g}}(\lie{h})^{\theta\sigma}): \quad & \dynkin[scale=2, labels={, \alpha_4, \alpha_3, \alpha_2 + \alpha_1, , \beta_2, \beta_1}]{A}{ooo*ooo}
\end{align*}

In the second embedding, we have $(\lie{g}^{\theta\sigma}, \lieNorm_{\lie{g}}(\lie{h})^{\theta\sigma}) \simeq (\lieSU(4, 4), \lie{s}(\lieU(1, 0) \oplus \lieU(3, 1) \oplus \lieU(0, 3)))$.
Taking the standard realization of the root system of $\lieSL(8, \CC)$, we have
\begin{align*}
	\Delta(\lieC{g}^{\theta\sigma}, \lieC{t}) &= \set{\pm (\varepsilon_i - \varepsilon_j): 1\leq i < j \leq 8}, \\
	\Delta(\lie{p}_+ \cap (\lieC{h})^\perp, \lieC{t}) &= \set{\varepsilon_i - \varepsilon_j: 1\leq i \leq 4,  6\leq j \leq 8} \\
	&\cup \set{\varepsilon_1 - \varepsilon_5}.
\end{align*}
Then $\set{\varepsilon_1 - \varepsilon_5, \varepsilon_2 - \varepsilon_6, \varepsilon_3 - \varepsilon_7, \varepsilon_4 - \varepsilon_8}$ is a set of strongly orthogonal roots in $\Delta(\lie{p}_+ \cap (\lieC{h})^\perp, \lieC{t})$.
Hence we have $r = 4 = \rank_{\RR}(\lie{g}^{\theta\sigma})$.
The assumptions of Proposition \ref{prop:NonMinimalFullRank} hold.

\subsubsection{\texorpdfstring{$\lie{g} = \lieE_{7(-5)}$}{g=e7(-5)}}


We have $\lie{g}^{\theta\sigma} \simeq \lieSO^*(12) \oplus \lieSL(2,\RR)$.
Then $r$ in Subsection \ref{subsect:AssociatedPair} satisfies
\begin{align*}
	r \leq \rank_{\RR}(\lie{g}^{\theta\sigma}) = 4 < 6 = \rank(\lie{g}) - 1.
\end{align*}
There are three embeddings of $\lieSU(1, k)$ into $\lieE_{7(-5)}$.

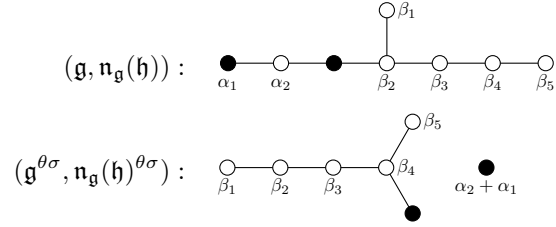
\begin{align*}
	(\lie{g}, \lieNorm_{\lie{g}}(\lie{h})): \quad & \dynkin[scale=2, extended, labels={\alpha_1, \alpha_2, \beta_1, , \beta_2, \beta_3, \beta_4, \beta_5}, affine mark=*]{E}{oo*oooo} \\
	(\lie{g}^{\theta\sigma}, \lieNorm_{\lie{g}}(\lie{h})^{\theta\sigma}): \quad & \dynkin[scale=2, labels={\beta_1, \beta_2, \beta_3, \beta_4, \beta_5, }, label directions={,,,right}]{D}{ooooo*}
	\dynkin[scale=2, labels={\alpha_2 + \alpha_1}]{A}{*}
\end{align*}
In the first embedding, we have $(\lie{g}^{\theta\sigma}, \lieNorm_{\lie{g}}(\lie{h})^{\theta\sigma}) \simeq (\lieSO^*(12) \oplus \lieSL(2,\RR), \lieU(6) \oplus\lieSL(2,\RR))$, which is a symmetric pair.
The assumptions of Proposition \ref{prop:NonMinimalSymmetric} hold.

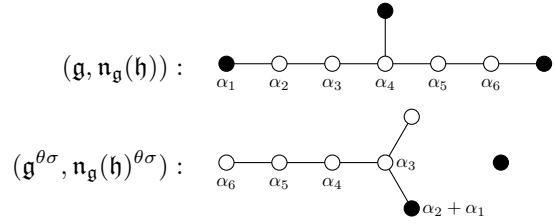
\begin{align*}
	(\lie{g}, \lieNorm_{\lie{g}}(\lie{h})): \quad & \dynkin[scale=2, extended, labels={\alpha_1, \alpha_2, , \alpha_3, \alpha_4, \alpha_5, \alpha_6, }, affine mark=*]{E}{o*oooo*} \\
	(\lie{g}^{\theta\sigma}, \lieNorm_{\lie{g}}(\lie{h})^{\theta\sigma}): \quad & \dynkin[scale=2, labels={\alpha_6, \alpha_5, \alpha_4, \alpha_3, , \alpha_2 + \alpha_1}, label directions={,,,right}]{D}{ooooo*}
	\dynkin[scale=2, labels={}]{A}{*}
\end{align*}
In the second embedding, we have $(\lie{g}^{\theta\sigma}, \lieNorm_{\lie{g}}(\lie{h})^{\theta\sigma}) \simeq (\lieSO^*(12) \oplus \lieSL(2,\RR), \lieU(5,1) \oplus\lieSO(2,\RR))$, which is a symmetric pair.
The assumptions of Proposition \ref{prop:NonMinimalSymmetric} hold.

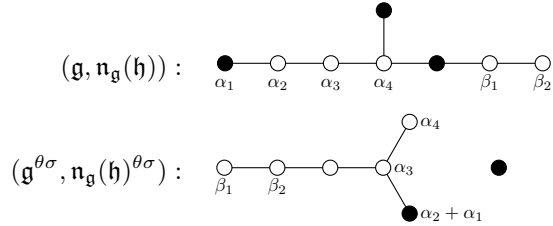
\begin{align*}
	(\lie{g}, \lieNorm_{\lie{g}}(\lie{h})): \quad & \dynkin[scale=2, extended, labels={\alpha_1, \alpha_2, , \alpha_3, \alpha_4, , \beta_1, \beta_2}, affine mark=*]{E}{o*oo*oo} \\
	(\lie{g}^{\theta\sigma}, \lieNorm_{\lie{g}}(\lie{h})^{\theta\sigma}): \quad & \dynkin[scale=2, labels={\beta_1, \beta_2, , \alpha_3, \alpha_4, \alpha_2 + \alpha_1}, label directions={,,,right}]{D}{ooooo*}
	\dynkin[scale=2, labels={}]{A}{*}
\end{align*}

In the third embedding, we have $(\lie{g}^{\theta\sigma}, \lieNorm_{\lie{g}}(\lie{h})^{\theta\sigma}) \simeq (\lieSO^*(12) \oplus \lieSL(2,\RR), \lieSO^*(6) \oplus \lieU(3, 0) \oplus\lieSO(2))$.
Taking the standard realizations of the root systems of $\lieSO(12,\CC)$ and $\lieSL(2,\CC)$, we have
\begin{align*}
	\Delta(\lieC{g}^{\theta\sigma}, \lieC{t}) &= \set{\pm \varepsilon_i \pm \varepsilon_j: 1\leq i < j \leq 6} \cup \set{\pm \varepsilon_7}, \\
	\Delta(\lie{p}_+ \cap (\lieC{h})^\perp, \lieC{t}) &= \set{\varepsilon_i + \varepsilon_j: 1\leq i < j \leq 3} \\
	&\cup \set{\varepsilon_i + \varepsilon_j: 1\leq i \leq 3, 4\leq j \leq 6} \cup \set{\varepsilon_7}.
\end{align*}
Then $\set{\varepsilon_1 + \varepsilon_4, \varepsilon_2 + \varepsilon_5, \varepsilon_3 + \varepsilon_6, \varepsilon_7}$ is a set of strongly orthogonal roots in $\Delta(\lie{p}_+ \cap (\lieC{h})^\perp, \lieC{t})$.
Hence we have $r = 4 = \rank_{\RR}(\lie{g}^{\theta\sigma})$.
The assumptions of Proposition \ref{prop:NonMinimalFullRank} hold.

\subsubsection{\texorpdfstring{$\lie{g} = \lieE_{7(-25)}$}{g=e7(-25)}}


We have $\lie{g}^{\theta\sigma} \simeq \lieE_{6(-14)} \oplus \lieSO(2)$.
Then $r$ in Subsection \ref{subsect:AssociatedPair} satisfies
\begin{align*}
	r \leq \rank_{\RR}(\lie{g}^{\theta\sigma}) = 2 < 6 = \rank(\lie{g}) - 1.
\end{align*}
There is one embedding of $\lieSU(1, k)$ into $\lieE_{7(-25)}$.

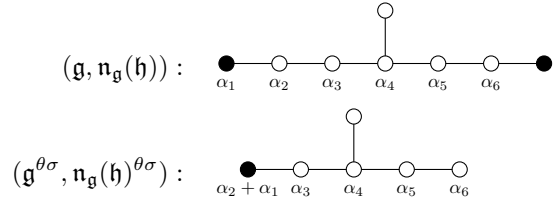
\begin{align*}
	(\lie{g}, \lieNorm_{\lie{g}}(\lie{h})): \quad & \dynkin[scale=2, extended, labels={\alpha_1, \alpha_2, , \alpha_3, \alpha_4, \alpha_5, \alpha_6, }, affine mark=*]{E}{oooooo*} \\
	(\lie{g}^{\theta\sigma}, \lieNorm_{\lie{g}}(\lie{h})^{\theta\sigma}): \quad & \dynkin[scale=2, labels={\alpha_2+\alpha_1, , \alpha_3, \alpha_4, \alpha_5, \alpha_6}]{E}{*ooooo}
\end{align*}

In the embedding, we have $(\lie{g}^{\theta\sigma}, \lieNorm_{\lie{g}}(\lie{h})^{\theta\sigma}) \simeq (\lieE_{6(-14)} \oplus \lieSO(2), \lieSU(1, 5) \oplus \lieSO(2) \oplus \lieSO(2))$.
Then the subalgebra $\lieNorm_{\lie{g}}(\lie{h})^{\theta\sigma}$ is contained in a symmetric subalgebra $(\lie{g}^{\theta\sigma})^{\tau}$ isomorphic to $\lieSU(1,5) \oplus \lieSL(2,\RR)\oplus \lieSO(2)$, where $\tau$ is an involution commuting with $\theta$.
By \cite[Table I]{OsSe84}, we have $(\lie{g}^{\theta\sigma})^{\theta\tau} \simeq \lieSO^*(10) \oplus \lieSO(2) \oplus \lieSO(2)$ and
\begin{align*}
	\Delta(\lie{p}_+ \cap (\lieC{h})^\perp, \lieC{t}) \supset \Delta(\lie{p}_+ \cap (\lieC{g}^{\theta\sigma})^{\theta \tau}, \lieC{t}).
\end{align*}
By the construction of a maximal set of strongly orthogonal roots (\cite[Lemma 7.143]{Kn02}), we have $2 = \rank_{\RR}((\lie{g}^{\theta\sigma})^{\theta\tau}) \leq r$.
Hence we obtain $r = 2 = \rank_{\RR}(\lie{g}^{\theta\sigma})$.
The assumptions of Proposition \ref{prop:NonMinimalFullRank} hold.

\subsubsection{\texorpdfstring{$\lie{g} = \lieE_{8(8)}$}{g=e8(8)}}


We have $\lie{g}^{\theta\sigma} \simeq \lieSO^*(16)$.
Then $r$ in Subsection \ref{subsect:AssociatedPair} satisfies
\begin{align*}
	r \leq \rank_{\RR}(\lie{g}^{\theta\sigma}) = 4 < 7 = \rank(\lie{g}) - 1.
\end{align*}
There are four embeddings of $\lieSU(1, k)$ into $\lieE_{8(8)}$.

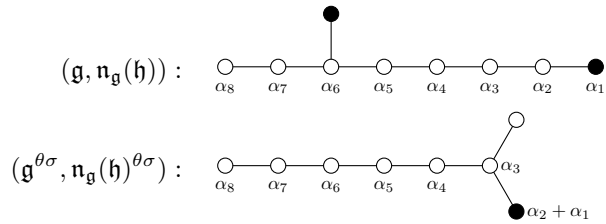
\begin{align*}
	(\lie{g}, \lieNorm_{\lie{g}}(\lie{h})): \quad & \dynkin[scale=2, extended, labels={\alpha_1, \alpha_8, , \alpha_7, \alpha_6, \alpha_5, \alpha_4, \alpha_3, \alpha_2}, affine mark=*]{E}{o*oooooo} \\
	(\lie{g}^{\theta\sigma}, \lieNorm_{\lie{g}}(\lie{h})^{\theta\sigma}): \quad & \dynkin[scale=2, labels={\alpha_8, \alpha_7, \alpha_6, \alpha_5, \alpha_4, \alpha_3, , \alpha_2+\alpha_1}, label directions={,,,,,right}]{D}{ooooooo*}
\end{align*}

In the first embedding, we have $(\lie{g}^{\theta\sigma}, \lieNorm_{\lie{g}}(\lie{h})^{\theta\sigma}) \simeq (\lieSO^*(16), \lieU(1, 7))$, which is a symmetric pair.
The assumptions of Proposition \ref{prop:NonMinimalSymmetric} hold.

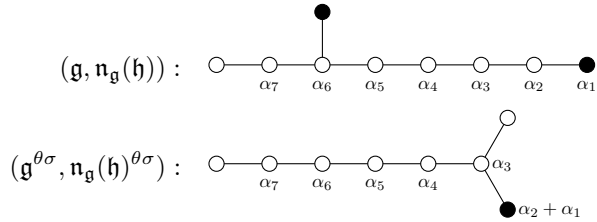
\begin{align*}
	(\lie{g}, \lieNorm_{\lie{g}}(\lie{h})): \quad & \dynkin[scale=2, extended, labels={\alpha_1, , , \alpha_7, \alpha_6, \alpha_5, \alpha_4, \alpha_3, \alpha_2}, affine mark=*]{E}{o*oooooo} \\
	(\lie{g}^{\theta\sigma}, \lieNorm_{\lie{g}}(\lie{h})^{\theta\sigma}): \quad & \dynkin[scale=2, labels={, \alpha_7, \alpha_6, \alpha_5, \alpha_4, \alpha_3, , \alpha_2+\alpha_1}, label directions={,,,,,right}]{D}{ooooooo*}
\end{align*}

In the second embedding, we have $(\lie{g}^{\theta\sigma}, \lieNorm_{\lie{g}}(\lie{h})^{\theta\sigma}) \simeq (\lieSO^*(16), \lieU(1, 6) \oplus \lieU(1))$.
Taking the standard realization of the root system of $\lieSO(16, \CC)$, we have
\begin{align*}
	\Delta(\lieC{g}^{\theta\sigma}, \lieC{t}) &= \set{\pm \varepsilon_i \pm \varepsilon_j: 1\leq i < j \leq 8},\\
	\Delta(\lie{p}_+ \cap (\lieC{h})^\perp, \lieC{t}) &= \set{\varepsilon_i + \varepsilon_j: 1\leq i < j \leq 7} \\
	&\cup \set{\varepsilon_1 + \varepsilon_8}.
\end{align*}
Then $\set{\varepsilon_1 + \varepsilon_8, \varepsilon_2 + \varepsilon_7, \varepsilon_3 + \varepsilon_6, \varepsilon_4 + \varepsilon_5}$ is a set of strongly orthogonal roots in $\Delta(\lie{p}_+ \cap (\lieC{h})^\perp, \lieC{t})$.
Hence we have $r = 4 = \rank_{\RR}(\lie{g}^{\theta\sigma})$.
The assumptions of Proposition \ref{prop:NonMinimalFullRank} hold.

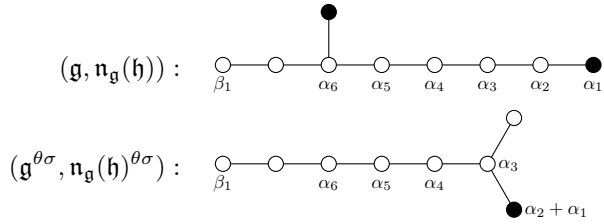
\begin{align*}
	(\lie{g}, \lieNorm_{\lie{g}}(\lie{h})): \quad & \dynkin[scale=2, extended, labels={\alpha_1, \beta_1, , , \alpha_6, \alpha_5, \alpha_4, \alpha_3, \alpha_2}, affine mark=*]{E}{o*oooooo} \\
	(\lie{g}^{\theta\sigma}, \lieNorm_{\lie{g}}(\lie{h})^{\theta\sigma}): \quad & \dynkin[scale=2, labels={\beta_1, , \alpha_6, \alpha_5, \alpha_4, \alpha_3, , \alpha_2+\alpha_1}, label directions={,,,,,right}]{D}{ooooooo*}
\end{align*}

In the third embedding, we have $(\lie{g}^{\theta\sigma}, \lieNorm_{\lie{g}}(\lie{h})^{\theta\sigma}) \simeq (\lieSO^*(16), \lieU(1, 5) \oplus \lieU(2))$.
Taking the standard realization of the root system of $\lieSO(16, \CC)$, we have
\begin{align*}
	\Delta(\lieC{g}^{\theta\sigma}, \lieC{t}) &= \set{\pm \varepsilon_i \pm \varepsilon_j: 1\leq i < j \leq 8},\\
	\Delta(\lie{p}_+ \cap (\lieC{h})^\perp, \lieC{t}) &= \set{\varepsilon_i + \varepsilon_j: 1\leq i < j \leq 7} \\
	&\cup \set{\varepsilon_1 + \varepsilon_8, \varepsilon_2 + \varepsilon_8}.
\end{align*}
Then $\set{\varepsilon_1 + \varepsilon_8, \varepsilon_2 + \varepsilon_7, \varepsilon_3 + \varepsilon_6, \varepsilon_4 + \varepsilon_5}$ is a set of strongly orthogonal roots in $\Delta(\lie{p}_+ \cap (\lieC{h})^\perp, \lieC{t})$.
Hence we have $r = 4 = \rank_{\RR}(\lie{g}^{\theta\sigma})$.
The assumptions of Proposition \ref{prop:NonMinimalFullRank} hold.

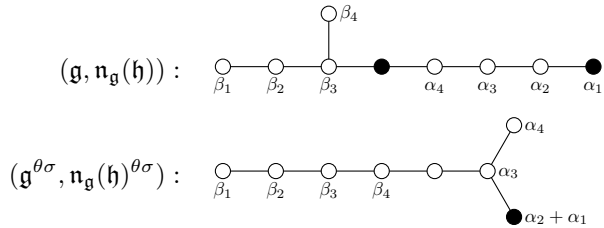
\begin{align*}
	(\lie{g}, \lieNorm_{\lie{g}}(\lie{h})): \quad & \dynkin[scale=2, extended, labels={\alpha_1, \beta_1, \beta_4, \beta_2, \beta_3, , \alpha_4, \alpha_3, \alpha_2}, affine mark=*]{E}{oooo*ooo} \\
	(\lie{g}^{\theta\sigma}, \lieNorm_{\lie{g}}(\lie{h})^{\theta\sigma}): \quad & \dynkin[scale=2, labels={\beta_1, \beta_2, \beta_3, \beta_4, , \alpha_3, \alpha_4, \alpha_2+\alpha_1}, label directions={,,,,,right}]{D}{ooooooo*}
\end{align*}

In the fourth embedding, we have $(\lie{g}^{\theta\sigma}, \lieNorm_{\lie{g}}(\lie{h})^{\theta\sigma}) \simeq (\lieSO^*(16), \lieU(5) \oplus \lieSO^*(6))$.
Taking the standard realization of the root system of $\lieSO(16, \CC)$, we have
\begin{align*}
	\Delta(\lieC{g}^{\theta\sigma}, \lieC{t}) &= \set{\pm \varepsilon_i \pm \varepsilon_j: 1\leq i < j \leq 8},\\
	\Delta(\lie{p}_+ \cap (\lieC{h})^\perp, \lieC{t}) &= \set{\varepsilon_i + \varepsilon_j: 1\leq i < j \leq 5} \\
	&\cup \set{\varepsilon_i + \varepsilon_j: 1\leq i \leq 5, 6\leq j \leq 8}.
\end{align*}
Then $\set{\varepsilon_1 + \varepsilon_8, \varepsilon_2 + \varepsilon_7, \varepsilon_3 + \varepsilon_6, \varepsilon_4 + \varepsilon_5}$ is a set of strongly orthogonal roots in $\Delta(\lie{p}_+ \cap (\lieC{h})^\perp, \lieC{t})$.
Hence we have $r = 4 = \rank_{\RR}(\lie{g}^{\theta\sigma})$.
The assumptions of Proposition \ref{prop:NonMinimalFullRank} hold.

\subsubsection{\texorpdfstring{$\lie{g} = \lieE_{8(-24)}$}{g=e8(-24)}}


We have $\lie{g}^{\theta\sigma} \simeq \lieE_{7(-25)} \oplus \lieSL(2,\RR)$.
Then $r$ in Subsection \ref{subsect:AssociatedPair} satisfies
\begin{align*}
	r \leq \rank_{\RR}(\lie{g}^{\theta\sigma}) = 4 < 7 = \rank(\lie{g}) - 1.
\end{align*}
The Vogan diagram of $\lie{g}^{\theta\sigma}$ is given by
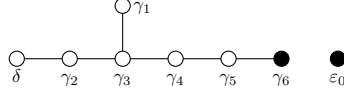
\begin{align*}
	\dynkin[scale=2, labels={\delta, \gamma_1, \gamma_2, \gamma_3, \gamma_4, \gamma_5, \gamma_6}]{E}{oooooo*}
	\quad \dynkin[scale=2, labels={\varepsilon_0}]{A}{*}
\end{align*}
Here the labels are given by a realization of the root system $\Delta(\lieC{g}^{\theta\sigma}, \lieC{t})$ in the Euclidean space $\bigoplus_{i=0}^8 \RR \varepsilon_i$ such that
\begin{align*}
	\Delta^+(\lieC{g}^{\theta\sigma}, \lieC{t}) &= \set{\varepsilon_0} \cup \set{\varepsilon_j \pm \varepsilon_i: 1\leq i < j \leq 6} \cup \set{\varepsilon_8 - \varepsilon_7}\\
	&\cup \set{\frac{1}{2}\sum_{i=1}^8 a_i \varepsilon_i: a_i = \pm 1, a_7=-1, a_8=1, a_1a_2\cdots a_6 = -1}, \\
	\Delta^+(\lieC{g}^{\theta, \sigma}, \lieC{t}) &= \set{\varepsilon_j \pm \varepsilon_i: 1\leq i < j \leq 5} \\
	& \cup \set{\frac{1}{2}\sum_{i=1}^8 a_i \varepsilon_i: a_i = \pm 1, a_6=a_7=-1,a_8=1,a_1a_2\cdots a_5 = 1},\\
	\gamma_i &= \varepsilon_i - \varepsilon_{i - 1} \quad (2\leq i \leq 6), \quad \gamma_1 = \varepsilon_2 + \varepsilon_1, \\
	\delta &= \frac{1}{2}(\varepsilon_8 - \varepsilon_7 - \varepsilon_6 - \varepsilon_5 - \varepsilon_4 - \varepsilon_3 - \varepsilon_2 + \varepsilon_1)
\end{align*}
See \cite[Chapter VI, \S 4.11]{Bo02}.
Then we have
\begin{align*}
	\Delta(\lie{p}_+, \lieC{t}) &= \set{\varepsilon_0} \cup \set{\varepsilon_6 \pm \varepsilon_i: 1\leq i \leq 5} \cup \set{\varepsilon_8 - \varepsilon_7} \\
	&\cup \set{\frac{1}{2}\sum_{i=1}^8 a_i \varepsilon_i: a_i = \pm 1, a_7=-1,a_6=a_8=1,a_1a_2\cdots a_5 = -1}.
\end{align*}

There are three embeddings of $\lieSU(1, k)$ into $\lieE_{8(-24)}$.

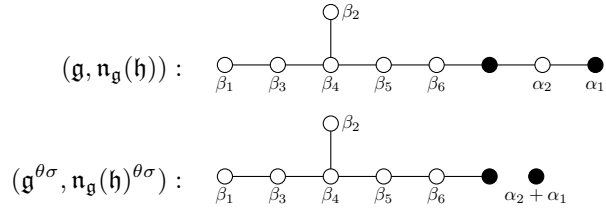
\begin{align*}
	(\lie{g}, \lieNorm_{\lie{g}}(\lie{h})): \quad & \dynkin[scale=2, extended, labels={\alpha_1, \beta_1, \beta_2, \beta_3, \beta_4, \beta_5, \beta_6, , \alpha_2}, affine mark=*]{E}{oooooo*o} \\
	(\lie{g}^{\theta\sigma}, \lieNorm_{\lie{g}}(\lie{h})^{\theta\sigma}): \quad & \dynkin[scale=2, labels={\beta_1, \beta_2, \beta_3, \beta_4, \beta_5, \beta_6, }]{E}{oooooo*}
	\dynkin[scale=2, labels={\alpha_2 + \alpha_1}]{A}{*}
\end{align*}

In the first embedding, we have $(\lie{g}^{\theta\sigma}, \lieNorm_{\lie{g}}(\lie{h})^{\theta\sigma}) \simeq (\lieE_{7(-25)} \oplus \lieSL(2,\RR), \lieE_{6(-78)} \oplus \lieSO(2) \oplus \lieSL(2,\RR))$, which is a symmetric pair.
The assumptions of Proposition \ref{prop:NonMinimalSymmetric} hold.

\begin{align*}
	(\lie{g}, \lieNorm_{\lie{g}}(\lie{h})): \quad & \dynkin[scale=2, extended, labels={\alpha_1, \beta_1, , , \alpha_6, \alpha_5, \alpha_4, \alpha_3, \alpha_2}, affine mark=*]{E}{o**ooooo} \\
	(\lie{g}^{\theta\sigma}, \lieNorm_{\lie{g}}(\lie{h})^{\theta\sigma}): \quad & \dynkin[scale=2, labels={\beta_1, \alpha_6, , \alpha_5, \alpha_4, \alpha_3, \alpha_2 + \alpha_1}]{E}{oooooo*}
	\quad \dynkin[scale=2, labels={}]{A}{*}
\end{align*}

In the second embedding, we have $(\lie{g}^{\theta\sigma}, \lieNorm_{\lie{g}}(\lie{h})^{\theta\sigma}) \simeq (\lieE_{7(-25)} \oplus \lieSL(2,\RR), \lieSU(1, 5) \oplus \lieU(1) \oplus \lieSU(2) \oplus \lieSO(2))$.
Then we have
\begin{align*}
	&\Delta(\lie{p}_+ \cap (\lieC{h})^\perp, \lieC{t}) \\
	&= \set{\varepsilon_0} \cup \set{\varepsilon_6 + \varepsilon_i: 2\leq i \leq 5} \cup \set{\varepsilon_6 - \varepsilon_1} \cup \set{\varepsilon_8 - \varepsilon_7} \\
	&\cup \set{\frac{1}{2}\sum_{i=1}^8 a_i \varepsilon_i: a_i = \pm 1, a_7=-1,a_6=a_8=1,a_1a_2\cdots a_5 = -1}.
\end{align*}
Hence
\begin{align*}
	\set{\begin{aligned}
		&\varepsilon_0, \varepsilon_6 + \varepsilon_5, \\
		&\frac{1}{2}(\varepsilon_1 + \varepsilon_2 - \varepsilon_3 - \varepsilon_4 - \varepsilon_5 + \varepsilon_6 - \varepsilon_7 + \varepsilon_8), \\
		&\frac{1}{2}(-\varepsilon_1 - \varepsilon_2 + \varepsilon_3 + \varepsilon_4 - \varepsilon_5 + \varepsilon_6 - \varepsilon_7 + \varepsilon_8)
	\end{aligned}}
\end{align*}
is a set of strongly orthogonal roots in $\Delta(\lie{p}_+ \cap (\lieC{h})^\perp, \lieC{t})$.
Hence we have $r = 4 = \rank_{\RR}(\lie{g}^{\theta\sigma})$.
The assumptions of Proposition \ref{prop:NonMinimalFullRank} hold.

\begin{align*}
	(\lie{g}, \lieNorm_{\lie{g}}(\lie{h})): \quad & \dynkin[scale=2, extended, labels={\alpha_1, , , \alpha_7, \alpha_6, \alpha_5, \alpha_4, \alpha_3, \alpha_2}, affine mark=*]{E}{**oooooo} \\
	(\lie{g}^{\theta\sigma}, \lieNorm_{\lie{g}}(\lie{h})^{\theta\sigma}): \quad & \dynkin[scale=2, labels={\alpha_7, , \alpha_6, \alpha_5, \alpha_4, \alpha_3, \alpha_2 + \alpha_1}]{E}{oooooo*}
	\quad \dynkin[scale=2, labels={}]{A}{*}
\end{align*}
In the third embedding, we have $(\lie{g}^{\theta\sigma}, \lieNorm_{\lie{g}}(\lie{h})^{\theta\sigma}) \simeq (\lieE_{7(-25)} \oplus \lieSL(2,\RR), \lieU(1, 6) \oplus \lieSO(2))$.
Then we have
\begin{align*}
	&\Delta(\lie{p}_+ \cap (\lieC{h})^\perp, \lieC{t}) \\
	&= \set{\varepsilon_0} \cup \set{\varepsilon_6 + \varepsilon_i: 1\leq i \leq 5} \cup \set{\varepsilon_8 - \varepsilon_7} \\
	&\cup \set{\frac{1}{2}\sum_{i=1}^8 a_i \varepsilon_i: a_i = \pm 1, a_7=-1,a_6=a_8=1,a_1a_2\cdots a_5 = -1}\\
	&\backslash \set{\frac{1}{2}\left(-\sum_{i=1}^5 \varepsilon_i + \varepsilon_6 - \varepsilon_7 + \varepsilon_8\right)}.
\end{align*}
Hence
\begin{align*}
	\set{\begin{aligned}
		&\varepsilon_0, \varepsilon_6 + \varepsilon_5,\\
		&\frac{1}{2}(\varepsilon_1 + \varepsilon_2 - \varepsilon_3 - \varepsilon_4 - \varepsilon_5 + \varepsilon_6 - \varepsilon_7 + \varepsilon_8), \\
		&\frac{1}{2}(-\varepsilon_1 - \varepsilon_2 + \varepsilon_3 + \varepsilon_4 - \varepsilon_5 + \varepsilon_6 - \varepsilon_7 + \varepsilon_8)
	\end{aligned}}
\end{align*}
is a set of strongly orthogonal roots in $\Delta(\lie{p}_+ \cap (\lieC{h})^\perp, \lieC{t})$.
Hence we have $r = 4 = \rank_{\RR}(\lie{g}^{\theta\sigma})$.
The assumptions of Proposition \ref{prop:NonMinimalFullRank} hold.

\subsubsection{\texorpdfstring{$\lie{g} = \lieF_{4(4)}$}{g=F4(4)}}

For $\lie{g} = \lieF_{4(4)}$, we check the assumptions of Proposition \ref{prop:ReductionRealStronglyOrthogonal}.

\begin{align*}
	(\lie{g}, \lieNorm_{\lie{g}}(\lie{h})): \quad & \dynkin[scale=2, extended, labels={\alpha_1, \alpha_2, , \beta_1, \beta_2}, affine mark=*]{F}{o*oo} \\
	(\lie{g}^{\theta\sigma}, \lieNorm_{\lie{g}}(\lie{h})^{\theta\sigma}): \quad & \dynkin[scale=2, labels={\beta_2, \beta_1}]{C}{oo*}
	\dynkin[scale=2, labels={\alpha_2 + \alpha_1}]{A}{*}
\end{align*}

In the embedding, we have $(\lie{g}^{\theta\sigma}, \lieNorm_{\lie{g}}(\lie{h})^{\theta\sigma}) \simeq (\lieSp(3, \RR) \oplus \lieSL(2,\RR), \lieU(3) \oplus \lieSL(2, \RR))$, which is a symmetric pair.
Let $\lie{g}^{\theta\sigma} = \lie{i} \oplus \lie{j}$ be the decomposition into simple ideals such that $\lie{j} \simeq \lieSp(3, \RR)$.

\begin{lemma} \label{lem:ProjectionToGPrime}
	The projection of $\lie{g}_0\cap \lie{k}$ to $\lie{j}$ coincides with $\lie{j}\cap \lie{k}$.
\end{lemma}

\begin{proof}
	Let $Z\in \lie{g}_2$ be the element such that $\ad(Z+\theta(Z))$ defines a complex structure on $\lie{g}_1 \oplus \lie{g}_{-1}$ as in Lemma \ref{lem:ComplexStructure}.
	Then $\ad(Z+\theta(Z))$ gives a complex structure on $\lie{g}^{-\sigma, -\theta}$.
	Hence $Z+\theta(Z)$ has non-zero projections to both factors $\lie{i}$ and $\lie{j}$.
	Since $(\lie{g}_0\cap \lie{k}) \oplus \RR (Z+\theta(Z)) = \lie{k}^{\sigma}$ is an orthogonal direct sum, the projection of $\lie{g}_0\cap \lie{k}$ to $\lie{j}$ coincides with $\lie{j}\cap \lie{k}$.
\end{proof}

By the isomorphism $(\lie{g}^{\theta\sigma}, \lieNorm_{\lie{g}}(\lie{h})^{\theta\sigma}) \simeq (\lieSp(3, \RR) \oplus \lieSL(2,\RR), \lieU(3) \oplus \lieSL(2, \RR))$ and Lemma \ref{lem:ProjectionToGPrime}, the assumptions of Proposition \ref{prop:ReductionRealStronglyOrthogonal} hold.

\subsubsection{\texorpdfstring{$\lie{g} = \lieG_{2(2)}$}{g=G2(2)}}


For $\lie{g} = \lieG_{2(2)}$, we check the assumptions of Proposition \ref{prop:ReductionRealStronglyOrthogonal}.

\begin{align*}
	(\lie{g}, \lieNorm_{\lie{g}}(\lie{h})): \quad & \dynkin[scale=2, extended, labels={\alpha_1, \alpha_2, }, affine mark=*]{G}{o*} \\
	(\lie{g}^{\theta\sigma}, \lieNorm_{\lie{g}}(\lie{h})^{\theta\sigma}): \quad & \dynkin[scale=2, labels={\alpha_2+\alpha_1}]{A}{*}
	\dynkin[scale=2, labels={}]{A}{*}
\end{align*}

In the embedding, we have $(\lie{g}^{\theta\sigma}, \lieNorm_{\lie{g}}(\lie{h})^{\theta\sigma}) \simeq (\lieSL(2, \RR) \oplus \lieSL(2,\RR), \lieSL(2, \RR) \oplus \lieSO(2))$, which is a symmetric pair.
Let $\lie{g}^{\theta\sigma} = \lie{i} \oplus \lie{j}$ be the decomposition into simple ideals such that $\lie{i}\subset \lie{h}$.
By the same argument as in the proof of Lemma \ref{lem:ProjectionToGPrime}, the projection of $\lie{g}_0\cap \lie{k}$ to $\lie{j}$ coincides with $\lie{j}\cap \lie{k}$.
Therefore, by the isomorphism $(\lie{g}^{\theta\sigma}, \lieNorm_{\lie{g}}(\lie{h})^{\theta\sigma}) \simeq (\lieSL(2, \RR) \oplus \lieSL(2,\RR), \lieSL(2, \RR) \oplus \lieSO(2))$, the assumptions of Proposition \ref{prop:ReductionRealStronglyOrthogonal} hold.

We have completed the proof of Theorem \ref{thm:NonMinimalExceptional}, and therefore the proof of Theorem \ref{thm:MainTargetMinimal} in all cases.

\section{Admissibility} \label{section:Admissibility}

In this section, we discuss the relation between admissibility and discrete decomposability in the branching problem.
We show in Theorem \ref{thm:DiscretelyDecomposableAndAdmissibility} that the two notions are equivalent modulo compact factors.
We then use the insight obtained from Theorem \ref{thm:DiscretelyDecomposableAndAdmissibility} to give an alternative proof of Fact \ref{fact:Implication} (1) $\Rightarrow$ (2).

\subsection{Criteria for admissibility}

Let $G$ be a connected reductive Lie group with a Cartan involution $\theta$ and $H$ a $\theta$-stable closed connected reductive subgroup of $G$.
Set $K\coloneq G^\theta$ and $K_H \coloneq K \cap H$.

We shall recall results on admissibility in the branching problem.
The `only if' part of the following fact was proved by T.\ Kobayashi in \cite[Proposition 1.6]{Ko98_discrete_decomposable_3}.
The `if' part was proved by Zhu--Liang in \cite{ZhLi10} for unitarizable irreducible modules, and by the author in \cite{Ki24} for general irreducible modules.

\begin{fact} \label{fact:AdmissibleKH}
	Let $V$ be an irreducible $(\lie{g}, K)$-module.
	Then $V$ is $K_H$-admissible if and only if $V$ is $\lie{h}$-admissible.
\end{fact}

The following theorem was proved in \cite[Theorem 20]{Ki24}.
For compact subgroups, the theorem was proved by T.\ Kobayashi in \cite[Theorem 2.9]{Ko98_discrete_decomposable_2} and \cite[Theorem 1.1]{Ko19_admissible}.

\begin{theorem} \label{thm:AdmissibilityCriterion}
	Let $V$ be an irreducible $(\lie{g}, K)$-module.
	Then $V$ is $\lie{h}$-admissible if and only if $\lie{h}^\perp \cap \WF(V) = \set{0}$.
\end{theorem}

Theorem \ref{thm:AdmissibilityCriterion} asserts that admissibility is equivalent to a purely Lie-algebraic condition on nilpotent orbits.
We shall show a nilpotent-orbit counterpart of Fact \ref{fact:AdmissibleKH}.

\begin{proposition} \label{prop:AdmissibilityCriterionWavefront}
	Let $\orbit{O}$ be a nilpotent $G$-orbit in $\lie{g}$.
	Then $\overline{\orbit{O}} \cap \lie{h}^\perp = \set{0}$ if and only if $\overline{\orbit{O}} \cap \lie{k}_H^\perp = \set{0}$.
\end{proposition}

\begin{proof}
	The `if' part is obvious from $\lie{k}_H \subset \lie{h}$.
	To show the converse, we assume that $\overline{\orbit{O}} \cap \lie{k}_H^\perp \neq \set{0}$.
	Note that $\lie{k}_H^\perp = \lie{h}^{-\theta} \oplus \lie{h}^\perp$.
	Take a non-zero $X\coloneq X_1 + X_2 \in \overline{\orbit{O}} \cap \lie{k}_H^\perp$ with $X_1 \in \lie{h}^{-\theta}$ and $X_2 \in \lie{h}^\perp$.
	
	Observe that $\ad(X_1)$ is semisimple and has real eigenvalues.
	Since $\overline{\orbit{O}}$ and $\lie{h}^\perp$ are $\exp(\RR \ad(X_1))$-stable, the same argument as in the proof of Proposition \ref{prop:NullCone} allows us to assume that $X$ is an eigenvector of $\ad(X_1)$, that is, $X$ has one of the following forms:
	\begin{enumparen}
		\item $X = X_1$,
		\item $X = X_2$,
		\item $X_1, X_2 \neq 0$ and $[X_1, X_2] = 0$.
	\end{enumparen}
	If $X = X_1$, then $X$ is semisimple, which contradicts the assumption that $X$ is non-zero and nilpotent.
	If $X = X_2$, then $0 \neq X \in \overline{\orbit{O}} \cap \lie{h}^\perp$ and we have nothing to prove.

	Assume that $X_1, X_2 \neq 0$ and $[X_1, X_2] = 0$.
	Since $X_1 \in \lie{h}$ and $X_2 \in \lie{h}^\perp$, we have $(X_1, X_2) = 0$.
	Since $X$ is a nilpotent element commuting with the semisimple element $X_1$, they are orthogonal.
	In fact, $X$ belongs to the derived subalgebra of $\lieCent_{\lie{g}}(X_1)$.
	Hence we have
	\begin{align*}
		0 = (X_1, X_2) = (X_1, X - X_1) = -(X_1, X_1).
	\end{align*}
	Since $X_1 \in \lie{h}^{-\theta}$, we obtain $X_1 = 0$, which is a contradiction.
	We have shown the proposition.
\end{proof}

Let $\lie{h}_\nc$ be the ideal of $\lie{h}$ generated by $\lie{h}^{-\theta}$ and $H_\nc$ the analytic subgroup of $H$ with the Lie algebra $\lie{h}_\nc$.
In other words, $\lie{h}_{\nc}$ is the non-compact part of $\lie{h}$.
Then $H_\nc$ is a closed normal subgroup of $H$ and $H/H_\nc$ is compact.
Set $H' \coloneq N_{G}(H_{\nc})_o$ and $K_{H'} \coloneq K \cap H'$.
Then we have
\begin{align*}
	\lie{h}' &= \lieNorm_{\lie{g}}(\lie{h}_\nc) = \lie{h}_\nc + \lieCent_{\lie{g}}(\lie{h}_\nc) = \lie{h}_\nc + \lie{g}^{H_{\nc}}, \\
	\lie{h}_{\nc}^\perp &= (\lie{h'})^\perp \oplus (\lie{h}_\nc^\perp \cap (\lie{h}')^{H_{\nc}}), \\
	\lie{h}_{\nc}^\perp &= \lie{h}^\perp \oplus (\lie{h}_\nc^\perp \cap \lie{h}^{H_{\nc}}).
\end{align*}

\begin{lemma} \label{lem:NilpotentAndNonCompact}
	One has
	\begin{align*}
		\Nilpotent(\lie{h}^\perp, H) = \Nilpotent((\lie{h}_\nc)^\perp, H_\nc) = \Nilpotent((\lie{h}')^\perp, H_\nc).
	\end{align*}
\end{lemma}

\begin{proof}
	Since $H/H_\nc$ is compact, we have $\Nilpotent(\lie{g}, H) = \Nilpotent(\lie{g}, H_\nc)$.
	Hence we may replace $H$ with $H_\nc$ in the assertion.
	
	Set $W\coloneq \lie{h}_\nc^\perp \cap \lie{h}^{H_{\nc}}$.
	Let $p\colon \lie{g} \rightarrow W$ be the orthogonal projection.
	Then $p$ is $H_{\nc}$-equivariant, and hence we have
	\begin{align*}
		p(\Nilpotent(\lie{h}_\nc^\perp, H_{\nc})) \subset \Nilpotent(W, H_{\nc}) = \set{0}.
	\end{align*}
	This shows that $\Nilpotent(\lie{h}_\nc^\perp, H_{\nc}) \subset \lie{h}^\perp$, which implies the first equality.
	The second equality follows from the same argument as above.
\end{proof}

Set $\lie{h''} \coloneq \lie{h} + \lie{k}_{H'}$ and $H'' \coloneq H K_{H'}$.
Then $K_{H'}$ is a maximal compact subgroup of $H''$ and $H''/H_\nc$ is compact.

\begin{lemma} \label{lem:NilpotentAndInvariant}
	One has $((\lie{h''})^\perp)^{H_\nc} \cap \Nilpotent(\lie{g}, G) = \set{0}$.
\end{lemma}

\begin{proof}
	Since $\lie{h'} = \lie{h''} + (\lie{h'})^{-\theta}$, we have
	\begin{align*}
		((\lie{h''})^\perp)^{H_\nc} \subset (\lie{h''})^\perp \cap \lie{h'} \subset (\lie{h'})^{-\theta}.
	\end{align*}
	Hence any element of $((\lie{h''})^\perp)^{H_\nc}$ is semisimple and is therefore not nilpotent unless it is zero.
	This proves the assertion.
\end{proof}

\begin{corollary} \label{cor:NilpotentAndPerp}
	Let $\orbit{O}$ be a nilpotent $G$-orbit in $\lie{g}$.
	Then $\overline{\orbit{O}} \cap \Nilpotent((\lie{h''})^\perp, H'') = \set{0}$ if and only if $\overline{\orbit{O}} \cap (\lie{h''})^\perp = \set{0}$.
\end{corollary}

\begin{proof}
	The assertion follows from Proposition \ref{prop:IntersectionW} and Lemma \ref{lem:NilpotentAndInvariant}.
\end{proof}

The following lemma was proved in \cite[Proposition 22]{Ki24} in a more general setting.

\begin{lemma} \label{lem:DiscreteDecomposabilityAndAdmissibilityWavefront}
	Let $\orbit{O}$ be a nilpotent $G$-orbit in $\lie{g}$.
	Then the following conditions are equivalent.
	\begin{enumparen}
		\item $\Nilpotent(\lie{h}^\perp, H) \cap \overline{\orbit{O}} = \set{0}$.
		\item $\Nilpotent(\lie{h}_\nc^\perp, H_\nc) \cap \overline{\orbit{O}} = \set{0}$.
		\item $(\lie{h}'')^\perp \cap \overline{\orbit{O}} = \set{0}$.
		\item $\lie{k}_{H'}^\perp \cap \overline{\orbit{O}} = \set{0}$.
		\item $(\lie{h}')^\perp \cap \overline{\orbit{O}} = \set{0}$.
	\end{enumparen}
\end{lemma}

\begin{proof}
	(1) $\Leftrightarrow$ (2) follows from Lemma \ref{lem:NilpotentAndNonCompact}.
	The equivalence of (3), (4) and (5) has been proved in Proposition \ref{prop:AdmissibilityCriterionWavefront}.
	Applying Lemma \ref{lem:NilpotentAndNonCompact} to $H''$, we have
	\begin{align*}
		\Nilpotent(\lie{h}_\nc^\perp, H_\nc) = \Nilpotent((\lie{h''})^\perp, H'').
	\end{align*}
	By Corollary \ref{cor:NilpotentAndPerp}, we have
	\begin{align*}
		\overline{\orbit{O}} \cap \Nilpotent(\lie{h}_\nc^\perp, H_\nc) = \set{0}
		&\Leftrightarrow \overline{\orbit{O}} \cap \Nilpotent((\lie{h''})^\perp, H'') = \set{0} \\
		&\Leftrightarrow \overline{\orbit{O}} \cap (\lie{h''})^\perp = \set{0}.
	\end{align*}
	This shows that (2) $\Leftrightarrow$ (3).
	We have shown the lemma.
\end{proof}

\begin{theorem} \label{thm:DiscretelyDecomposableAndAdmissibility}
	Let $V$ be an irreducible $(\lie{g}, K)$-module.
	Then the following conditions are equivalent.
	\begin{enumparen}
		\item $V|_{\lie{h}, K_H}$ is discretely decomposable.
		\item $V$ is $\lie{h}''$-admissible.
		\item $V$ is $K_{H'}$-admissible.
		\item $V$ is $\lie{h}'$-admissible.
	\end{enumparen}
\end{theorem}

\begin{proof}
	The assertion follows from Theorems \ref{thm:MainTheorem} and \ref{thm:AdmissibilityCriterion} and Lemma \ref{lem:DiscreteDecomposabilityAndAdmissibilityWavefront}.
\end{proof}

Theorem \ref{thm:DiscretelyDecomposableAndAdmissibility} asserts that there is no essential difference between admissibility and discrete decomposability.
For example, the following proposition gives a situation in which these two conditions are equivalent.

\begin{proposition} \label{prop:AdmissibilityAndDiscreteDecomposability}
	Let $V$ be an irreducible $(\lie{g}, K)$-module.
	Assume that $(\lie{h}^\perp)^{H_\nc} \cap \WF(V) = \set{0}$.
	Then $V$ is $\lie{h}$-admissible if and only if $V|_{\lie{h}, K_H}$ is discretely decomposable.
\end{proposition}

\begin{remark}
	Clearly, if $(\lie{h}^\perp)^{H_\nc} \cap \Nilpotent(\lie{g}, G) = \set{0}$, then the assumption holds.
	For example, $\lie{h} = \lie{h}''$ implies $(\lie{h}^\perp)^{H_\nc} \cap \Nilpotent(\lie{g}, G) = \set{0}$ as we have seen in Lemma \ref{lem:NilpotentAndInvariant}.
	We have given a necessary and sufficient condition for $(\lie{h}^\perp)^{H_\nc} \cap \Nilpotent(\lie{g}, G) = \set{0}$ in \cite[Proposition 22]{Ki24}.
\end{remark}

\begin{proof}
	By assumption and Proposition \ref{prop:IntersectionW}, we have
	\begin{align*}
		\WF(V) \cap \lie{h}^\perp = \set{0} \Leftrightarrow \WF(V) \cap \Nilpotent(\lie{h}^\perp, H) = \set{0}.
	\end{align*}
	Hence the assertion follows from Theorems \ref{thm:MainTheorem} and \ref{thm:AdmissibilityCriterion}.
\end{proof}

As in the proof, the conclusion of Proposition \ref{prop:AdmissibilityAndDiscreteDecomposability} holds under the condition
\begin{align*}
	\WF(V) \cap \lie{h}^\perp = \set{0} \Leftrightarrow \WF(V) \cap \Nilpotent(\lie{h}^\perp, H) = \set{0}.
\end{align*}
We have given a sufficient condition for this equivalence in \cite[Theorem 23]{Ki24}.

\subsection{Alternative proof of Fact \ref{fact:Implication} \texorpdfstring{(1) $\Rightarrow$ (2)}{(1) => (2)}}

In \cite{Ki24}, we proved the implication (1) $\Rightarrow$ (2) in Fact \ref{fact:Implication} by using several tools, e.g., heat kernels, wave front sets and translation functors.
Motivated by the insight obtained from Theorem \ref{thm:DiscretelyDecomposableAndAdmissibility}, we give an alternative proof of the implication (1) $\Rightarrow$ (2) in Fact \ref{fact:Implication}.
Retain the notation from the previous subsection.

We use the following fact, which was proved by T.\ Kobayashi in \cite[Theorem 2.9]{Ko98_discrete_decomposable_2}.
The original form of the fact is stated in terms of asymptotic $K$-support.
We shall state the fact in terms of the wave front set.
See \cite[Corollary 9]{Ki24} for the equivalence of the two conditions.

\begin{fact} \label{fact:AdmissibilityCriterionCompact}
	Let $V$ be an irreducible $(\lie{g}, K)$-module.
	If $\WF(V) \cap \lie{k}_{H}^\perp = \set{0}$, then $V$ is $K_{H}$-admissible and, in particular, $\lie{h}$-admissible.
\end{fact}

Although Fact \ref{fact:AdmissibilityCriterionCompact} is a special case of Theorem \ref{thm:AdmissibilityCriterion}, we have stated it separately to make explicit the results needed to prove the following theorem.

\begin{theorem}
	Let $V$ be an irreducible $(\lie{g}, K)$-module.
	If $\WF(V) \cap \Nilpotent(\lie{h}^\perp, H) = \set{0}$, then $V|_{\lie{h}, K_H}$ is discretely decomposable.
\end{theorem}

\begin{proof}
	Assume that $\WF(V) \cap \Nilpotent(\lie{h}^\perp, H) = \set{0}$.
	By Lemma \ref{lem:DiscreteDecomposabilityAndAdmissibilityWavefront}, we have $\WF(V) \cap (\lie{k}_{H'})^\perp = \set{0}$.
	Hence $V$ is $\lie{h}''$-admissible by Fact \ref{fact:AdmissibilityCriterionCompact}.
	Take an $(\lie{h}'', K_{H'})$-module filtration $V = \bigcup_{i=0}^\infty V_i$ such that each $V_i$ has finite length.
	From the definition of $\lie{h}''$ and $\lie{h}_\nc$, there exists a compact ideal $\lie{k}''$ of $\lie{h}''$ such that
	\begin{align*}
		\lie{h}_\nc \subset \lie{h} \subset \lie{h}_{\nc} \oplus \lie{k}'' = \lie{h}''.
	\end{align*}
	This implies that $V_i$ has finite length as an $(\lie{h}, K_{H})$-module for each $i$.
	Hence $V|_{\lie{h}, K_H}$ is discretely decomposable.
\end{proof}

\bibliographystyle{abbrv}

\def\cprime{$'$} \def\cprime{$'$}

\end{document}